\documentclass[11pt, reqno]{amsart}

\usepackage{amsmath,amssymb,amsxtra}
\usepackage{mathtools}
\usepackage{cases}
\usepackage{geometry}
\usepackage{enumitem}
\usepackage{xcolor}
\usepackage{graphicx}
\usepackage[colorlinks=true]{hyperref}
\hypersetup{urlcolor=blue, citecolor=red, linkcolor=blue}

\newcommand{\epsl}{\varepsilon}
\newcommand{\fai}{\varphi}
\newcommand\wan[1]{\widetilde{#1}}
\newcommand\mao[1]{\widehat{#1}}

\newcommand{\cl}[1]{\overline{#1}}
\newcommand{\ul}[1]{\underline{#1}}

\newcommand{\ppd}[1]{\partial_{#1}}

\newcommand{\lle}{\lesssim}
\newcommand{\gge}{\gtrsim}

\newcommand{\me}{\mathrm{e}}
\newcommand{\mi}{\mathrm{i}}
\newcommand{\medi}[1]{#1 \kern-1.25ex{-}}

\newcommand{\bN}{\mathbb{N}}
\newcommand{\bZ}{\mathbb{Z}}

\newcommand{\bR}{\mathbb{R}}

\newcommand{\bS}{\mathbb{S}}

\newcommand{\hti}[1]{\mathcal{#1}}
\newcommand{\gti}[1]{\mathfrak{#1}}

\newcommand\abs[1]{\lvert #1 \rvert}
\newcommand\norm[1]{\lVert #1 \rVert}
\newcommand\spa[1]{\langle #1 \rangle}
\newcommand{\flo}[1]{\lfloor #1 \rfloor}
\newcommand\set[1]{\{ #1 \}}

\newcommand\Abs[1]{\left\lvert #1 \right\rvert}
\newcommand\Norm[1]{\left\lVert #1 \right\rVert}

\newcommand\Brac[1]{\left( #1 \right)}
\newcommand\SBrac[1]{\left[ #1 \right]}
\newcommand\thmref[1]{Theorem~\ref{#1}}

\newcommand\lemmaref[1]{Lemma~\ref{#1}}
\newcommand\cororef[1]{Corollary~\ref{#1}}
\newcommand{\hhat}[1]{\widehat{#1}}
\newcommand{\BB}{\mathcal{B}}
\newcommand{\TT}{\mathcal{T}}
\newcommand{\QQ}{\mathcal{Q}}
\newcommand{\QU}{\mathcal{QU}}

\newcommand{\FF}{\mathcal{F}}
\newcommand{\II}{\mathcal{I}}
\newcommand{\bbT}{\mathbf{T}}
\newcommand{\bbI}{\mathbf{I}}
\newcommand{\III}{\mathfrak{I}}
\newcommand{\JJJ}{\mathfrak{J}}
\newcommand{\TTT}{\mathfrak{T}}

\newcommand{\va}{V^{\iota_1,p_1}_{k_1}}
\newcommand{\vb}{V^{\iota_2,p_2}_{k_2}}
\newcommand{\vvc}{V^{\iota_3,p_3}_{k_3}}

\newcommand{\Vb}{V^{\iota_2,p_2}_{k_2,j_2}}

\newcommand{\Vaa}{V^{\iota_1,p_1}_{k_1,j_1,n_1}}
\newcommand{\Vbb}{V^{\iota_2,p_2}_{k_2,j_2,n_2}}
\newcommand{\Vcc}{V^{\iota_3,p_3}_{k_3,j_3,n_3}}

\newcommand{\Vff}{V^{\iota_6,p_6}_{k_6,j_6,n_6}}

\newcommand{\vx}[1]{V^{\iota_{#1},p_{#1}}_{k_{#1}}}
\newcommand{\Vxx}[1]{V^{\iota_{#1},p_{#1}}_{k_{#1},j_{#1},n_{#1}}}

\DeclareMathOperator{\re}{Re}

\DeclareMathOperator{\supp}{supp}
\DeclareMathOperator{\dif}{d\!}

\DeclareMathOperator{\dive}{div}

\allowdisplaybreaks
\newtheorem{theorem}{Theorem}[section]
\newtheorem{lemma}{Lemma}[section]
\newtheorem{proposition}{Proposition}[section]
\newtheorem{remark}{Remark}[section]

\newtheorem{corollary}{Corollary}[section]
\newtheorem*{main-theorem}{Main Theorem}
\newtheorem{mainthm}{Theorem}
\newtheorem*{remarkk}{Remark}

\numberwithin{equation}{section}

\author{Han Cui}
\address[Han Cui]{Yau Mathematical Sciences Center, Tsinghua University}
\email{cuih@mail.tsinghua.edu.cn}

\author{Linhao Shi}
\address[Linhao Shi]{Institute of Mathematical Sciences,
The Chinese University of Hong Kong, Hong Kong}
\email{linhao.shi99@gmail.com}

\begin{document}

\title[Global Existence and Scattering of a Class of Two-Dimensional EPK Systems]{Global Existence and Scattering of a Class of Two-Dimensional Euler--Poisson--Korteweg Systems}

\begin{abstract}
For a class of two-dimensional Euler--Poisson--Korteweg systems of quantum hydrodynamic type, we establish global existence and scattering for small, smooth, spatially localized perturbations of a constant equilibrium with irrotational velocity. 
This extends the three-dimensional well-posedness result of Song [Math. Ann. 392 (2025), 2157--2224] to two dimensions. Our proof follows the space-time resonance approach developed by Deng, Ionescu, Pausader, and Pusateri for the three-dimensional gravity-capillary water-wave system. The new ingredients are twofold. First, to allow non-neutral initial data, we exploit the absence of space resonances and of time resonances---including those arising after the second and third iterations---and the improved approximate-finite-speed-of-propagation bounds. Second, in the energy estimates, we adapt the almost orthogonality argument used in the works of Ionescu--Pausader and Guo--Ionescu--Pausader to the regime in which the input and output frequencies are all comparable to $1$. In this regime, time resonances may occur and the relevant curvature quantity may vanish.
\end{abstract}

\maketitle

\section{Introduction}
\subsection{Setting and Reformulation}
The Euler--Poisson--Korteweg (EPK) system
\begin{align*}
	\rho_t+\dive(\rho u)&=0,\\
	(\rho u)_t+\dive(\rho u \otimes u)&=-\nabla p+\rho \nabla \phi+\rho \nabla K(\rho),\\
	\Delta \phi&=\rho-1,
\end{align*}
with
\[
K(\rho)=m(\rho) \Delta \rho+\frac 12 m'(\rho) \abs{\nabla \rho}^2,
\]
describes the dynamics of a liquid--vapor mixture in a diffuse-interface model under electric forces. Here $K(\rho)$ arises from the Korteweg tensor. The particular case $m(\rho)=\kappa/\rho$ corresponds to the so-called quantum hydrodynamics (QHD) system (which is related to the nonlinear Schr\"odinger--Poisson equation via the Madelung transform). For further background and references, we refer the reader to \cite{MR2481754,MR2039696} and to the discussion of previous results in Section~\ref{sec:preresu}.

In this paper, we study the long-time existence and behavior of solutions to the Cauchy problem for the two-dimensional Euler--Poisson--Korteweg system, thereby extending the well-posedness result of \cite{zbMATH08045545} to two dimensions. For simplicity, we take $p(\rho)=\frac{1}{2}\rho^2$ (indeed, our argument applies whenever $p'(1)>0$). We assume that the initial velocity field is irrotational, a condition that is preserved by the flow. Let $\psi$ be the potential such that $u=\nabla \psi$. We consider perturbations around the constant state $(\rho,\psi,\phi)=(1,0,0)$. In this setting, the EPK system becomes
\begin{subequations}\label{meq0}
	\begin{align}
	\rho_t+\dive(\rho \nabla \psi)&=0,\\
	\psi_t+\frac 12 \abs{\nabla \psi}^2&=-(\rho-1)+\phi+K(\rho),\\
	\Delta \phi&=\rho-1.
	\end{align}	
\end{subequations}
We assume that $m \in C^\infty((0,\infty))$ and $m(1)>0$, and set $\wan \rho=\rho-1$. To eliminate the explicit singularities in the Fourier multipliers, we use the variables
\[
n=\Lambda^{-1} \wan \rho, \quad \wan\psi=\Lambda \psi
\]
($\Lambda=\abs{D}$\footnote{We use the convention $\mao{m(D)f}(\xi)=m(\xi)\mao f(\xi)$.}). Henceforth, for notational simplicity, we write $\rho, \psi, m$ in place of $n, \wan \psi, m(\cdot+1)$, respectively. With this convention, the system for $(\rho,\psi)$ reads
 \begin{subequations}\label{meq}
  \begin{gather}
   	\rho_t+R [(\Lambda \rho+1)R\psi]=0, \label{eqrho}\\
 	\psi_t+\frac 12 \Lambda \abs{R\psi}^2+(1-\Delta)\rho-\Lambda [m(\Lambda \rho) \Delta \Lambda \rho+\frac 12 m'(\Lambda \rho) \abs{\nabla \Lambda \rho}^2]=0, \label{eqpsi}
 \end{gather}
 where \(R=\nabla/\Lambda\) denotes the vector of Riesz transforms.
 \end{subequations}
The initial data are given by
\begin{equation}\label{initial}
	\rho(x,0)=\rho_0(x), \quad \psi(x,0)=\psi_0(x).
\end{equation}

We assume that $\rho, \psi \colon \bR^2 \times I \to \bR$. We establish global existence and scattering for the Euler--Poisson--Korteweg system in the QHD-type case $m(0)+m'(0)=0$. For general Korteweg tensors, we also prove long-time existence. Our analysis is based on the space-time resonance method, which has been successfully applied to many significant models; see, for example, the works \cite{MR3024265,MR3283401,MR3450481,zbMATH06702025,MR3665671,MR3784694,MR3810467,MR3914181,MR4157108,MR4526823,zbMATH08045545,zbMATH08053129}, which are particularly relevant to this paper. The present work follows the approach of \cite{MR3784694}. Compared with \cite{MR3784694}, this work overcomes the slow decay of the low-frequency components of the solution, thereby allowing for non-neutral initial data, and provides an alternative treatment of the small modulation cases in the energy estimates.

\subsection{Main Results}
We now state the main results. Set $\gamma_0=2^{-1/2}m(0)^{-1/4}$. We first introduce the $Z$-norm to quantify the spatial localization:
\[
\norm{f}_{Z_\delta}=\sup_{k \in \bZ, j \ge k_-,0 \le n \le j+1}2^{10k_+}2^{\delta k}2^{(1-20\delta)(j+k)}2^{-(1/2-18\delta)n}\norm{A^{(j)}_nQ^{(k)}_jP_kf}_2,
\]
where $0 <\delta \le 10^{-5}$ and $k_{\pm}$ denotes $\max \set{\pm k,0}$. ($A^{(j)}_nQ^{(k)}_jP_kf=\wan \fai^{-j-1,0}_{-n}(2^{100}(\abs{D}-2\gamma_0))\fai_j^{k_-}(x)\psi_k(D)f$; see Section~\ref{notations}.) $\Omega$ denotes the rotational vector field $x_1 \ppd{x_2}-x_2 \ppd{x_1}$ and $\Omega^{[0,M]}f$ denotes $\Omega^af$ with $a \in \set{0,1,\cdots,M}$.
Let
\[
U=\rho-\mi q(D) \psi,
\]
where
\[
q(\xi)=\frac 1 {\sqrt{m(0)\abs{\xi}^4+\abs{\xi}^2+1}},
\]
and
\[
V=\me^{-\mi t\Lambda(D)} U,
\]
where
\[
\Lambda(\xi)=\sqrt{m(0)\abs \xi^4+\abs \xi^2+1}
\]
is the dispersion relation. Set $N_{0,\delta}=\flo{10^5/\delta^4}$, $N_{1,\delta}=\flo{10^3/\delta^3}$, $N_{2,\delta}=\flo{600/\delta^3}$, and $N_{3,\delta}=N_{2,\delta}-\flo{100/\delta}$.

We first state our main result on global existence and scattering in the QHD-type case. In the following theorem, we fix $\delta=10^{-5}$ and suppress the dependence on $\delta$ in the notation.
\begin{mainthm}\label{qhdthm}
		Suppose that $m(0)+m'(0)=0$. Assume that
	\[
	\norm{(\Lambda^{\delta}+\Lambda^{N_0})U_0}_2+\norm{\Lambda^{\delta}\Omega^{N_1}U_0}_2+\norm{\Omega^{[0,N_2]}U_0}_Z=\epsl_0 \le \cl{\epsl},
	\footnote{The parameters $N_{0}$, $N_{1}$, and $N_{2}$ are not optimized.}
	\]
	where
	\[
	U_0=\rho_0-\mi q(D) \psi_0,
	\]
	for some sufficiently small constant $\cl{\epsl}$. Then the Cauchy problem \eqref{meq}-\eqref{initial} admits a unique global solution $(\rho,\psi)$ satisfying the following uniform bound
	\[
	 \spa{t}^{-\delta^4}\Brac{\norm{(\Lambda^{\delta}+\Lambda^{N_0})U(t)}_2+\norm{\Lambda^{\delta}\Omega^{N_1}U(t)}_2}+\norm{\Omega^{[0,N_2]}V(t)}_Z \lle \epsl_0,
	\]
	for $t \in [0,\infty)$. Moreover, the solution scatters in the $Z$-norm: there exists $V_\infty$ such that
	\[
	\lim_{t \to \infty}\norm{\Omega^{[0,N_2]}(\me^{-\mi t \Lambda(D)}U(t)-V_\infty)}_Z=0.
	\]
\end{mainthm}

For general Korteweg tensors, we also obtain a long-time existence result.\begin{mainthm}\label{generalthm}
		Assume that
	\[
	\norm{(\Lambda^{\delta}+\Lambda^{N_{0,\delta}})U_0}_2+\norm{\Lambda^{\delta}\Omega^{N_{1,\delta}}U_0}_2+\norm{\Omega^{[0,N_{2,\delta}]}U_0}_{Z_\delta}=\epsl_0 \le \cl{\epsl}_\delta,
	\]
	where
	\[
	U_0=\rho_0-\mi q(D) \psi_0,
	\]
	for some sufficiently small constant $\cl{\epsl}_\delta$. Then the Cauchy problem \eqref{meq}-\eqref{initial} admits a unique solution $(\rho,\psi)$ on the time interval $[0,\epsl_0^{-(100\delta)^{-1}}]$, and the solution satisfies the uniform bound
	\[
	\norm{(\Lambda^{\delta}+\Lambda^{N_{0,\delta}})U(t)}_2+\norm{\Lambda^{\delta}\Omega^{N_{1,\delta}}U(t)}_2+\norm{\Omega^{[0,N_{2,\delta}]}V(t)}_{Z_\delta}\lle_\delta \epsl_0.
	\]
\end{mainthm}
\begin{remarkk}
	\begin{enumerate}
		\item The conditions imposed on the initial data allow for non-neutral perturbations, due to the presence of the $\delta$-fractional derivatives in the norms and the derivatives in the $Z$-norm. However, treating such general data leads to the issue of low frequency. The low-frequency components of the solution have a slow decay rate in the absence of a null structure, and they still exhibit non-optimal decay even when the null structure is present; see Section~\ref{sec:newingre} for details.
		\item The reason we can obtain global results only in the QHD-type case lies in the energy estimates. When we exploit the absence of time resonances (or perform the normal form transformation), we will encounter loss of derivatives, which can be handled by the symmetrization argument. However, the dispersive operator $\Lambda(D)$ is of second order and the symmetrization argument can recover at most one derivative. We therefore need to exploit the depletion factors to close the estimates. For a general Korteweg tensor, some quadratic nonlinearities may not contain the depletion factors, whereas in the QHD-type case this obstruction is absent.
	\end{enumerate}
\end{remarkk}

Since the proof of \thmref{qhdthm} involves all the main difficulties, we only present this proof in detail. The proof of \thmref{generalthm} is simpler and follows from a combination of relatively straightforward energy estimates and the dispersive estimates. (See Section~\ref{gekotensors} for an explanation of the energy estimates.) Henceforth we assume $m(0)=1$. We fix $\delta=10^{-5}$ except in Section~\ref{gekotensors}.

\subsection{Previous Results}\label{sec:preresu}
The Euler--Poisson--Korteweg system is closely related to both the Euler--Poisson and Euler--Korteweg systems.  We first recall the results for the Euler--Poisson system, obtained by removing the Korteweg term.  For the repulsive electron model in three dimensions, Guo \cite{MR1637856} proved the global existence of smooth solutions for small irrotational perturbations of a constant equilibrium.  In two dimensions, global stability of the constant equilibrium for small irrotational data was established independently by Ionescu and Pausader \cite{MR3024265} and Li and Wu \cite{zbMATH06377403}. In one dimension, Guo, Han, and Zhang \cite{MR3595365} showed that small-amplitude smooth solutions persist globally without forming shocks.  For the ion model, Guo and Pausader \cite{MR2775116} established global smooth dynamics in three dimensions, while the two-dimensional small-data result was recently proved in \cite{arXiv:2510.18655}.

When the electric potential is suppressed, one recovers the Euler--Korteweg system.  Benzoni-Gavage, Danchin, and Descombes developed the well-posedness theory first in one dimension and then in several space dimensions \cite{MR2226932,MR2354691}.  For small perturbations of suitable global reference solutions, their estimates also give an almost global lower bound on the lifespan.  Audiard and Haspot \cite{zbMATH06702025} established global well-posedness for small irrotational data in dimensions $d \ge 3$.

When the Korteweg tensor is induced by the Bohm potential, we refer to the resulting class of models as quantum hydrodynamic (QHD) systems. For the QHD form of the EPK system in one dimension, local and global existence results, including exponential convergence to steady states under suitable assumptions, were established in \cite{MR2016508,MR2086047,MR2206278,MR2228690}. On a multidimensional torus, Li and Marcati \cite{MR2039696} proved local existence for large regular irrotational data with positive density and global existence with exponential convergence for small perturbations of a stationary state satisfying a subsonic stability condition.

We next discuss weak solutions of the QHD systems. For the collisional QHD form of the EPK system, Antonelli and Marcati \cite{MR2481754} used the connection with the Schr\"odinger equation to construct global finite-energy weak solutions in three dimensions for arbitrarily large energy data generated by wave functions. They subsequently extended this construction to two dimensions \cite{zbMATH06101981}, where they also addressed the additional difficulty arising from the logarithmic Poisson kernel. In a genuinely hydrodynamic framework, Antonelli, Marcati, and Zheng established the existence of global weak solutions for the QHD form of the Euler--Korteweg system in one dimension \cite{MR4244267}.  In their subsequent multidimensional work, they constructed global weak solutions for two classes of data: two-dimensional data with quantized point vortices and multidimensional spherically symmetric data \cite{MR4562812}. In both the one-dimensional and multidimensional settings, they also obtained dispersive estimates and stability results.

For the three-dimensional EPK system, Song \cite{zbMATH08045545} established global well-posedness and scattering for small, smooth, localized perturbations of a constant equilibrium, and proved convergence to Euler--Poisson solutions in the vanishing-capillarity limit. We also note that the beam equation studied in \cite{zbMATH08053129} has the same dispersion relation as the EPK system. The present paper addresses the corresponding two-dimensional problem: it proves global existence and scattering in the QHD-type case and establishes a long lifespan for general Korteweg tensors, while allowing non-neutral perturbations.

\subsection{New Ingredients}\label{sec:newingre}
The analysis in this paper relies on the space-time resonance method, which is used to estimate space-time oscillatory multilinear operators, such as the bilinear operator
\begin{equation*}
	T_{\gti m}(f,g)=\hti F^{-1} \int \me^{\mi s \Phi(\xi,\eta)} \gti m(\xi,\eta) \mao f(\xi-\eta,s) \mao g(\eta,s) \dif \eta \dif s,
\end{equation*}
where $\hti F^{-1}$ denotes the inverse Fourier transform. In our setting, the phase function $\Phi$ has a special form $-\Lambda(\xi) \pm\Lambda(\xi-\eta) \pm\Lambda(\eta)$, which we call the resonance function, also referred to as the modulation. Such resonance functions play a central role in the space-time resonance method. The space-time resonance method is based on the stationary phase approach and uses integration by parts in frequency or time to gain decay. This strategy cannot be directly implemented on the space resonant set $\set{\nabla_\eta \Phi(\xi,\eta)=0}$ or the time resonant set $\set{\Phi(\xi,\eta)=0}$. Their intersection is called the space-time resonant set. Near these resonant sets, one usually needs a more refined analysis. For a more detailed discussion of the space-time resonance method, we refer to \cite{zbMATH06025001,MR3784694}.

The new ingredients are twofold.

First, the low-frequency components of the solution exhibit slow decay due to the absence of a null structure. At the level of decay, the choice of the $Z$-norm yields the bound $\norm{P_kU(t)}_\infty \lle  2^{-k}t^{-1}$ \footnote{$P_k$ denotes the Littlewood--Paley projection; see Section~\ref{notations}.} (for simplicity we use a rough bound), which involves a frequency loss of $2^{-k}$. This loss is harmful in the low-frequency regime and can be compensated by a null structure (the derivative factors in the multiplier). On the other hand, one also has a rough bound $\norm{P_kU(t)}_\infty \lle t^{-1/2}$, which avoids the frequency loss but provides a weaker decay rate. The term $\abs{\nabla \psi}^2$ in \eqref{meq0} lacks the null structure. This difficulty arises primarily in Section~\ref{sec:badfre}, Section~\ref{subsec:3.4.2}, and, most notably, Section~\ref{section:cubic}. The key structure is that
\[
\min\{\abs{\xi}, \abs{\xi-\eta} , \abs{\eta}\}\max\{\abs{\xi}, \abs{\xi-\eta} , \abs{\eta}\} \ll 1 \implies \abs{\Phi(\xi,\eta)} \gge 1,
\]
where $\Phi$ is the resonance function. This lower bound allows us to exploit the absence of the first, second, and third time resonances. Here second and third time resonances refer to time resonances that arise after the second and third iterations, respectively; see Section~\ref{notations}. We also make essential use of the absence of space resonances and the improved approximate-finite-speed-of-propagation bounds. We expect that our argument may also be useful for studying global solutions to the two-dimensional Euler--Poisson system for electrons with non-neutral initial data. Moreover, even in the presence of the null structure, the $\delta$-fractional derivatives in the $Z$-norm cause the low-frequency components of the solution to exhibit non-optimal decay, which leads to additional arguments in Section~\ref{sec:badfre} and \ref{sec:l2lemma}.

The second new point is an alternative treatment of the small modulation cases in the energy estimates. In \cite{MR3784694}, the Fourier integral operator technique is employed to address the small modulation cases, based on the non-degeneracy of the curvature quantity $\mao \Upsilon(\xi,\eta)=\nabla^2_{\xi,\eta}\Phi(\xi,\eta)[\mao{\nabla_\xi \Phi}^\perp,\mao{\nabla_\eta \Phi}^\perp]$\footnote{For a vector $v=(v_1,v_2) \in \bR^2$, we write $v^\perp=(-v_2,v_1)$ and, when $v \ne 0$, $\mao v=v/\abs v$.}. However, in the region $\abs{\xi} \simeq \abs{\xi-\eta} \simeq \abs{\eta} \simeq 1$, $\mao \Upsilon$ may vanish. In that case, \cite{MR3784694} estimates the size of the set $\set{\mao \Upsilon \approx 0, \Phi \approx 0}$, which relies on the non-degeneracy of $\det [(\nabla \Upsilon)^T,(\nabla \Phi)^T ]$\footnote{$\Upsilon$ denotes the same expression with unnormalized gradients.}. In our setting, the non-degeneracy is more difficult to verify due to the additional parameter $m(0)$. Instead, we adapt the almost orthogonality argument used in \cite{MR3450481,MR3283401} to the case $\abs{\xi} \simeq \abs{\xi-\eta} \simeq \abs{\eta} \simeq 1$; see Section~\ref{subsec:3.4.1}. Here we sketch the almost orthogonality argument. For
\begin{align*}
Tf(\xi)={}&\int \me^{\mi s \Phi(\xi,\eta)} \gti m(\xi,\eta) f(\eta,s) \dif \eta \dif s\\
={}&\int \Brac{\int \me^{\mi s \Phi(\xi,\eta)} \gti m(\xi,\eta) f(\eta,s) \dif \eta} \dif s
=:\int S_s f(s) \dif s,
\end{align*}
we aim to estimate $\norm{Tf}_{L^2}$. The key is to exploit the almost orthogonality in time to gain extra decay. For two different times $s$ and $s'$, we show that $(S_sf(s),S_{s'}f(s'))_{L^2}$ is negligible. It suffices to estimate the kernel of $(S_{s'})^*S_s$:
\[
K_{s,s'}(\eta,\eta')=\int \me^{\mi(s\Phi(\xi,\eta)-s'\Phi(\xi,\eta'))} \gti m(\xi,\eta) \cl{\gti m(\xi,\eta')} \dif \xi.
\]
We would like to integrate by parts in $\xi$ to obtain rapid decay of the kernel. Hence the key condition is a lower bound on
\[
s\nabla_\xi \Phi(\xi,\eta)-s'\nabla_\xi \Phi(\xi,\eta').
\]
When $\eta$ is essentially determined by $\xi$ due to the restrictions in $\gti m$, such a lower bound may follow. In our setting, the same mechanism applies, but each $\xi$ may correspond to two choices of $\eta$. We therefore insert suitable cutoff functions into $T$, which requires additional arguments.

\subsection{Sketch of the Proof}
Next, we sketch the proof of \thmref{qhdthm}. The local well-posedness of \eqref{meq}-\eqref{initial} follows from the paralinearized formulation \eqref{energymain} (see \cite{MR2805065}). The continuous propagation of the vector-field $L^2$-norm and the $Z$-norm can be proved by arguments similar to those in \cite[Proposition~2.1]{MR3784694} and \cite[Section~3]{MR3283401}, respectively. Hence our task is to prove a priori estimates for $U$ by a bootstrap argument. The bootstrap assumptions are
\begin{subequations}
\begin{gather}
	\norm{(\Lambda^\delta+\Lambda^{N_0})U(t)}_2+\norm{\Lambda^{\delta}\Omega^{N_1}U(t)}_2 \le \epsl_1 \spa{t}^{\delta^4}, \label{energyest}\\
	\norm{\Omega^{[0,N_2]}V(t)}_Z \le \epsl_1, \label{zest}
\end{gather}
\end{subequations}
for $t \in [0,T]$, where $\epsl_1=100\epsl_0$. Our goal is to prove that
\begin{subequations}
\begin{gather}
\norm{(\Lambda^{\delta}+\Lambda^{N_0})U(t)}_2+\norm{\Lambda^{\delta}\Omega^{N_1}U(t)}_2 \le (\epsl_1/2) \spa{t}^{\delta^4}, \label{desireboot1}\\
\norm{\Omega^{[0,N_2]} V(t)}_Z \le \epsl_1/2. \label{desireboot2}
\end{gather}
\end{subequations}
for $t \in [0,T]$.

The energy estimate \eqref{desireboot1} is proved for high-frequency and low-frequency components in different ways. For high-frequency energy, loss of derivatives arises and paralinearization is introduced to isolate the derivative-loss terms. The difficult case occurs when the input fails to exhibit optimal decay. We perform integration by parts in time for the large modulation cases, where the resulting loss of derivatives is handled by the symmetrization argument and the depletion factors. The small modulation cases are handled by the Fourier integral operator technique. For low-frequency energy, there is no issue of loss of derivatives, so the dispersive formulation is used. When the input and output frequencies are $\simeq 1$ and the input does not exhibit optimal decay, we employ the almost-orthogonality argument. When low frequencies appear in the inputs or outputs, the analysis relies on the absence of first, second, and third time resonances, which is similar to, but simpler than, the analysis in Section~\ref{section:cubic}.

For the dispersive estimate \eqref{desireboot2}, we distinguish between cases according to the modulation. In the small and intermediate modulation cases, no low frequency can occur among the inputs or outputs. We use the argument in \cite{MR3784694} to treat these cases. The non-resonant regime is excluded by exploiting the absence of space and time resonances and the improved approximate-finite-speed-of-propagation bounds, while in the resonant regime the factor $2^{-(1/2-18\delta)n}$ in the $Z$-norm plays a central role. For the large modulation cases (where low frequencies may appear in the inputs or outputs), integration by parts in time converts quadratic terms to cubic terms. Although the resulting terms are cubic, in the absence of a null structure we obtain only the decay rate $t^{-1} \cdot t^{-1/2}$, which is not sufficient to compensate for the growth of $t^2$: one factor of $t$ comes from the time integration, while the other comes from the weight $2^{(1-20\delta)j}$ in the $Z$-norm. The cubic terms are then handled through the absence of second and third time resonances, the absence of space resonances, and the improved approximate-finite-speed-of-propagation bounds.

\subsection{Organization}
The rest of this paper is organized as follows. Section~2 contains some
preliminaries. Sections~3 and~4 are devoted to the proofs of the energy estimate \eqref{desireboot1} and the dispersive estimate \eqref{desireboot2}, respectively.

\section{Preliminaries}
\subsection{Notation}\label{notations}
We set $a \wedge b:=\min \{a,b\}$ and $a \vee b:=\max \{ a,b \}$. Let $k_+:=k \vee 0$, $k_-:=(-k) \vee 0$, and $\wan k:=-k_-$. For a nonzero vector $a$,  we denote its unit vector by $\mao a := a/|a|$. Moreover, for a vector $a=(x,y) \in \bR^2$, we define $a^\perp:=(-y,x)$. The Fourier transform is defined by
\[
\FF(f)(\xi):=\mao f(\xi)=\frac 1{(2\pi)^2}\int f(x) \me^{-\mi x \cdot \xi} \dif x
\]
and the inverse Fourier transform is defined by
\[
\check f(x):= \int f(\xi) \me^{\mi x \cdot \xi} \dif \xi.
\]
We also use the convention 
\[
\mathcal F (m(D)f)(\xi):=m(\xi) \widehat f(\xi).
\]
For $M \ge 0$, we write $\Omega^{[0,M]} f$ for $\Omega^a f$ with $a \in [0,M] \cap \bZ$, and write $\Omega_{\xi,\eta} := \Omega_\xi + \Omega_\eta$. Define $B(x_0,r):=\set{x \in \bR^2 \colon \abs{x-x_0}<r}$ and $A(x_0;r_1,r_2):=\set{x \in \bR^2 \colon r_1<\abs{x-x_0}<r_2}$.

We denote by $E_t f := \me^{\mi t \Lambda(D)} f$ the evolution of a profile $f$ under the dispersive semigroup. 
When no confusion arises, we abbreviate $E_t f$  as $Ef$ or $E^+ f$ and $E_{-t} f$ as $E^- f$. For the solution $U$ and the profile $V$, we set $U^+:=U$, $V^+:=V$, $U^-:=\cl U$, and $V^-:=\cl V$. For an expression of the form  $U^\iota$, $V^\iota$, or $E^{-\iota} f$, we refer to $\iota$ as its sign.
We define the resonance function (modulation)
$
\Phi_{\iota_1 \iota_2}(\xi,\eta) := -\Lambda(\xi) + \iota_1 \Lambda(\xi-\eta) + \iota_2 \Lambda(\eta),
$
where $\iota_1,\iota_2 \in \set{+,-}$. (In some situations, the subscripts $\iota_1, \iota_2$ may be omitted.) We denote $\nabla_\eta \Phi$ by $\Xi$ and $\Omega_\eta \Phi$ by $\Theta$. Similarly, we define the second resonance function
$
\wan \Phi_{\iota_1 \iota_2 \iota_3} := -\Lambda(\xi) + \iota_1 \Lambda(\xi-\eta) + \iota_2 \Lambda(\eta-\sigma) + \iota_3 \Lambda(\sigma),
$
with gradient $\wan \Xi := \nabla_{\eta,\sigma} \wan \Phi$, and the third resonance function
$
\cl \Phi_{\iota_1\iota_2\iota_3\iota_4} := -\Lambda(\xi) + \iota_1 \Lambda(\xi-\eta) + \iota_2 \Lambda(\eta-\sigma) + \iota_3 \Lambda(\sigma-\zeta) + \iota_4 \Lambda(\zeta).
$
We say that a second (resp. third) time resonance occurs when the second (resp. third) resonance function vanishes.

We next introduce the dyadic decompositions.
Let $\varphi$ be a smooth bump function supported in $[-3/2,3/2]$ and equal to $1$ on $[-5/4,5/4]$. Set $\psi(\xi) := \varphi(\xi) - \varphi(2\xi)$. For any $l \in\bZ$, define
\[
\varphi_l(\xi) := \varphi(\xi/2^l), \quad 
\psi_l(\xi) := \psi(\xi/2^l).
\]
For $I \subset \bR$, set
\[
\fai_I:=\sum_{k \in I} \psi_k.
\]
In particular, $\fai_{\le \lambda}:=\fai_{(-\infty,\lambda]}$, $\fai_{<\lambda}:=\fai_{(-\infty,\lambda)}$, $\fai_{\ge \lambda}:=\fai_{[\lambda,\infty)}$, and $\fai_{>\lambda}:=\fai_{(\lambda,\infty)}$.
Moreover, for integers $ k \ge l$, let
\[
\varphi_k^l(\xi) :=
\begin{cases}
\psi_k(\xi), &\text{if } k>l,\\
\varphi_l(\xi), &\text{if } k=l,
\end{cases}
\]
and for integers $l_1 \le k \le l_2$, let
\[
\wan{\fai}_k^{l_1,l_2}(\xi) :=
\begin{cases}
\varphi_{\le l_1}(\xi), &\text{if } k=l_1,\\
\psi_k(\xi), &\text{if } l_1<k<l_2,\\
\varphi_{\ge l_2}(\xi), &\text{if } k=l_2.
\end{cases}
\]
For simplicity, we use the same notation for the corresponding radial functions on $\bR^2$.
Define $P_kf:=\psi_k(D)f$ for $k \in \bZ$, $Q^{(k)}_j f:=\fai^{k_-}_{j}(x)f$ for integers $j \ge k_-$, and $A^{(j)}_n f:=\wan \fai^{-j-1,0}_{-n}(2^{100}(\abs{D}-2\gamma_0))f$ for integers $0 \le n \le j+1$. Set $\hti A=\set{(k,j,n) \in \bZ^3 \colon j \ge k_-, 0 \le n \le j+1}$. We write $P_I$ for $\fai_I(D)$. Let
\[
Q_{jk} := Q^{(k)}_j P_k, \quad Q^*_{jk} := P_{[k-2,k+2]} Q^{(k)}_j P_k.
\]
We use the shorthand
\[
f_k := P_k f, \quad f_{k,j} := Q^*_{jk} f, \quad f_{k,j,n} := A^{(j)}_n Q^*_{jk} f, \quad f_{k,n}:=A_nP_kf, \quad f^{p} := \Omega^p f.
\]
Throughout this paper, we frequently use the decompositions $f=\sum_{(k,j,n) \in \hti A}f_{k,j,n}$ and $f_k=\sum_{(j,n) \colon (k,j,n) \in \hti A} f_{k,j,n}$. For brevity, we often omit the condition $(k,j,n) \in \hti A$. For dyadic parameters, $a \lle b$ means that $a \le b+C$ for some constant $C$.

Unless otherwise specified, we use $\xi_i$ and $\xi$ to denote the frequencies of the $i$-th input and the output frequency, respectively. We also write $\xi_{i_1\cdots i_n}$ for $\xi_{i_1}+\cdots +\xi_{i_n}$. We denote by $k_I$ the frequency localization parameter of $\xi_I$. For instance, $k$, $k_i$, and $k_{ij}$ denote the localization parameters of $\xi$, $\xi_i$, and $\xi_{ij}$, respectively. We use $\ul k$ and $\cl k$ to denote the minimum and maximum among all frequency localization parameters involved in a multilinear operator. In particular,
$
\ul k := \min\{k,k_1,k_2\}$ and $\cl k := \max\{k,k_1,k_2\}
$
in the bilinear case. We define
\[
\Phi^{\alpha_1,\cdots,\alpha_n}_{\iota_0\iota_1 \cdots\iota_n}:=-\iota_0\Lambda(\xi_{\alpha_1\cdots\alpha_n})+\iota_1\Lambda(\xi_{\alpha_1})+\cdots+\iota_n\Lambda(\xi_{\alpha_n}),
\]
where $\iota_i \in \set{+,-}$ for $0 \le i \le n$.

Fixing $t$, we take functions $q_0, \dots, q_{L+1}$ from $\mathbb{R}$ to $[0,1]$ such that
\begin{itemize}
\item $\abs{L-\log_2(2+t)} \leq 2$;
\item $\sum_{m=0}^{L+1} q_m(s)=1_{[0,t]}(s)$;
\item $\supp q_0 \subset [0,2]$, $\supp q_{L+1} \subset [t-2,t]$, $\supp q_m \subset [2^{m-1},2^{m+1}]$, $q_m \in C^1(\mathbb{R})$ and $\int_0^t \abs{q_m'(s)}\dif s \lesssim 1$ for $1 \leq m \leq L$.
\end{itemize}
Set $D=10^6$ and $D_1=10^3$. The cases $0 \le m \le D^4$ and $m=L+1$ are typically easier and will be omitted.

\subsection{Multilinear Operators}\label{multilinear}
We first define notation for multilinear operators. The inputs, such as $V^\iota$, typically carry sign labels; in what follows, we denote by $\iota_i$ the sign of the input $f_i$. We denote by $T_{\gti m}$ and $I_{\gti m}$ the multilinear operators and the oscillatory multilinear operators, respectively, given by
\begin{gather*}
	\FF T_{\gti m}(f_1,\cdots,f_n)(\xi):=\int_{\xi_1+\cdots+\xi_n=\xi} \gti m(\xi_1,\cdots,\xi_n) \mao{f_1}(\xi_1) \cdots \mao{f_n}(\xi_n) \dif \xi_1 \cdots \dif \xi_n,\\
	\FF I_{\gti m,t}(f_1,\cdots,f_n)(\xi):=\int_{\xi_1+\cdots+\xi_n=\xi} \me^{\mi t \Psi} \gti m(\xi_1,\cdots,\xi_n) \mao{f_1}(\xi_1) \cdots \mao{f_n}(\xi_n) \dif \xi_1 \cdots \dif \xi_n,
	\end{gather*}
where the phase function $\Psi=-\Lambda(\xi)+\iota_1\Lambda(\xi_1)+\cdots+\iota_n\Lambda(\xi_n)$. Sometimes, we write $I_{\gti m,m}$ instead of $I_{\gti m,t}$ to indicate that the variable $t$ in $I$ satisfies $t \simeq 2^m$. Their dual forms are denoted by $\bbT_{\gti m}$ and $\bbI_{\gti m}$, respectively, and are given by
\begin{gather*}
	\bbT_{\gti m}(f_1,\cdots,f_{n+1}):=\int_{\eta_1+\cdots+\eta_{n+1}=0} \gti m(\eta_1,\cdots,\eta_{n}) \mao{f_1}(\eta_1) \cdots \mao{f_{n+1}}(\eta_{n+1})  \dif \eta_1 \cdots \dif \eta_{n+1},\\
	\bbI_{\gti m,t}(f_1,\cdots,f_{n+1}):=\int_{\eta_1+\cdots+\eta_{n+1}=0} \me^{\mi t \Psi} \gti m(\eta_1,\cdots,\eta_{n}) \mao{f_1}(\eta_1) \cdots \mao{f_{n+1}}(\eta_{n+1})  \dif \eta_1 \cdots \dif \eta_{n+1},
\end{gather*}
where the phase function $\Psi=-\Lambda(\eta_{n+1})+\iota_1 \Lambda(\eta_1)+\cdots+\iota_n \Lambda(\eta_n)$. For bilinear and trilinear operators, we use the parametrizations $(\xi_1,\xi_2)=(\xi-\eta,\eta)$ and $(\xi_1,\xi_2,\xi_3)=(\xi-\eta,\eta-\sigma,\sigma)$, respectively. For the corresponding dual forms, we use the parametrizations $(\eta_1,\eta_2,\eta_3)=(\xi-\eta,\eta,-\xi)$ and $(\eta_1,\eta_2,\eta_3,\eta_4)=(\xi-\eta,\eta-\sigma,\sigma,-\xi)$, respectively. We use the same notation $\gti m$ for the multiplier under these parametrizations.

For a multiplier $\gti m(\xi_1,\cdots,\xi_n)$, we define the symbol norm $\norm{\gti m}_S:=\norm{\check{\gti m}}_1$. For a set $K$ of localization parameters, we use the $S_K$-norm to indicate that the cutoff functions associated with the localization parameters are inserted into $\gti m$. For instance,
\[
\norm{\gti m(\xi_1,\xi_2)}_{S_{k,k_1,k_2}}:=\norm{\gti m \psi_k(\xi_1+\xi_2)\psi_{k_1}(\xi_1)\psi_{k_2}(\xi_2)}_S
\]
and
\[
\norm{\gti m(\xi_1,\xi_2,\xi_3)}_{S_{k,k_1,k_2,k_3,k_{23}}}:=\norm{\gti m  \psi_k(\xi_1+\xi_2+\xi_3)\psi_{k_1}(\xi_1)\psi_{k_2}(\xi_2)\psi_{k_3}(\xi_3)\psi_{k_{23}}(\xi_{23})}_S.
\] 

We recall the following multilinear estimate.
\begin{theorem}
	For $p,p_1, \cdots, p_n \in [1,\infty]$ satisfying $1/p_1+\cdots+1/p_n=1/p$, we have
	\[
	\norm{T_{\gti m}(f_1,\cdots, f_n)}_p \lle \norm{\gti m}_S \norm{f_1}_{p_1} \cdots \norm{f_n}_{p_n}.
	\]
\end{theorem}
To estimate the symbol norms, we use the following criterion for localized multipliers.
\begin{proposition}
	Suppose that $\gti m(\xi_1,\cdots,\xi_n)$ is supported in $B(0,\lambda_1) \times \cdots \times B(0,\lambda_n)$. Assume that $\lambda_i^{\abs \alpha} \abs{\partial^\alpha_{\xi_i} \gti m} \le M$ for $1 \le i \le n$ and $\abs \alpha \le 3n$. Then $\norm{\gti m}_S \lle M$.
\end{proposition}
The following two properties are useful for estimating the symbol norms.
\begin{gather*}
	\norm{\gti m_1 \gti m_2}_S \lle \norm{\gti m_1}_S \norm{\gti m_2}_S,\\
	\norm{\gti m \circ L}_S \lle_L \norm{\gti m}_S,
\end{gather*}
where $L$ is an invertible linear transformation.
For further discussion of these results, see \cite[Section~3.1]{MR3665671}.

We introduce the symbol classes
\[
\gti S^m:=\set{f \in C^\infty(\bR^2 \setminus \set{0}) \colon \abs{\xi}^k \abs{D^kf} \le C_k \spa{\xi}^m},
\]
and $\gti S:=\bigcup_{m \ge 0} \gti S^{m}$. We also define the Coifman--Meyer multiplier class
\[
\gti M^0:=\set{m \in C^\infty(\bR^4 \setminus \set{0}) \colon (\abs{\xi}+\abs{\eta})^k\abs{D^k m} \le C_k}.
\]

We define the space-time integral operators
\begin{gather*}
\FF\BB_{\gti m}(f_1,f_2)(\xi)
:= \int  \me^{\mi s\Phi_{\iota_1\iota_2}} \gti m(\xi,\eta) \mao f_1(\xi-\eta,s)\mao f_2(\eta,s)\, \dif \eta \dif s,\\
\begin{aligned}
	\FF\TT_{\gti m}(f_1,f_2,f_3)(\xi)
:={}&\int  \FF I_{\gti m,s}(f_1(s),f_2(s),f_3(s))(\xi) \dif s\\
={} &\int \me^{\mi s\wan{\Phi}_{\iota_1\iota_2\iota_3}} \gti m(\xi,\eta,\sigma) \mao f_1(\xi-\eta,s) \mao f_2(\eta-\sigma,s)\mao f_3(\sigma,s) \dif\eta \dif\sigma \dif s,
\end{aligned}
\\
\II_{\gti m,m}(f_1,\cdots,f_n)
:= \int q'_m(s) I_{\gti m,s}(f_1(s),\cdots,f_n(s)) \dif s.
\end{gather*}
When the multiplier is omitted, we take it to be $\gti m_q$ for $\BB$ and $\gti m_c$ for $\TT$. Here $\gti m_q$ and $\gti m_c$ are defined in Section~\ref{dispersiveformu}. Moreover, we define the quadrilinear space-time integral operator
\[
\QQ_{\gti m}(f_1,f_2,f_3,f_4) :=\int  I_{\gti m,s}(f_1(s),f_2(s),f_3(s),f_4(s)) \dif s.
\]
In some situations, we use multiple subscripts to encode localizations. For instance,
\[
\FF\BB_{m,l}(f_1,f_2)(\xi)=\int q_m(s) \me^{\mi s \Phi_{\iota_1\iota_2}} \fai^{l_0}_l(\Phi)\gti m_q(\xi,\eta)  \mao{f_1}(\xi-\eta,s) \mao{f_2}(\eta,s) \dif \eta \dif s.
\]
$m$ and $l$ denote the localization parameters for time and the first resonance function, respectively, throughout the paper.
We also adopt the convention that a superscript $'$ means division of the multiplier by $-\mi \Phi$. For instance,
\[
\FF\BB'_m(f_1,f_2)(\xi):=\int q_m(s) \me^{\mi s \Phi_{\iota_1\iota_2}} \frac{\gti m_q(\xi,\eta)}{-\mi \Phi} \mao{f_1}(\xi-\eta,s) \mao{f_2}(\eta,s) \dif \eta \dif s.
\]
We define the following space-time integrals 
\begin{gather*}
\TTT_{\gti m}(f_1,\cdots,f_n):=\int \bbT_{\gti m}(f_1(s),\cdots,f_n(s)) \dif s,\\
\gti R_{m,\gti m}(f_1,\cdots,f_n):=\int q'_m(s) \bbT_{\gti m}(f_1(s),\cdots,f_n(s)) \dif s,\\
\III_{\gti m}(f_1,\cdots,f_n):=\int \bbI_{\gti m,s}(f_1(s),\cdots,f_n(s)) \dif s,\\
\JJJ_{m,\gti m}(f_1,\cdots,f_n):=\int q'_m(s) \bbI_{\gti m,s}(f_1(s),\cdots,f_n(s)) \dif s,
\end{gather*}
which distinguish between oscillatory and non-oscillatory contributions, and also between bulk and boundary terms arising from integration by parts in time.

\subsection{Resonance Analysis}
Recall that the dispersion relation is $\Lambda(\xi)=\sqrt{\abs{\xi}^4+\abs{\xi}^2+1}$. We first collect some properties of $\lambda(r)=\sqrt{r^4+r^2+1}$. We have
\begin{gather*}
	\lambda'(r)=\frac{2r^3+r}{\sqrt{r^4+r^2+1}},\\
\lambda''(r)=\frac{2r^6 + 3r^4 + 6r^2 + 1}{(r^4 + r^2 + 1)^{3/2}}.
\end{gather*}
The convexity of $\lambda$ plays an important role in the analysis. Let $\gamma_0=1/\sqrt 2$. One can then verify that
\[
\lambda(2\gamma_0)=2\lambda(\gamma_0).
\]
For fixed $r>0$, the function $f(s)=\lambda(r+s)-\lambda(s)$ is strictly increasing. Moreover, $f(0)<\lambda(r)$ and $f(\infty)>\lambda(r)$. Therefore, there exists a unique $\tau(r)$ such that
\[
\lambda(r+\tau(r))-\lambda(\tau(r))=\lambda(r).
\]
The function $\tau$ is useful in the study of time resonances. We have
\[
\tau'(r)=\frac{\lambda'(r)-\lambda'(r+\tau(r))}{\lambda'(r+\tau(r))-\lambda'(\tau(r))}<0.
\]
Hence the function $\tau$ is strictly decreasing. Then we have
\begin{gather*}
	\tau(r)=\frac{\gamma_0^2}r+O\Brac{\frac 1{r^2}} \quad \text{as } r \to \infty,\\
	\tau(r)=\frac{\gamma_0^2}r+O\Brac{r} \quad \text{as } r \to 0.
\end{gather*} 
\begin{lemma}\label{timereso}
Suppose that vectors $a,b \ne 0$ and $\abs{b} \le \tau(\abs a)$. The following inequality holds:
\[
	\Lambda(a)+\Lambda(b)-\Lambda(a+b) \gge [\tau(\abs{a})-\abs b]\abs{a}.
\]
\end{lemma}
\begin{proof}
We have
	\begin{align*}
		&\Lambda(a)+\Lambda(b)-\Lambda(a+b)\\
		\ge{} &\lambda(\abs a)+\lambda(\abs{b})-\lambda(\abs a+\abs b)\\
		={} &\lambda(\abs a+\tau(\abs a))-\lambda(\tau(\abs a))-(\lambda(\abs a+\abs b)-\lambda(\abs{b})).
	\end{align*}
	Let $F(x)=\lambda(\abs a+x)-\lambda(x)$. Then $F' \simeq \abs{a}$. Hence the conclusion follows.
\end{proof}
Throughout this subsection, we assume that $m \ge D^4$.
\begin{lemma}\label{resoregion}
	Let $k, k_1, k_2$ be such that $-\delta^2 m \le \ul k \le \cl k \le \delta^2m$. Assume that $\xi$, $\xi-\eta$, and $\eta$ lie in the annuli $A(0;2^{k-1},2^{k+1})$, $A(0;2^{k_1-1},2^{k_1+1})$, and $A(0;2^{k_2-1},2^{k_2+1})$, respectively. Set $\varrho=\abs{\xi}$, $\varrho_1=\abs{\xi-\eta}$, and $\varrho_2=\abs{\eta}$. If
	\[
	\abs{\Phi_{\iota_1\iota_2}(\xi,\eta)} \le \kappa_0 \quad \text{and} \quad \abs{\nabla_\eta \Phi_{\iota_1\iota_2}(\xi,\eta)} \le \kappa_1,
	\]
	with $\kappa_0, \kappa_1 \le 2^{-2\delta^2m}$, then $(\iota_1,\iota_2)=(+,+)$,
	\[
	\abs{\varrho-2\gamma_0} \lle \kappa_0+\kappa_1^2 2^{2\delta^2m}, \abs{\varrho_1-\gamma_0} \vee \abs{\varrho_2-\gamma_0} \lle \kappa_0+\kappa_12^{\delta^2 m}, \text{ and } \abs{\varrho_1-\varrho_2} \vee \abs{2\varrho_1-\varrho} \lle \kappa_1.
	\]
	Similarly, if
	\[
	\abs{\Phi_{\iota_1\iota_2}(\xi,\eta)} \le \kappa_0 \quad \text{and} \quad \abs{\nabla_\xi \Phi_{\iota_1\iota_2}(\xi,\eta)} \le \kappa_1,
	\]
	with $\kappa_0, \kappa_1 \le 2^{-2\delta^2m}$, then $(\iota_1,\iota_2)=(-,+)$,
	\[
	\abs{\varrho_2-2\gamma_0} \lle \kappa_0+\kappa_1^2 2^{2\delta^2m}, \abs{\varrho-\gamma_0} \vee \abs{\varrho_1-\gamma_0} \lle \kappa_0+\kappa_12^{\delta^2 m}, \text{ and } \abs{\varrho_1-\varrho} \vee \abs{2\varrho_1-\varrho_2} \lle \kappa_1.
	\]
\end{lemma}
\begin{proof}
	We only prove the first conclusion. Since $\abs{\nabla_\eta \Phi_{\iota_1\iota_2}} \le \kappa_1$,
	\begin{gather}
		\abs{\lambda'(\varrho_1)-\lambda'(\varrho_2)} \lle \kappa_1, \label{eq:2.1}\\
		1-\iota_1\iota_2 \mao{\xi-\eta} \cdot \mao{\eta} \lle \kappa_1^2 \label{eq:2.2}.
	\end{gather}
	It follows from \eqref{eq:2.1} and $\abs{\Phi_{\iota_1\iota_2}} \lle \kappa_0$ that $(\iota_1,\iota_2)=(+,+)$. By the cosine formula, \eqref{eq:2.2} yields
	\[
	\frac{(\varrho_1+\varrho_2-\varrho)(\varrho_1+\varrho_2+\varrho)}{2\varrho_1 \varrho_2} \lle \kappa_1^2.
	\]
	It follows that
	\begin{equation}\label{eq:2.3}
	\abs{\varrho_1+\varrho_2-\varrho} \lle \kappa_1^22^{\delta^2 m}.		
	\end{equation}
	Using \eqref{eq:2.1}, we have
	\begin{equation}\label{eq:2.4}
	\abs{\varrho_1-\varrho_2} \lle \kappa_1.		
	\end{equation}
	Combining \eqref{eq:2.3}, \eqref{eq:2.4}, and $\abs{\Phi_{\iota_1\iota_2}} \lle \kappa_0$, we deduce that
	\[
	\abs{\lambda(2\varrho_1)-2\lambda(\varrho_1)} \lle \kappa_0+\kappa_1 2^{\delta^2 m}.
	\]
	Hence
	\[
	\abs{\varrho_1-\gamma_0} \lle \kappa_0+\kappa_1 2^{\delta^2 m},
	\]
	and the same conclusion holds for $\varrho_2$.
	
The bounds	$\abs{\Phi_{\iota_1\iota_2}} \lle \kappa_0$ and \eqref{eq:2.3} yield 
	\[
	\abs{\lambda(\varrho)-\lambda(\varrho-\varrho_1)-\lambda(\varrho_1)} \lle \kappa_0+\kappa_1^2 2^{2\delta^2 m}.
	\]
	Expanding $F(\varrho,\varrho_1)=\lambda(\varrho)-\lambda(\varrho-\varrho_1)-\lambda(\varrho_1)$ around $(2\gamma_0,\gamma_0)$ leads to
	\[
	\abs{[\lambda'(2\gamma_0)-\lambda'(\gamma_0)](\varrho-2\gamma_0)} \lle \kappa_0+\kappa_1^2 2^{2\delta^2 m}+(\kappa_0+\kappa_1 2^{\delta^2 m})^2 \lle \kappa_0+\kappa_1^2 2^{2\delta^2 m}.
	\]
	Therefore we obtain
	\[
	\abs{\varrho-2\gamma_0} \lle \kappa_0+\kappa_1^2 2^{2\delta^2m}.	
	\]
\end{proof}
\begin{lemma}\label{iteratedreso}
	Suppose that $\xi-\eta$, $\eta-\sigma$, and $\sigma$ lie in the annulus $A(0;\gamma_0-2^{-D_1},\gamma_0+2^{-D_1})$. Assume that 
	\[
	\abs{\wan \Phi_{\iota_1\iota_2\iota_3}(\xi,\eta,\sigma)} \le \kappa_0 \quad \text{and} \quad \abs{\nabla_{\eta,\sigma}\wan \Phi_{\iota_1\iota_2\iota_3}(\xi,\eta,\sigma)} \le \kappa_1,
	\]
	with $\kappa_0, \kappa_1 \le 2^{-D_1}$. Then $(\iota_1,\iota_2,\iota_3)\in\{(-,+,+), (+,-,+), (+,+,-)\}$ and
	\[
	\abs{\nabla_\xi \wan \Phi(\xi,\eta,\sigma)} \lle \kappa_1.
	\]
\end{lemma}
\begin{proof}
	From $\abs{\nabla_{\eta,\sigma}\wan \Phi_{\iota_1\iota_2\iota_3}} \le \kappa_1$, we have 
	\begin{gather}
	\abs{\lambda'(\abs{\xi-\eta})-\lambda'(\abs{\eta-\sigma})}^2+\lambda'(\abs{\xi-\eta})\lambda'(\abs{\eta-\sigma}) (1-\iota_1\iota_2\mao{\xi-\eta} \cdot \mao{\eta-\sigma}) \lle \kappa_1^2 \tag{a},\\
	\abs{\lambda'(\abs{\eta-\sigma})-\lambda'(\abs{\sigma})}^2+\lambda'(\abs{\eta-\sigma})\lambda'(\abs{\sigma}) (1-\iota_2\iota_3\mao{\eta-\sigma} \cdot \mao{\sigma}) \lle \kappa_1^2 \tag{b}.
	\end{gather}
	It follows that
	\begin{gather*}
		\abs{\mao{\xi-\eta} \cdot \mao{\eta-\sigma}^\perp} \lle \kappa_1,\\
		\abs{\mao{\eta-\sigma} \cdot \mao{\sigma}^\perp} \lle \kappa_1,
	\end{gather*}
	which implies that
	\[
	\abs{\xi \cdot (\xi-\eta)^\perp} \lle \kappa_1, \ \abs{\xi \cdot (\eta-\sigma)^\perp} \lle \kappa_1, \ \abs{\xi \cdot \sigma^\perp} \lle \kappa_1.
	\]
	It follows that
	\[
	\abs{\xi-\eta}=\abs{\eta-\sigma}+O(\kappa_1)=\abs \sigma+O(\kappa_1).
	\]
	Since $\abs{\wan \Phi_{\iota_1\iota_2\iota_3}} \le \kappa_0$, we have
	\[
	\abs{-\lambda(\abs \xi)+(\iota_1+\iota_2+\iota_3)\lambda(\abs{\xi-\eta})} \lle \kappa_0+\kappa_1.
	\]
	Note that $\xi-\eta$, $\eta-\sigma$, and $\sigma$ lie in the annulus $A(0;\gamma_0-2^{-D_1},\gamma_0+2^{-D_1})$. It follows that
	\begin{equation}\label{eq:2.2.10}
	\abs{-\lambda(\abs \xi)+(\iota_1+\iota_2+\iota_3)\lambda(\gamma_0)} \lle 2^{-D_1}.		
	\end{equation}
	Hence $(\iota_1,\iota_2,\iota_3)=(+,+,+)$, $(-,+,+)$, $(+,-,+)$, or $(+,+,-)$, and $2^{-100} \le \abs{\xi} \le 2^{100}$. If $(\iota_1,\iota_2,\iota_3)=(+,+,+)$, then
	\[
	\xi \cdot (\xi-\eta), \ \xi \cdot (\eta-\sigma), \ \xi \cdot \sigma>0.
	\]
	It follows that
	\begin{equation}\label{eq:2.2.11}
		\begin{aligned}
		\xi&=\xi-\eta+\eta-\sigma+\sigma\\
		&=[(\xi-\eta) \cdot \mao \xi]\mao \xi+[(\eta-\sigma) \cdot \mao \xi]\mao \xi+(\sigma \cdot \mao \xi)\mao \xi+O(\kappa_1)\\
		&=\abs{\xi-\eta}\mao \xi+\abs{\eta-\sigma}\mao \xi+\abs \sigma \mao \xi+O(\kappa_1)\\
		&=3\gamma_0 \mao \xi+O(2^{-D_1}),
	\end{aligned}
	\end{equation}
	which contradicts \eqref{eq:2.2.10}. Therefore, the case $(\iota_1,\iota_2,\iota_3)=(+,+,+)$ is excluded. As the other three cases are similar, we only consider the case $(-,+,+)$. As in \eqref{eq:2.2.11}, we have
	\[
	\xi \cdot (\xi-\eta)<0, \ \xi \cdot (\eta-\sigma)>0, \ \xi \cdot \sigma>0,
	\]
	and
	\begin{align*}
		\xi&=\xi-\eta+\eta-\sigma+\sigma\\
		&=-\abs{\xi-\eta}\mao \xi+\abs{\eta-\sigma}\mao \xi+\abs{\sigma}\mao \xi+O(\kappa_1)\\
		&=\abs{\xi-\eta}\mao \xi+O(\kappa_1).
	\end{align*}
	Consequently,
	\[
	\nabla_\xi \wan \Phi=-\lambda'(\abs{\xi})\mao \xi-\lambda'(\abs{\xi-\eta})\mao{\xi-\eta}=O(\kappa_1).
	\]
\end{proof}

Set
\[
\Upsilon(\xi,\eta)=\nabla^2_{\xi,\eta}\Phi(\xi,\eta)[\nabla_\xi \Phi^\perp(\xi,\eta),\nabla_\eta \Phi^\perp(\xi,\eta)],
\]
and
\[
\mao \Upsilon(\xi,\eta)=\nabla^2_{\xi,\eta}\Phi(\xi,\eta)[\mao{\nabla_\xi \Phi}^\perp(\xi,\eta),\mao{\nabla_\eta \Phi}^\perp(\xi,\eta)].
\]
We next prove a property of $\Upsilon$. 
\begin{lemma}\label{upsilon}
	Suppose that $\abs{\xi-\eta} \le 2$ and $\abs{\xi}, \abs{\eta} \ge 2^{40}$.	Then $\abs{\mao{\Upsilon}(\xi,\eta)} \gge 1$.
\end{lemma}
\begin{proof}
	In this proof, we set $\Phi=\iota_1 \Lambda(\xi)+\Lambda(\xi-\eta)+\iota_2\Lambda(\eta)$. Then
	\[
	\nabla^2_{\xi,\eta} \Phi=-\frac{\lambda'(\abs{\xi-\eta})}{\abs{\xi-\eta}} (I-\mao{\xi-\eta} \otimes \mao{\xi-\eta})-\lambda''(\abs{\xi-\eta})\mao{\xi-\eta} \otimes \mao{\xi-\eta},
	\]
	and
	\begin{gather*}
		\nabla_\xi \Phi^\perp=\iota_1 \lambda'(\abs{\xi}) \mao \xi^\perp+\lambda'(\abs{\xi-\eta})\mao{\xi-\eta}^\perp,\\
		\nabla_\eta \Phi^\perp=\iota_2 \lambda'(\abs{\eta}) \mao \eta^\perp+\lambda'(\abs{\xi-\eta})\mao{\eta-\xi}^\perp.
	\end{gather*}
	Suppose that $\abs \xi \in [2^{k-1},2^{k+1}]$.
	The cosine formula yields
	\[
	1-\mao \xi \cdot \mao \eta=\frac{\abs{\xi-\eta}^2-(\abs{\xi}-\abs{\eta})^2}{2\abs \xi \abs \eta} \lle 2^{-2k}.
	\]
	Hence
	\[
	\abs{\mao \xi-\mao \eta} \simeq \sqrt{1-\mao \xi \cdot \mao \eta} \lle 2^{-k}.
	\]
	It follows that
\begin{gather*}
		\nabla_\xi \Phi^\perp=\iota_1 \lambda'(\abs{\xi}) \mao \eta^\perp+O(1),\\
		\nabla_\eta \Phi^\perp=\iota_2 \lambda'(\abs{\eta}) \mao \eta^\perp+O(1).
	\end{gather*}
	Thus,
	\[
	\abs{\Upsilon} \gge 2^{2k} \quad \text{and} \quad \abs{\mao \Upsilon} \gge 1.
	\]
\end{proof}

We next introduce two lemmas for use in Section~\ref{subsec:3.4.1}. Define $\rho_{\iota_1\iota_2}$ as follows:
\[
	\rho_{++}=\tau(2\gamma_0)+2\gamma_0, \quad
	\rho_{-+}=\tau(2\gamma_0), \quad
	\rho_{+-}=\gamma_0.
\]
Set
\[
T(s)=
\begin{cases}
	1 & \text{if } s>1;\\
	s & \text{if } -1 \le s \le 1;\\
	-1 & \text{if } s<-1.
\end{cases}
\]
Let $p_{\iota_1\iota_2}(r)>0$ satisfy
\[
	\begin{cases}
		-\lambda(r)+\lambda(2\gamma_0)+\lambda(p_{\iota_1\iota_2}(r))=0, & r > \rho_{++}-2^{-D_1/2} , \text{ if } (\iota_1,\iota_2)=(+,+);\\
		-\lambda(r)-\lambda(2\gamma_0)+\lambda(p_{\iota_1\iota_2}(r))=0, & r > \rho_{-+}-2^{-D_1/2},  \text{ if } (\iota_1,\iota_2)=(-,+);\\
		-\lambda(r)+\lambda(2\gamma_0)-\lambda(p_{\iota_1\iota_2}(r))=0, & \abs{r-\rho_{+-}}<2^{-D_1/2}, \text{ if } (\iota_1,\iota_2)=(+,-).
	\end{cases}
\]
For $\xi$ with $\abs \xi$ in the corresponding domain of $p_{\iota_1\iota_2}$, take $\theta_{\iota_1\iota_2}^{\pm}(\xi) \in \bS^1$ such that
\[
\theta_{\iota_1\iota_2}^{\pm}(\xi) \cdot \mao \xi=T \Brac{\frac{\abs{\xi}^2+p_{\iota_1\iota_2}(\abs{\xi})^2-4\gamma_0^2}{2\abs \xi p_{\iota_1\iota_2}(\abs{\xi})}},
\]
and $\pm \theta_{\iota_1\iota_2}^{\pm}(\xi) \cdot \mao \xi^\perp \ge 0$. We then define $p_{\iota_1\iota_2}^{\pm}(\xi)=p_{\iota_1\iota_2}(\abs{\xi}) \theta_{\iota_1\iota_2}^{\pm}(\xi)$. Note that in the case $(+,-)$, $\theta^{\pm}_{+-}(\xi)=-\mao \xi$.

We have the following properties of $\theta^{\pm}_{\iota_1\iota_2}$.
\begin{lemma}\label{propertyp}
Suppose that $(\iota_1,\iota_2)=(+,+)$ or $(-,+)$, $2^{-D_1} \le \abs{\xi} \le 2^{D_1}$, and $\abs{\xi} > \rho_{\iota_1\iota_2}-2^{-D_1/2}$. Then
\[
	\abs{\theta^{\pm}_{\iota_1\iota_2}(\xi) \cdot \mao \xi^\perp} \simeq (\abs \xi-\rho_{\iota_1\iota_2})_+^{1/2}.
\]
\end{lemma}
\begin{lemma}\label{reso:almostortho}
	Suppose that $(\iota_1,\iota_2) \in \set{(+,+),(-,+),(+,-)}$ and $2^{-D_1} \le \abs \xi, \abs{\xi-\eta}, \abs \eta \le 2^{D_1}$. If $\abs{\abs{\xi-\eta}-2\gamma_0} \le \kappa_0 \le 2^{-10D_1}$, $\mao \xi^\perp \cdot \mao \eta \ge 0$ $(\mao \xi^\perp \cdot \mao \eta \le 0)$, and $\abs{\Phi_{\iota_1\iota_2}(\xi,\eta)} \le \kappa_1 \le 2^{-10D_1}$, then the following results hold:
	\begin{enumerate}[label=(\roman*)]
		\item $\abs{\xi} \ge \rho_{\iota_1\iota_2}-2^{3D_1}(\kappa_1+\kappa_0)$ if $(\iota_1,\iota_2)=(+,+)$ or $(-,+)$, and $\abs{\abs{\xi}-\rho_{+-}} \le 2^{3D_1}(\kappa_1^{1/2}+\kappa_0^{1/2})$ if $(\iota_1,\iota_2)=(+,-)$;
		\item $\abs{\eta-p^+_{\iota_1\iota_2}(\xi)} \lle \kappa_1^{1/2}+\kappa_0^{1/2}$ $(\abs{\eta-p^-_{\iota_1\iota_2}(\xi)} \lle \kappa_1^{1/2}+\kappa_0^{1/2})$;
		\item $\left|\nabla_\xi\Phi_{\iota_1\iota_2} \bigl(\xi,p^\pm_{\iota_1\iota_2}(\xi)\bigr)\right| \gtrsim 1$.
	\end{enumerate}
\end{lemma}
\begin{proof}
	We begin with the cases $(\iota_1,\iota_2)=(+,+)$ and $(-,+)$. For (i) and (ii), we prove only the case $(\iota_1,\iota_2)=(+,+)$ and $\mao \xi^\perp \cdot \mao \eta \ge 0$, as the other cases are similar. We first verify (i). $\abs{\Phi_{++}(\xi,\eta)} \le \kappa_1$ yields
	\[
	-\lambda(\abs \xi)+\lambda(\abs{\xi-\eta})+\lambda(\abs \eta) \le \kappa_1, \text{ and } \abs \xi \ge \abs{\xi-\eta}.
	\]
	It follows that
	\[
	-\lambda(\abs \xi)+\lambda(\abs{\xi-\eta})+\lambda(\abs \xi-\abs{\xi-\eta}) \le \kappa_1.
	\]
	Together with $\abs{\abs{\xi-\eta}-2\gamma_0} \le \kappa_0$, we have
	\[
	-\lambda(\abs \xi)+\lambda(2\gamma_0)+\lambda(\abs \xi-2\gamma_0) \le \kappa_1+2^{2D_1}\kappa_0.
	\]
	Since the function $F(r)=-\lambda(r)+\lambda(2\gamma_0)+\lambda(r-2\gamma_0)$ is strictly decreasing on $(2\gamma_0,\infty)$ and $F(\rho_{++})=0$,
	\[
	\abs{\xi} \ge \rho_{++}-2^{3D_1}(\kappa_1+\kappa_0).
	\]

	We next prove (ii). It suffices to show that $\abs{\abs{\eta}-p_{++}(\abs \xi)} \lle \kappa_1^{1/2}+\kappa_0^{1/2}$ and $\abs{\mao \eta-\theta^+_{++}(\xi)} \lle \kappa_1^{1/2}+\kappa_0^{1/2}$. Since $\abs{\Phi_{++}(\xi,\eta)} \le \kappa_1$, we have
	\[
	\abs{-\lambda(\abs \xi)+\lambda(\abs{\xi-\eta})+\lambda(\abs \eta)} \lle \kappa_1.
	\]
	It follows that
	\[
	\abs{-\lambda(\abs \xi)+\lambda(2\gamma_0)+\lambda(\abs \eta)} \lle \kappa_1+\kappa_0.
	\]
	Combining this with
	\[
	-\lambda(\abs \xi)+\lambda(2\gamma_0)+\lambda(p_{++}(\abs{\xi}))=0,
	\]
	we obtain
	\[
	\abs{\abs \eta-p_{++}(\abs \xi)} \lle \kappa_1+\kappa_0.
	\]
	
	We next estimate the angular part. By the cosine formula,
	\[
	\mao \xi \cdot \mao \eta=\frac{\abs \xi^2+\abs \eta^2-\abs{\xi-\eta}^2}{2\abs{\xi} \abs{\eta}}.
	\]
	Moreover,
	\[
	\mao \xi \cdot \theta^+_{++}(\xi)=T \Brac{\frac{\abs{\xi}^2+p_{++}(\abs{\xi})^2-4\gamma_0^2}{2\abs \xi p_{++}(\abs{\xi})}}.
	\]
	It follows that
	\[
	\abs{\mao \xi \cdot \mao \eta-\mao \xi \cdot \theta^+_{++}(\xi)} \lle \kappa_1+\kappa_0.
	\]
	Let $(x_1,y_1)=(\mao \xi \cdot \mao \eta,\mao \xi^\perp \cdot \mao \eta)$ and $(x_2,y_2)=(\mao \xi \cdot \theta^+_{++}(\xi),\mao \xi^\perp \cdot \theta^+_{++}(\xi))$. It remains to show that $\abs{y_1-y_2} \lle \kappa_1^{1/2}+\kappa_0^{1/2}$. We have
	\begin{equation}\label{eq:2.2.20}
		\abs{y_1^2-y_2^2}=\abs{x_1^2-x_2^2} \lle \kappa_1+\kappa_0.
	\end{equation}
	This, together with $y_1, y_2 \ge 0$, gives $\abs{y_1-y_2} \lle \kappa_1^{1/2}+\kappa_0^{1/2}$.
	 
	 For (iii), if $(\iota_1,\iota_2)=(+,+)$, then $\abs{\nabla_\xi \Phi_{++}(\xi,p^{\pm}_{++}(\xi))} \gge 1$ follows from $\abs \xi \ge \rho_{++}-2^{-D_1/2}$ and $\abs{\xi-p^{\pm}_{++}(\xi)} \approx 2\gamma_0$. When $(\iota_1,\iota_2)=(-,+)$, since $\abs{\xi-p^{\pm}_{-+}(\xi)} \approx 2\gamma_0$, it suffices to consider the case $\abs \xi \approx 2\gamma_0$. In this case, $\abs{p^{\pm}_{-+}(\xi)}$ is not close to $4\gamma_0$ or $0$, and hence the desired conclusion follows.

	 For the case $(\iota_1,\iota_2)=(+,-)$, (ii) and (iii) follow by arguments similar to those for the case $(\iota_1,\iota_2)=(+,+)$. We only prove (i). $\abs{\Phi_{+-}(\xi,\eta)} \le \kappa_1$ gives
\[
	-\lambda(\abs \xi)+\lambda(\abs{\xi-\eta})-\lambda(\abs \eta) \ge -\kappa_1, \text{ and } \abs \xi < \abs{\xi-\eta}.
	\]
	It follows that
	\[
	-\lambda(\abs \xi)+\lambda(\abs{\xi-\eta})-\lambda(\abs{\xi-\eta}-\abs{\xi}) \ge -\kappa_1.
	\]
	Together with $\abs{\abs{\xi-\eta}-2\gamma_0} \le \kappa_0$, we have
	\[
	-\lambda(\abs \xi)+\lambda(2\gamma_0)-\lambda(2\gamma_0-\abs \xi) \ge -\kappa_1-2^{2D_1}\kappa_0.
	\]
	Set $F(r)=-\lambda(r)+\lambda(2\gamma_0)-\lambda(2\gamma_0-r)$. Then $F(\gamma_0)=F'(\gamma_0)=0$ and Taylor's theorem imply that
	\[
	-\lambda(r)+\lambda(2\gamma_0)-\lambda(2\gamma_0-r) \le -2^{-D_1/2}(r-\gamma_0)^2
	\]
	for $r \in (0,2\gamma_0)$.
	It follows that
	\[
	\abs{\abs{\xi}-\rho_{+-}} \le 2^{3D_1}(\kappa_1^{1/2}+\kappa_0^{1/2}).
	\]
\end{proof}

\subsection{Linear Estimates}
In this subsection, we establish several estimates for the profile $V$. We first estimate $\norm{\mao V(t)}_{L^2(r\dif r)L^\infty_\theta}$ and $\norm{\mao V(t)}_\infty$, where $(r,\theta)$ denotes the polar coordinates. For convenience, we consider general spatial localization operators. Let $\chi_0$ and $\chi_1$ be smooth functions supported in $B(0,2)$ and $A(0;1/2,2)$, respectively. Set $\cl{\chi}:=(\chi_0,\chi_1)$, and for integers $j \ge k$, define
\[
\cl{\chi}_j^k(x) :=
\begin{cases}
	\chi_0(x/2^j) & \text{ if } j=k;\\
	\chi_1(x/2^j) & \text{ if } j>k.
\end{cases}
\]
\begin{lemma}\label{fourierl2r}
\begin{align}
	&\begin{multlined}[t]
\norm{\Omega^{[0,N_3+100]}
	\FF\big(A_n^{(j)}P_{[k-2,k+2]}\cl{\chi}_j^{k_-}(x)P_kV\big)(t)}
	_{L^2(r \dif r)L^\infty_\theta}\\
	 \lle
	\epsl_1 2^{-10k_+}2^{-\delta k}
	2^{-(1-20\delta)(j+k)}2^{(1/2-18\delta)n},	
\end{multlined}
\\
	&\begin{multlined}[t]
\norm{\Omega^{[0,N_2]}
	\FF\big(A_n^{(j)}P_{[k-2,k+2]}\cl{\chi}_j^{k_-}(x)P_kV\big)(t)}
	_{L^2(r \dif r)L^\infty_\theta}\\
	 \lle
	\epsl_1 2^{-9k_+}2^{-\delta k}
	2^{-(1-21\delta)(j+k)}2^{(1/2-18\delta)n}
	\spa{t}^{\delta^4}.
\end{multlined}
	\label{fl2r}
\end{align}
\end{lemma}
\begin{proof}
The Sobolev interpolation inequality in the angular variable yields
	\begin{equation}\label{eq:2.3.20}
		\norm{f}_{L^2(r\dif r)L^\infty_\theta} \lle \norm{\norm{f}_{L^2_\theta}^{1-(2N')^{-1}}\norm{\Omega^{[0,N']}f}_{L^2_\theta}^{(2N')^{-1}}}_{L^2(r \dif r)}.
	\end{equation}
Applying \eqref{eq:2.3.20} with $f=\Omega^{[0,N_3+100]}\FF(A_n^{(j)}P_{[k-2,k+2]}\cl{\chi}_j^{k_-}(x)P_kV)(t)$ and $N'=1$, and with $f=\Omega^{[0,N_2]}\FF(A_n^{(j)}P_{[k-2,k+2]}\cl{\chi}_j^{k_-}(x)P_kV)(t)$ and $N'=1/\delta^2$, we obtain the desired estimates from \eqref{zest}, \eqref{energyest}, and \lemmaref{zproperty}(iii).
\end{proof}
\begin{lemma}\label{vinftyest}
\begin{align}
	\norm{\Omega^{[0,N_3+100]}\widehat {V_{k,j,n}}(t)}_\infty &\lle  \epsl_1 2^{-9k_+}2^{-k}2^{-\delta k}2^{-(1/2-21\delta)(j+k)}2^{(1/2-17\delta)n}, \label{vinftyest2} \\
	\norm{\Omega^{[0,N_2]}\widehat {V_{k,j,n}}(t)}_\infty &\lle  \epsl_1 2^{-9k_+}2^{-k}2^{-\delta k}2^{-(1/2-21\delta)(j+k)}2^{(1/2-17\delta)n}\spa{t}^{\delta^4}.\label{vinftyest1}
\end{align}
\end{lemma}
\begin{proof}
Assume that $a \in \set{0,1,\cdots, N_2}$. For any fixed $\theta \in\bS^1$, it follows from the Sobolev interpolation inequality in the radial variable that
\begin{align*}
	\abs{\Omega^a\widehat{V_{k,j,n}}(r\theta,t)} \lle{}  &\norm{\Omega^a\widehat{V_{k,j,n}}(r\theta,t)}^{1/2}_{L_r^2}\norm{\ppd r\Omega^a\widehat{V_{k,j,n}}(r\theta,t)}^{1/2}_{L_r^2}\\
	 \lle{} &2^{(j-k)/2} \sup_{\hti B}\norm{\FF(A_{n'}^{(j)}P_{[k'-2,k'+2]}\cl{\chi}_j^{k_-}(x)P_k\Omega^aV)(r\theta,t)}_{L^2(r \dif r)L^\infty_\theta},
\end{align*}
where $\hti B=\set{(n',k',\cl{\chi}_j^{k_-}) \colon \abs{n'-n} \le 2,\abs{k'-k} \le 2, \cl{\chi}_j^{k_-} \in \set{\fai^{k_-}_j,(x_i/2^j) \fai^{k_-}_j: 1 \le i \le 2}}$.
Then the desired estimates follow from \lemmaref{fourierl2r}.
\end{proof}
A similar argument yields the following corollary.
\begin{corollary}
\begin{align*}
	\norm{D^{\alpha}\Omega^{[0,N_3+100]}\widehat {V_{k,j,n}}(t)}_\infty &\lle  \epsl_1 2^{\abs \alpha j}2^{-9k_+}2^{-k}2^{-\delta k}2^{-(1/2-21\delta)(j+k)}2^{(1/2-17\delta)n},  \\
	\norm{D^{\alpha}\Omega^{[0,N_2]}\widehat {V_{k,j,n}}(t)}_\infty &\lle  \epsl_1 2^{\abs \alpha j}2^{-9k_+}2^{-k}2^{-\delta k}2^{-(1/2-21\delta)(j+k)}2^{(1/2-17\delta)n}\spa{t}^{\delta^4}.
\end{align*}	
\end{corollary}

We next establish decay estimates.
\begin{lemma}\label{decaylemma}
Suppose that
\begin{equation}\label{eq:2.2.1}
	\norm{\Omega^{[0,N_2]}W}_Z \lle \epsl_1.
\end{equation}
Then
\begin{equation}\label{decayest}
\begin{aligned}
	&\abs{(\me^{\mi t \Lambda(D)} \Omega^{[0,N_3]}W_{k,j,n})(x)}\\
	\lle{} &\epsl_1
	\begin{cases}
		\spa{t}^{-1}  2^{-9k_+}2^{-k}2^{-\delta k}2^{-1/7(j+k)},  &\text{if } n=0 \text{ and } \abs{x}/t \ge 2^{-10}2^k; \\
\spa{t}^{-1}  2^{-\delta j}2^{5\delta n},  &\text{if } n> 0 \text{ and } \abs{x}/t \ge 2^{-10}2^k; \\
	\spa{t}^{-1} 2^{-9k_+}2^{-k}2^{-\delta k} 2^{20\delta(j+k)}, &\text{if } n=0 \text{ and }\abs{x}/t < 2^{-10}2^k;\\
	\spa{t}^{-1} 2^{20\delta j}2^{-18\delta n}, &\text{if } n> 0 \text{ and } \abs{x}/t < 2^{-10}2^k.
	\end{cases}
\end{aligned}
\end{equation}
Moreover, if $\abs x/t < 2^{-10}2^k$ and $j \le (1-\delta^2)m+k$, where $\spa t \in [2^{m-1},2^{m+1}]$, then
\begin{equation}\label{decayest2}
	\abs{\me^{\mi t \Lambda(D)} \Omega^{[0,N_3]}W_{k,j,n}(x)} \lle \epsl_1 \spa{t}^{-10}2^{-5k_+}.
\end{equation}
\end{lemma}
\begin{proof}
		The proof follows the argument in \cite[Section~5]{MR3318019}. We may assume that $t \gg 1$. We have
	\begin{equation}\label{decay1}
		\me^{\mi t \Lambda(D)} \Omega^a W_{k,j,n}(x)=\int \me^{\mi (t \Lambda(\xi)+x \cdot \xi)}\Omega^a \hhat{W_{k,j,n}}(\xi) \dif \xi,
	\end{equation}
	where $0 \le a \le N_3$. Let $f=\Omega^a \hhat{W_{k,j,n}}$. If $\abs{x}/t < 2^{-10}2^k$, then integrating by parts once yields the corresponding estimates in \eqref{decayest}, and repeated integration by parts gives \eqref{decayest2}. It remains to consider the case $\abs{x}/t \ge 2^{-10}2^k$. We decompose $f$ as the Fourier series $f(\rho,\theta)=\sum_m f_m(\rho) \me^{\mi m \theta}$. Using \eqref{eq:2.2.1} and the identity $\abs{m f_m}=\abs{(\Omega f)_m}$, we obtain
	\begin{gather*}
	\sum_m \spa{m}^{10} \norm{f_m}_\infty \lle \epsl_1 2^{-9k_+}2^{-k}2^{-\delta k}2^{-(1/2-21\delta)(j+k)}2^{(1/2-17\delta)n},\\
	\sum_m \spa{m}^{10} \norm{f_m}_2 \lle \epsl_1 2^{-10k_+}2^{-(1-20\delta)(j+k)}2^{(1/2-18\delta)n}2^{-\delta k}2^{-k/2}.
	\end{gather*}
	Thus it suffices to estimate
	\begin{align*}
	&\int \me^{\mi (t \Lambda(\rho)+x \cdot \rho \theta)} f_m(\rho) \me^{\mi m \theta}  \rho \dif \rho \dif \theta\\
	={}&\int \me^{\mi (t \Lambda(\rho)+r\rho)} f_m(\rho) Z^m_{+}(x,r \rho) \rho \dif \rho+\int \me^{\mi (t \Lambda(\rho)-r\rho)} f_m(\rho) Z^m_{-}(x,r \rho) \rho \dif \rho, 
	\end{align*}
	where $r=\abs x$. Here we use
	\[
	\int \me^{\mi x \cdot \rho \theta} \me^{\mi m \theta} \dif \theta=\me^{\mi r \rho}Z^m_+(x,r \rho)+\me^{-\mi r \rho}Z^m_-(x,r \rho),
	\]
	where $Z^m_{\pm}$ satisfies
	\[
	\abs{D^k Z^m_{\pm}(x,s)} \lle \spa{m}^{5} \spa{s}^{-1/2-k} 
	\]
	for $0 \le k \le 3$. We suppress the variable $x$ in $Z^m_\pm$ below. We estimate only the $Z^m_-$-part as the $Z^m_+$-part is easier. Set $r'=r/t$; then $r' \gge 2^k$. We now estimate the integral
	\begin{equation}\label{osilainte}
		\int \me^{\mi t(\Lambda(\rho)-\rho r')} f_m(\rho) Z^m_-(t\rho r')  \rho \dif \rho.
	\end{equation}
	
	 We insert the cutoff functions $\fai_{\ge 0}(t^{1/2}(\Lambda'(\rho)-r'))$ and $\fai_{<0}(t^{1/2}(\Lambda'(\rho)-r'))$ into \eqref{osilainte} to obtain $I_e$ and $I_i$, respectively. For $I_e$, integration by parts yields
	\begin{align*}
	\abs{I_e} &\lle \frac 1t \Abs{\int \me^{\mi t(\Lambda(\rho)-\rho r')} \ppd \rho [\fai_{\ge 0}(t^{1/2}(\Lambda'(\rho)-r'))(\Lambda'(\rho)-r')^{-1}f_m(\rho) Z^m_-(t\rho r') \rho] \dif \rho}\\
	&\lle t^{-1/2}(I_{e1}+I_{e2}+I_{e3}),
	\end{align*}
	where
	\begin{gather*}
		I_{e1}=\int \abs{\ppd \rho \wan{\fai_{\ge 0}}(t^{1/2}(\Lambda'(\rho)-r')) f_m(\rho) Z^m_-(t\rho r') \rho} \dif \rho,\\
		I_{e2}=\int \abs{\wan{\fai_{\ge 0}}(t^{1/2}(\Lambda'(\rho)-r')) \ppd \rho f_m(\rho) Z^m_-(t\rho r') \rho} \dif \rho,\\
		I_{e3}=\int \abs{\wan{\fai_{\ge 0}}(t^{1/2}(\Lambda'(\rho)-r')) f_m(\rho) \ppd \rho[Z^m_-(t\rho r') \rho]} \dif \rho,
	\end{gather*}
	where $\wan{\fai_{\ge 0}}(x)=\fai_{\ge 0}(x)/x$. For notational simplicity, we suppress the cutoff functions $\psi_k(\rho), \wan \fai^{-j-1,0}_{-n}(2^{100}(\rho-2\gamma_0))$ in the norms below. We have
	\begin{align*}
		I_{e1} &\lle \norm{f_m(\rho)}_\infty \spa{m}^5 \spa{t2^kr'}^{-1/2}2^k \norm{\wan \fai^{-j-1,0}_{-n}(2^{100}(\rho-2\gamma_0))\ppd \rho \wan{\fai_{\ge 0}}(t^{1/2}(\Lambda'(\rho)-r'))}_1\\
	&\lle \norm{f_m(\rho)}_\infty \spa{m}^5 \spa{t2^kr'}^{-1/2}t^{1/2}2^{k} \norm{\wan \fai^{-j-1,0}_{-n}(2^{100}(\rho-2\gamma_0))\wan{\fai_{\ge 0}}'(t^{1/2}(\Lambda'(\rho)-r'))}_1. 
	\end{align*}
The term $\norm{\wan{\fai_{\ge 0}}'(t^{1/2}(\Lambda'(\rho)-r'))}_1$ can be handled by the substitution $v=t^{1/2}(\Lambda'(\rho)-r')$. Hence,
	\[
	I_{e1} \lle \spa{m}^5 \norm{f_m(\rho)}_\infty (t^{-1/2} \wedge 2^{-n}).
	\]
	For $I_{e2}$, we have
	\begin{align*}
	I_{e2} &\lle \norm{\ppd \rho f_m(\rho)}_2 \spa{m}^5 \spa{t2^kr'}^{-1/2}2^k \norm{\wan \fai^{-j-1,0}_{-n}(2^{100}(\rho-2\gamma_0))\wan{\fai_{\ge 0}}(t^{1/2}(\Lambda'(\rho)-r'))}_2\\
	&\lle 2^j \spa{m}^5 \norm{f_m(\rho)}_2 t^{-1/2}(t^{-1/4}\wedge 2^{-n/2}).
	\end{align*}
	Similarly, for $I_{e3}$,
	\begin{align*}
		I_{e3} &\lle \norm{f_m(\rho)}_2 \spa{m}^5 \spa{t2^kr'}^{-1/2} \norm{\wan \fai^{-j-1,0}_{-n}(2^{100}(\rho-2\gamma_0))\wan{\fai_{\ge 0}}(t^{1/2}(\Lambda'(\rho)-r'))}_2\\
		&\lle \spa{m}^5 \norm{f_m(\rho)}_2 t^{-1/2}2^{-k}(t^{-1/4}\wedge 2^{-n/2}).
	\end{align*}
	For $I_i$, we estimate directly:
	\begin{align*}
		I_i &\lle \norm{f_m(\rho)}_\infty \spa{m}^5 \spa{t2^kr'}^{-1/2}2^k (t^{-1/2} \wedge 2^{-n})\\
	&\lle \spa{m}^5 \norm{f_m(\rho)}_\infty t^{-1/2} (t^{-1/2} \wedge 2^{-n}). 
	\end{align*}
	Combining the above estimates, we obtain
	\begin{equation}\label{decay1.1}
		\begin{aligned}
			\eqref{decay1} \lle &\epsl_1  t^{-1/2} (t^{-1/2} \wedge 2^{-n})2^{-9k_+} 2^{-k}2^{-\delta k}2^{-(1/2-21\delta)(j+k)}2^{(1/2-17\delta)n}\\
			&+\epsl_1 t^{-1}(t^{-1/4}\wedge 2^{-n/2})2^{-k/2}2^{-9k_+}2^{-\delta k}2^j2^{-(1-21\delta)(j+k)}2^{(1/2-18\delta)n}.
		\end{aligned}
	\end{equation}
	On the other hand, without using integration by parts, we have
	\begin{equation}\label{decay1.2}
	\begin{aligned}
		\eqref{decay1} &\lle \sum_m \spa{m}^5 \norm{f_m(\rho)}_\infty \spa{t2^{2k}}^{-1/2}2^{k}(2^{k}2^{-n})\\
		&\lle \epsl_1 t^{-1/2} (2^{k} \wedge 2^{-n})2^{-9k_+}2^{-k}2^{-\delta k}2^{-(1/2-21\delta)(j+k)}2^{(1/2-17\delta)n}.
	\end{aligned}
	\end{equation}
	Combining \eqref{decay1.1} and \eqref{decay1.2}, we obtain the desired estimate.
\end{proof}

In the following corollaries, we assume that the time $s$ in $V$ and the time $t$ in $\me^{\mi t \Lambda(D)}$ satisfy that $\spa s, \spa t \in [2^{m-4},2^{m+4}]$.
\begin{corollary}
\begin{equation}\label{decaycor}
		\norm{\me^{\mi t \Lambda(D)} \Omega^{[0,N_3]}V_{k,j,n}(s)}_\infty \lle \epsl_1
		\begin{cases}
			\spa{t}^{-(1-20\delta)} 2^{-9k_+}2^{-(1-39\delta)k}, &\text{if } n=0;\\
			\spa{t}^{-(1-20\delta)} 2^{-18\delta n}, &\text{if } n>0.
		\end{cases}
		\end{equation}
		The first estimate also holds for $V_k$.
\end{corollary}

We will use the following two interpolation inequalities from \cite[Section~2.4]{arXiv:2510.18655}.
\begin{lemma}\label{angularinte1}
Assume that $M,N \in \bN$, $M \le N$, and $N \ge 1$.
\[
	\norm{\Omega^M f}_\infty \lle \norm{f}^{1-\theta}_\infty\norm{\spa{D}^2\Omega^{[0,N]} f}_2^\theta,
	\]
	where $\theta=M/N$.
\end{lemma}
\begin{lemma}
Let $s_0, s \in \bR$ and $N_0, N \in \bN$ with $s>0$ and $N_0>N$. Suppose that $s_0 \ge \frac{N_0}{N_0-N}s$. Then
	\[
	\norm{\spa{D}^{s}\Omega^N f}_2 \lle \norm{\spa{D}^{s_0}f}_2+\norm{\Omega^{N_0}f}_2.
	\]
\end{lemma}

\begin{corollary}\label{decayestcom}
\[
\norm{\me^{\mi t \Lambda(D)} \Omega^{[0,N_2]}V_{k,j,n}(s)}_\infty
\]
\begin{numcases}{\lle \epsl_1}
\spa{t}^{-(1-\delta^2)}  2^{-4 k_+}2^{-\wan k}2^{-\delta \wan k},  &if $n=0$ and $j \le (1-\delta^2)m+k;$ \label{decayestg}\\
\spa{t}^{-(1-\delta^2)}  2^{4\delta n},  &if $n> 0$ and $j \le (1-\delta^2)m+k;$ \label{decayestgr}\\
\spa{t}^{-(1-21\delta)} 2^{-8k_+}2^{-(1-39\delta)\wan k}, &if $n=0;$ \label{decayestb}\\
\spa{t}^{-(1-21\delta)} 2^{-17\delta n}, &if $n>0$. \label{decayestbr}
\end{numcases}
\eqref{decayestb} also holds for $V_k$.
\end{corollary}
We have
\[
\norm{\me^{\mi t \Lambda(D)} \Omega^{[0,N_2]}V_{k,j,n}(s)}_\infty \lle \epsl_1 2^{(1-\delta)k}.
\]
By interpolating this bound with \eqref{decayestb}, we have the following decay estimate without loss of low frequencies.
\begin{corollary}\label{decayestlow}
\[
	\norm{\me^{\mi t \Lambda(D)} \Omega^{[0,N_2]}V_{k,j,n}(s)}_\infty \lle \epsl_1 \spa{t}^{-(1/2-2\delta)},
	\]
	if $k \le -10$.
	The above estimate also holds for $V_k$.
\end{corollary}
\begin{corollary}\label{decaydecom}
For any $p \in \set{0,1,\cdots,N_2}$ and $k \in \bZ$,
\[
\me^{\mi t \Lambda(D)} \Omega^{p} V_{k}(s)=I^p_k(s,t)+II^p_k(s,t),	
\]
where
\begin{gather*}
	\norm{I^p_k(s,t)}_\infty \lle \epsl_1 2^{-(1+\delta)k}\spa{t}^{-(1-5\delta)},\\
	\norm{II^p_k(s,t)}_2 \lle \epsl_1 2^{-(5/6)k}\spa{t}^{-1/2-1/100}.
\end{gather*}
Moreover, we have a similar decomposition for $\me^{\mi t \Lambda(D)} \Omega^{[0,N_2]} V_{k,j,n}(s)$.
\end{corollary}
\begin{proof}
We may assume that $t \gg 1$. We decompose $\me^{\mi t \Lambda(D)} \Omega^{[0,N_2]} V_{k}$ as 
\[
\sum_{j,n \colon (k,j,n) \in \hti A} \me^{\mi t \Lambda(D)} \Omega^{[0,N_2]} V_{k,j,n}.
\]
The terms with $j \le (1-\delta^2)m+k$ can be absorbed into $I^p_k$ as a consequence of \eqref{decayestg} and \eqref{decayestgr}. For the terms with $j \geq (1-\delta^2)m+k$, when $k \le -(2/5)m$, we estimate
\[
\norm{\me^{\mi t \Lambda(D)} \Omega^{[0,N_2]} V_{k,j,n}}_\infty
		\lle  2^k \norm{\Omega^{[0,N_2]} V_{k,j,n}}_2
		\lle \epsl_1 2^k2^{-(1-20\delta)(j+k)}2^{-\delta k},
\]
and hence these terms can be absorbed into $I^p_k$. When $k>-(2/5)m$, for the terms with $n \le 100$, we have
	\[
		\norm{\me^{\mi t \Lambda(D)} \Omega^{[0,N_2]} V_{k,j,n}}_2 \lle \epsl_1 2^{-(1-20\delta)(j+k)}2^{-\delta k}2^{-10k_+}.
	\]
	Therefore these terms can be absorbed into $II^p_k$. For the terms with $n>100$, if $n \le (9/8)m$ and $20j-18n \le 4m$, then \eqref{decayest}, together with \lemmaref{angularinte1}, yields the desired estimate. For the other cases, we have $j \ge (9/8)m$ or $n \le (10/9)j-(2/9)m$. The estimate \eqref{zest} then shows that these terms can be absorbed into $II^p_k$.
\end{proof}

\subsection{Reformulation for Dispersive Estimates}\label{dispersiveformu}
For \eqref{meq}, we extract linear terms as follows:
\begin{align*}
	\rho_t-\psi&=-R[\Lambda \rho R\psi],\\
	\psi_t+(1-\Delta)\rho+\Delta^2 \rho&=-\frac 12 \Lambda \abs{R\psi}^2+\Lambda[(m(\Lambda \rho)-1)\Delta\Lambda \rho] +\frac 12 \Lambda[m'(\Lambda \rho) \abs{\nabla \Lambda \rho}^2].
\end{align*}
Let
\[
U=\rho-\mi q(D) \psi,
\]
where
\[
q(\xi)=\frac 1 {\sqrt{\abs{\xi}^4+\abs{\xi}^2+1}}.
\]
Then we have
\begin{equation}\label{dispereq}
	U_t-\mi \Lambda(D) U=T_{\gti m_q}(U^{\iota_1},U^{\iota_2})+T_{\gti m_c}(U^{\iota_1},U^{\iota_2},U^{\iota_3})+R,
\end{equation}
where
\begin{itemize}
	\item The multiplier $\gti m_q$ has the generic form
	\[
	\gti m_q(\xi,\eta)=f_1(\xi) f_2(\xi-\eta)f_3(\eta),
	\]
	where $f_i \in \gti{S}^{4}$, $\norm{\gti m_q}_{S_{k,k_1,k_2}} \lle 2^{\cl k}2^{4\cl k_+}$, and $\Omega_{\xi,\eta}\gti m_q=0$;
	\item The multiplier $\gti m_c$ has the generic form
	\[
	\gti m_c(\xi,\eta,\sigma)=f_1(\xi) f_2(\xi-\eta) f_3(\eta-\sigma) f_4(\sigma),
	\]
	where $f_i \in \gti{S}^{4}$, $\norm{\gti m_c}_{S_{k,k_1,k_2,k_3}} \lle 2^{k}2^{k_1}2^{k_2}2^{k_3}2^{2\cl k_+}$, and $\Omega_{\xi,\eta,\sigma}\gti m_c=0$;
	\item The remainder $R$ has the following structure:
	\[
	R=m(D) [(D^{\alpha_1} \Lambda \rho) (D^{\alpha_2} \Lambda \rho) (D^{\alpha_3} \Lambda \rho) (D^{\alpha_4} \Lambda \rho) f(\Lambda \rho)],
	\]
	where $m \in \gti S^0$ satisfies $\Omega m=0$, $0 \le \abs{\alpha_i} \le 2$ for $1 \le i \le 4$, and $f \in C^\infty(\bR)$.
\end{itemize}
Indeed, the right-hand side is a sum over multiple choices of terms $T_{\gti m_q}(U^{\iota_1},U^{\iota_2})$,\\ $T_{\gti m_c}(U^{\iota_1},U^{\iota_2},U^{\iota_3})$, and $R$; for brevity, we display only one representative term. In dispersive estimates, the multipliers are typically products of factors in $\gti S^{100}$. For later use, we refer to a factor of the form $f(\xi_i)$ as a $\xi_i$-factor.

By Duhamel's formula, we have
\begin{multline*}
U(t)=\me^{\mi t \Lambda(D)}U(0)\\+\int_0^t \me^{\mi (t-s)\Lambda(D)} (T_{\gti m_q}(U^{\iota_1}(s),U^{\iota_2}(s))+T_{\gti m_c}(U^{\iota_1}(s),U^{\iota_2}(s),U^{\iota_3}(s))+R(s))\dif s.
\end{multline*}
Accordingly,
\begin{equation}\label{disperformula}
V(t)=V(0)+\BB_{\gti m_q}(V^{\iota_1},V^{\iota_2})+\TT_{\gti m_c}(V^{\iota_1},V^{\iota_2},V^{\iota_3})+\int_0^t \me^{-\mi s\Lambda(D)} R(s)\dif s,	
\end{equation}
which is the form we shall use to derive the dispersive estimates.

\subsection{Paralinearization}\label{paralinear}
We use Weyl quantization in the energy estimates. For $a=a(x,\xi)$, define
\[
\FF T_a f(\xi)=\int \mao a \Brac{\xi-\eta,\frac{\xi+\eta}2} \fai_{\le -10} \Brac{\frac{\abs{\xi-\eta}}{\abs{\xi+\eta}}} \mao f(\eta) \dif \eta
\]
where $\mao a$ denotes the Fourier transform of $a$ in the $x$-variable. We recall several properties of Weyl calculus. We define the Poisson bracket by
\[
\set{a,b} = \nabla_x a \cdot \nabla_{\xi}b - \nabla_x b \cdot \nabla_{\xi} a,
\]
and the symbol bound by
\[
\norm{a}^{l_1,l_2}_{m,p}=\sup_{\xi,\abs \alpha \le l_1, \abs \beta \le l_2} \spa{\xi}^{-m}\abs{\xi}^{\abs{\alpha}}\norm{D_\xi^\alpha D_x^\beta a(x,\xi)}_{L^p_x}.
\]
Moreover, define
	\begin{gather*}
	  E(a,b)=T_aT_b-T_{ab}-\tfrac{\mi}2 T_{\set{a,b}},\\
	  E(a,b,c)=T_aT_bT_c-T_{abc}-\tfrac{\mi}2 T_{\set{a,b}c+\set{b,c}a+\set{a,c}b}.
	\end{gather*}
Then the following properties hold:
\begin{itemize}
	\item $\norm{P_k T_a f}_p \lle 2^{mk_+}\norm{a}^{10,0}_{m,q} \norm{P_{[k-2,k+2]} f}_r$, where $1/p=1/q+1/r$ and $p,q,r \in [1,\infty]$.
	\item $\norm{P_kE(a,b)f}_2 \lle 2^{(m_1+m_2-2)k}\norm{a}^{20,2}_{m_1,\infty} \norm{b}^{20,2}_{m_2,\infty} \norm{P_{[k-5,k+5]}f}_2$ for $k \ge -100$.
	\item $T_a^*f=T_{\cl a}f$, where $*$ denotes the adjoint operator.
	\item $fg=T_f g+T_g f+R(f,g)$, where $R(f,g)=T_{\gti m}(f,g)$ with $\gti m$ satisfying $\norm{\gti m}_{S_{k,k_1,k_2}} \lle 2^{100(k_1 \wedge k_2)_+}2^{-100(k_1 \vee k_2)_+}$ and $\Omega_{\xi,\eta} \gti m=0$.
	\item Let $F$ be defined on $(-1,1)$, and satisfy that $F(0)=F'(0)=F''(0)=0$ and $\abs{D^{[0,n_1+100]}F} \le M$. Suppose that $\norm{\Omega^{[0,n_2/2]}u}_\infty \le 1/2$. Then $F(u)=T_{F'(u)}u+R$ and
	\begin{equation*}
	\norm{\spa{D}^{n_1+2}R}_2+\norm{\spa{D}^2\Omega^{n_2}R}_2 \lle \norm{\Omega^{[0,n_2/2]}u}^2_{W^{5,\infty}}(\norm{\spa{D}^{n_1}u}_2+\norm{\spa{D}^{-1}\Omega^{[0,n_2]}u}_2).
	\end{equation*}
	\item For $E(m(\xi),a), E(a,m(\xi)), E(m(\xi),a,n(\xi))$, we can compute them by applying the Taylor expansion to the $\xi$-dependent symbols. For instance,
	\begin{align*}
		\FF(E(m(\xi),a)f)={}&\int \Brac{\frac 12 D^2 m  \Brac{\frac{\xi+\eta}2} \SBrac{\frac{\xi-\eta}2,\frac{\xi-\eta}2}+R}\\ 
		&\mao a \Brac{\xi-\eta,\frac{\xi+\eta}2} \fai_{\le -10} \Brac{\frac{\abs{\xi-\eta}}{\abs{\xi+\eta}}}\mao f(\eta) \dif \eta.
	\end{align*}
\end{itemize}
In addition to these properties, \lemmaref{extrderisym} is used to handle terms such as $T_{\Lambda \Omega^{N_1}f}g$. We refer to \cite[Section~3.2]{MR3665671} for a detailed discussion of Weyl calculus. The approach in this subsection mainly follows \cite{MR3784694}.

Let
\begin{gather*}
\Sigma=\sqrt{(\Lambda \rho+1)[m(\Lambda \rho) \abs{\xi}^4+\abs{\xi}^2+1]},\\
B=m(\Lambda \rho) \abs{\xi}^4+\abs{\xi}^2+1,
\end{gather*}
and set
\[
\hti U=T_{\Sigma^{-1}}T_{B^{1/2}}\rho-\mi T_{B^{-1/2}}\psi.
\]
Then we have
\begin{gather*}
\norm{P_{\ge -10}\Lambda^{N_0}(\hti U(t)-U(t))}_2 \lle \epsl_1^2 \spa{t}^{-2/3},\\
\norm{P_{\ge -10}\Lambda^{\delta}\Omega^{N_1}(\hti U(t)-U(t))}_2 \lle \epsl_1^2 \spa{t}^{-2/3}.
\end{gather*}
We aim to prove that
\begin{equation}\label{energymain}
	\hti U_t-\mi T_\Sigma \hti U+T_{S} \hti U=\hti Q+\hti C,
\end{equation}
where $S=R\psi \cdot \mi \xi-\frac{1}{2}\mi \Delta \Lambda \rho$ and
\begin{itemize}
	\item $\hti Q=T_{\gti d}(U^{\iota_1},U^{\iota_2})+T_{\gti m}(U^{\iota_1},U^{\iota_2})$;
	\item The multipliers $\gti d, \gti m$ are of the form
	\[
	\gti m_0(\xi,\eta) f_1(\xi) f_2(\xi-\eta) f_3(\eta),
	\]
	where $\gti m_0 \in \gti M^0$ and $f_1,f_2,f_3 \in \gti S$;
	\item $\Omega_{\xi,\eta}\gti d=0$ and $\Omega_{\xi,\eta}\gti m=0$;
	\item $\norm{\gti d}_{S_{k,k_1,k_2}} \lle 2^{k_1 \wedge k_2}(1+2^{5(k_1 \wedge k_2)})1_{k_1 \ll k_2}$;
	\item If $\iota_2=+$, then
	 \[
	 \gti d=f_1(\xi-\eta) f_2(\xi+\eta) \chi\Brac{\frac{\abs{\xi-\eta}}{\abs{\xi+\eta}}}(\mao{\xi-\eta} \cdot \mao{\xi+\eta}),
	\]
	where $f_1\in \gti S^{10}$, $f_2 \in \gti S^0$, $\chi \in C_0^\infty(\bR)$, and $\supp \chi \subset [-1/100,1/100]$;
	\item $\norm{\gti m}_{S_{k,k_1,k_2}} \lle 2^{k_1 \wedge k_2}(1+2^{5(k_1 \wedge k_2)})/(1+2^{k_1 \vee k_2})$;
	\item $\norm{P_{\ge 0}\Lambda^{N_0}\hti C(t)}_2+\norm{P_{\ge 0} \Lambda^\delta \Omega^{N_1} \hti C(t)}_2 \lle \varepsilon_1^3 \spa{t}^{-1-1/100}$.
\end{itemize}
Indeed, the term $\hti Q$ is a sum over multiple choices of terms $T_{\gti d}(U^{\iota_1},U^{\iota_2})$ and $T_{\gti m}(U^{\iota_1},U^{\iota_2})$; for brevity, we display only one representative term. We call factors of the form $(\mao{\xi-\eta} \cdot \mao{\xi+\eta})$  \emph{depletion factors}; they are crucial in our analysis since such factors provide additional decay when the modulation is small.

We next perform paralinearization for \eqref{meq}. For \eqref{eqrho}, a term is called acceptable if it can be decomposed into $\hti Q+\hti C$ in \eqref{energymain}; such terms will be denoted by $\hti R_1$. Similarly, for \eqref{eqpsi}, a term $\hti R$ is called acceptable if $T_{B^{-1/2}}\hti R$ can be decomposed into $\hti Q+\hti C$ in \eqref{energymain}; these terms will be denoted by $\hti R_2$. After paralinearization, we obtain that
\begin{align*}
	&\rho_t+T_{\mi \mao \xi} T_{\Lambda \rho+1}T_{\mi \mao \xi} \psi+T_{\mi \mao \xi} T_{R \psi}T_{\abs{\xi}}\rho=\hti R_1,\\
	&\begin{multlined}[t]
	\psi_t+T_{\abs{\xi}}T_{R\psi}T_{\mi \mao \xi}\psi+T_{1+\abs{\xi}^2} \rho\\- T_{\abs{\xi}}[T_{m(\Lambda \rho)}T_{-\abs{\xi}^2}\Lambda \rho+T_{\Delta \Lambda \rho}m(\Lambda \rho)+T_{m'(\Lambda \rho)}T_{\nabla \Lambda \rho} T_{\mi \xi}\Lambda \rho]=\hti R_2.
	\end{multlined}	
\end{align*}
The verification of the acceptable terms is routine but lengthy. For the Korteweg tensor part, we compute
\begin{align*}
	&T_{\abs{\xi}}[T_{m(\Lambda \rho)}T_{-\abs{\xi}^2}\Lambda \rho+T_{\Delta \Lambda \rho}m(\Lambda \rho)+T_{m'(\Lambda \rho)}T_{\nabla \Lambda \rho} T_{\mi \xi}\Lambda \rho]\\
	={}&T_{\abs{\xi}}[T_{-m(\Lambda \rho) \abs{\xi}^2}\Lambda \rho+T_{-m'(\Lambda \rho)\nabla \Lambda \rho \cdot \mi \xi}\Lambda \rho +E(m(\Lambda \rho),-\abs{\xi}^2)\Lambda \rho- T_{\Delta \Lambda \rho} \Lambda \rho\\
	&+T_{m'(\Lambda \rho)}T_{\nabla \Lambda \rho \cdot \mi \xi}\Lambda \rho+T_{m'(\Lambda \rho)}T_{-\Delta \Lambda \rho/2}\Lambda \rho]+\hti R_2\\
	={}&T_{\abs{\xi}}[T_{-m(\Lambda \rho) \abs{\xi}^2}\Lambda \rho+E(-\Lambda \rho,-\abs{\xi}^2)\Lambda \rho- T_{\Delta \Lambda \rho} \Lambda \rho+\tfrac 12 T_{\Delta \Lambda \rho} \Lambda \rho]+\hti R_2\\
	={}&T_{\abs{\xi}}T_{-m(\Lambda \rho)\abs \xi^2}T_{\abs{\xi}}\rho-\tfrac{3}4 T_{\Delta \Lambda \rho \abs{\xi}^2}\rho+\hti R_2.
\end{align*}
The terms $T_{\mi \mao \xi} T_{\Lambda \rho+1}T_{\mi \mao \xi} \psi$, $T_{\mi \mao \xi} T_{R \psi}T_{\abs{\xi}}\rho$, $T_{\abs{\xi}}T_{R\psi}T_{\mi \mao \xi}\psi$, and $T_{\abs{\xi}}T_{-m(\Lambda \rho)|\xi|^2}T_{\abs{\xi}}\rho$ can be computed in a similar way, so we only discuss the last one. The key is to compute
\[
E(\abs{\xi},-m(\Lambda \rho)\abs \xi^2,\abs \xi) \rho.
\]
Since $m(0)=1$ and $m'(0)=-1$, we have
\begin{align*}
	&E(\abs{\xi},-m(\Lambda \rho)\abs \xi^2,\abs \xi) \rho=E(\abs{\xi},\Lambda \rho \abs \xi^2,\abs \xi) \rho+\hti R_2\\
={}& T_{\rho,\rho}\Bigl[\Bigl(\mao{\tfrac{\xi+\eta}2}^\perp \cdot \tfrac{\xi-\eta}2 \Bigr)^2 \abs{\xi-\eta}\Abs{\tfrac{\xi+\eta}2}^2 \chi \Brac{\tfrac{\abs{\xi-\eta}}{\abs{\xi+\eta}}}\Bigr]
- T_{\rho,\rho}\Bigl[\Bigl(\mao{\tfrac{\xi+\eta}2} \cdot \tfrac{\xi-\eta}2 \Bigr)^2 \abs{\xi-\eta}\Abs{\tfrac{\xi+\eta}2}^2 \chi \Brac{\tfrac{\abs{\xi-\eta}}{\abs{\xi+\eta}}}\Bigr]\\
&+\hti R_2\\
={}&-\tfrac{1}4 T_{\Delta \Lambda \rho \abs \xi^2} \rho+\hti R_2.
\end{align*}
Here we emphasize the multiplier in the bilinear operator $T$ by setting
\begin{align*}
	\hti F (T_{f,g}(\gti m))(\xi) = \int \gti m(\xi,\eta)\mao{f}(\xi-\eta)\mao{g}(\eta)\dif \eta.
\end{align*}
The formulas
\begin{gather*}
	D\abs{\xi} = \mao{\xi}, \quad 	D^2\abs{\xi} = \frac{1}{|\xi|} \mao{\xi}^{\perp} \otimes \mao{\xi}^{\perp},\\
	D\mao{\xi} = \frac{1}{|\xi|} \mao{\xi}^{\perp} \otimes \mao{\xi}^{\perp}, \quad
	D^2\mao{\xi} = -\frac{1}{|\xi|^2}(\mao{\xi} \otimes \mao{\xi}^\perp \otimes \mao{\xi}^\perp + \mao{\xi}^\perp \otimes \mao{\xi} \otimes \mao{\xi}^\perp + \mao{\xi}^\perp \otimes \mao{\xi}^\perp \otimes \mao{\xi}),
\end{gather*}
are useful in the computation.

It follows that
\begin{gather*}
	\rho_t-T_{\Lambda \rho+1}\psi-\tfrac 12 T_{\Delta \Lambda \rho \abs{\xi}^{-2}}\psi+T_{R\psi \cdot \mi \xi} \rho-T_{C}\rho=\hti R_1,\\
	\psi_t+T_{B}\rho+T_{\Delta \Lambda \rho \abs \xi^2}\rho+T_{R\psi \cdot \mi \xi}\psi+T_C \psi=\hti R_2,
\end{gather*}
where $C=(1/2) \Lambda \psi$.

We next verify \eqref{energymain}. We use $\hti R$ to denote a generic term of the form $\hti Q+\hti C$, with $\hti Q$ and $\hti C$ as in \eqref{energymain}. We compute
\begin{align*}
	&\hti U_t-\mi T_\Sigma \hti U+T_S \hti U\\
	={} &T_{\Sigma^{-1}}T_{B^{1/2}} \rho_t-\mi T_{B^{-1/2}}\psi_t-\mi T_\Sigma T_{\Sigma^{-1}}T_{B^{1/2}}\rho-T_\Sigma T_{B^{-1/2}}\psi \\
	&+	T_S(T_{\Sigma^{-1}}T_{B^{1/2}}\rho-\mi T_{B^{-1/2}}\psi)+
	  T_{(\Sigma^{-1})_t}T_{B^{1/2}} \rho+T_{\Sigma^{-1}}T_{(B^{1/2})_t} \rho-\mi T_{(B^{-1/2})_t}\psi\\
	={} &\underbrace{(T_{\Sigma^{-1}}T_{B^{1/2}}T_{\Lambda \rho+1}\psi-T_\Sigma T_{B^{-1/2}}\psi)}_{(1)}+\underbrace{\mi (T_{B^{-1/2}}T_B\rho- T_\Sigma T_{\Sigma^{-1}}T_{B^{1/2}}\rho)}_{(2)}\\
	&+\underbrace{[-T_{\Sigma^{-1}}T_{B^{1/2}}T_{R\psi \cdot \mi \xi}\rho+\mi T_{B^{-1/2}}T_{R\psi \cdot \mi \xi}\psi+T_{R\psi \cdot \mi \xi}(T_{\Sigma^{-1}}T_{B^{1/2}}\rho-\mi T_{B^{-1/2}}\psi)]}_{(3)}\\
	&+\underbrace{(T_{\Sigma^{-1}}T_{B^{1/2}}T_C\rho+\mi T_{B^{-1/2}}T_C\psi)}_{(4)}+\underbrace{[\tfrac 12 T_{\Sigma^{-1}}T_{B^{1/2}}T_{\Delta \Lambda \rho \abs{\xi}^{-2}}\psi+\mi T_{B^{-1/2}}T_{\Delta \Lambda \rho \abs{\xi}^2}\rho]}_{(5)}\\
	&+\underbrace{T_{(\Sigma^{-1})_t}T_{B^{1/2}} \rho+T_{\Sigma^{-1}}T_{(B^{1/2})_t} \rho-\mi T_{(B^{-1/2})_t}\psi}_{(6)}+T_{-(1/2)\mi \Delta \Lambda \rho}\hti U+\hti R.
\end{align*}
We expand
\begin{gather*}
\Sigma=\Sigma_0(1+\Sigma^1+\Sigma^{\ge 2}),\\
B=B_0(1+B^1+B^{\ge 2}),
\end{gather*}
where
\begin{gather*}
	\Sigma_0=B_0^{1/2}=\sqrt{\abs \xi^4+\abs \xi^2+1},\\
	\Sigma^1=\frac 12 \frac{\Lambda \rho (\abs \xi^2+1)}{\abs \xi^4+\abs \xi^2+1},\\
	B^1=\frac{-\Lambda \rho \abs{\xi}^4}{\abs \xi^4+\abs \xi^2+1}.
\end{gather*}
We now treat each group of terms. The zeroth- and first-order terms in the expansion of the Moyal product are straightforward, so we focus on the second-order terms. For (1),
\begin{align*}
	&T_{\Sigma^{-1}}T_{B^{1/2}}T_{\Lambda \rho+1}\psi-T_\Sigma T_{B^{-1/2}}\psi\\
	={} &E(\Sigma^{-1},B^{1/2},\Lambda \rho +1)\psi-E(\Sigma,B^{-1/2})\psi\\
	={} &E(\Sigma^{-1},B^{1/2},\Lambda \rho)\psi+E(\Sigma^{-1},B^{1/2})\psi-E(\Sigma,B^{-1/2})\psi\\
	={} &E(\Sigma_0^{-1},B_0^{1/2},\Lambda \rho)\psi+E(\Sigma_0^{-1}(1-\Sigma^1),B_0^{1/2}(1+\tfrac{1}{2}B^1))\psi\\
	&-E(\Sigma_0(1+\Sigma^1),B_0^{-1/2}(1-\tfrac{1}{2}B^1))\psi+\hti R\\
	={} &\tfrac 12 E(\Sigma_0^{-1},B_0^{1/2}B^1)\psi+\tfrac 12 E(\Sigma_0,B_0^{-1/2}B^1)\psi-E(\Sigma_0^{-1}\Sigma^1,B_0^{1/2})\psi-E(\Sigma_0 \Sigma^1,B_0^{-1/2})\psi+\hti R.
\end{align*}
Note that
\[
D^2[f(\abs \xi)]=f''(\abs \xi) \mao \xi \otimes \mao \xi+\tfrac{f'(\abs \xi)}{\abs \xi} \mao \xi^\perp \otimes \mao \xi^\perp.
\]
The first part contains depletion factors, so we focus on the second part. With this observation, it follows from a simple computation that
\[
T_{\Sigma^{-1}}T_{B^{1/2}}T_{\Lambda \rho+1}\psi-T_\Sigma T_{B^{-1/2}}\psi=\hti R.
\]
Similarly, for (2),
\begin{align*}
	&T_{B^{-1/2}}T_B\rho- T_\Sigma T_{\Sigma^{-1}}T_{B^{1/2}}\rho\\
	={} &E(B^{-1/2},B)\rho-E(\Sigma,\Sigma^{-1})T_{B^{1/2}}\rho\\
	={} &E(B_0^{-1/2}(1-\tfrac{1}{2}B^1),B_0(1+B^1))\rho-E(\Sigma_0(1+\Sigma^1),\Sigma_0^{-1}(1-\Sigma^1))T_{B_0^{1/2}}\rho+\hti R\\
	={} &-\tfrac 12 T_{\Delta \Lambda \rho}\rho+\hti R.
\end{align*}
Then it is easy to check that $(3)=\hti R$ and $(4)=T_C\cl{\hti U}+\hti R$. Moreover, $(5)=\tfrac 34 T_{\mi \Delta \Lambda \rho}\hti U+\hti R$ and $(6)=\hti R$. Therefore, the proof of \eqref{energymain} is complete.
\begin{remark}
	The QHD-type condition $m(0)+m'(0)=0$ makes $(6)$ acceptable.
\end{remark}

Commuting \eqref{energymain} with $P_{\ge 50}T_{\Sigma^{N_0/2}}$ and $P_{\ge 50}T_{\Sigma^{\delta/2}}\Omega^{N_1}$, we have
\begin{equation}\label{energymain1}
	\hti W_t-\mi T_\Sigma \hti W+T_{S} \hti W=\hti Q_{\hti W}+\hti C_{\hti W},
\end{equation}
where
\begin{itemize}
	\item $\hti W \in \gti W^0=\set{P_{\ge 50}T_{\Sigma^{N_0/2}} \hti U,P_{\ge 50}T_{\Sigma^{\delta/2}}\Omega^{N_1} \hti U}$;
	\item $\hti Q_{\hti W}=T_{\gti d}(G,W^1)+T_{\gti m}(G,W^1)+T_{\gti m'}(G,W^1)$ where
	\[
	G \in \gti G=\set{\spa{\nabla}^5 \Omega^{[0,N_1/2]} U^{\pm}}
	\]
	and
	\[
	W^1 \in \gti W=\set{T_{\Sigma^{N_0/2}} \hti U^{\pm},T_{\Sigma^{\delta/2}}\Omega^{N_1} \hti U^{\pm}, \spa{\nabla}^{100} \Omega^{[0,N_1-1]} U^{\pm}}
	\]
	(we let $\gti W^{+}=\set{T_{\Sigma^{N_0/2}} \hti U^{+},T_{\Sigma^{\delta/2}}\Omega^{N_1} \hti U^{+}, \spa{\nabla}^{100} \Omega^{[0,N_1-1]} U^{+}}$ and define $\gti W^-$ and $\gti G^{\pm}$ analogously);
	\item The multipliers $\gti d, \gti m, \gti m'$ are of the form
	\[
	\gti m_0(\xi,\eta) f_1(\xi) f_2(\xi-\eta) f_3(\eta),
	\]
	where $\gti m_0 \in \gti M^0$ and $f_1, f_2, f_3 \in \gti S$;
	\item $\norm{\gti d}_{S_{k,k_1,k_2}} \lle 2^{k_1 \wedge k_2}(1+2^{5(k_1 \wedge k_2)})1_{k_1 \ll k_2}$;
	\item If $W^1 \in \gti W^+$, then
	 \[
	 \gti d=f_1(\xi-\eta) f_2(\xi+\eta) \chi\Brac{\frac{\abs{\xi-\eta}}{\abs{\xi+\eta}}}(\mao{\xi-\eta} \cdot \mao{\xi+\eta}),
	 \]
	 where $f_1 \in \gti S^{10}$, $f_2 \in \gti S^0$, $\chi \in C_0^\infty(\bR)$, and $\supp \chi \subset [-1/100,1/100]$;
	\item $\norm{\gti m}_{S_{k,k_1,k_2}} \lle 2^{k_1 \wedge k_2}(1+2^{5(k_1 \wedge k_2)})/(1+2^{k_1 \vee k_2})$;
	\item $\norm{\gti m'}_{S_{k,k_1,k_2}} \lle 2^{k_1 \vee k_2}2^{88(k_1)_+}(1+2^{90(k_1 \vee k_2)})^{-1}$;
	\item $\norm{\hti C_{\hti W}(t)}_2 \lle \varepsilon_1^3 \spa{t}^{-1-1/100}$.
\end{itemize}
As before, the term $\hti Q_{\hti W}$ is a sum over multiple choices of terms $T_{\gti d}(G,W^1)$, $T_{\gti m}(G,W^1)$, and $T_{\gti m'}(G,W^1)$; for brevity, we display only one representative term.
\begin{remark}
	The QHD-type condition ensures that $\Sigma_0\Sigma^1$ is of order $0$. Without this condition, when we commute the equation for $\hti U$ with $P_{\ge 50}T_{\Sigma^{N_0/2}}$, the term $P_{\ge 50}T_{(\Sigma^{N_0/2})_t}\hti U$ contains a quadratic contribution not belonging to any of the three classes of terms defining $\hti Q_{\hti W}$.
\end{remark}

Finally, we will use the following estimates to derive the energy bounds for $U$:
\begin{subequations}\label{paraerror}
\begin{gather}
\norm{P_{\ge 50}(T_{\Sigma^{N_0/2}} \hti U(t)-T_{\Sigma_0^{N_0/2}}U(t))}_2 \lle \epsl_1^2 \spa{t}^{-2/3},\\
\norm{P_{\ge 50}(T_{\Sigma^{\delta/2}}\Omega^{N_1} \hti U(t)-T_{\Sigma_0^{\delta/2}}\Omega^{N_1}U(t))}_2 \lle \epsl_1^2 \spa{t}^{-2/3}.
\end{gather}	
\end{subequations}

\subsection{Integration by Parts with respect to Vector Fields}\label{integrationbyparts}
Define
\[
\gti M^{a,b,c,e}=\{ \abs{\xi}^a \abs{\xi-\eta}^b\abs{\eta}^c f_1(\xi)f_2(\xi-\eta)f_3(\eta) 2^{-e p}\psi_p(\Omega_\eta \Phi):  f_1,  f_2,  f_3 \in \gti S^0 \}.
\]
Then we have the following lemma concerning integration by parts with respect to vector fields, which follows from \cite{MR3024265}.
\begin{lemma}\label{rotainte}
	Let
	\[
	I=\psi_k(\xi) \int \me^{\mi t \Phi}  \gti m(\xi,\eta) \mao {f_{k_1,j_1,n_1}}(\xi-\eta) \mao {g_{k_2,j_2,n_2}}(\eta)  \dif \eta,
	\]
	where $\gti m \in \gti M^{0,0,0,0}$. Suppose that $t \in [2^{m-1},2^{m+1}]$ and $m \ge 10$. Then
	\[
	\abs{I} \lle \Brac{\frac{2^{k+k_2-2p}+2^{k_2-k_1-p}+2^{-p}+2^{j_1-k_1}}{2^m}}^N\norm{\Omega^{[0,N]}f_{k_1,j_1,n_1}}_2 \norm{\Omega^{[0,N]}g_{k_2,j_2,n_2}}_2.
	\]
\end{lemma}
\begin{proof}
	We integrate by parts with respect to $\Omega_\eta$ $N$ times. The key is to keep track of how the multiplier class $\gti M^{a,b,c,e}$ changes after each integration by parts. We have
	\begin{alignat*}{2}
	\gti M^{a,b,c,e} &\longrightarrow \gti M^{a,b-1,c+1,e+1}, &\quad &\text{if } \Omega_\eta \text{ hits } f_2(\xi-\eta);\\
	\gti M^{a,b,c,e} &\longrightarrow \gti M^{a,b,c,e+1}, & &\text{if } \Omega_\eta \text{ hits } f_3(\eta);\\		
	\gti M^{a,b,c,e} &\longrightarrow \gti M^{a+1,b,c+1,e+2}, \; \gti M^{a,b-1,c+1,e+1}, & &\text{if } \Omega_\eta \text{ hits } \psi_p(\Omega_\eta \Phi);\\
	\gti M^{a,b,c,e} &\longrightarrow \gti M^{a,b-1,c+1,e+1}, \; 2^{j_1}\gti M^{a,b-1,c,e}, & &\text{if } \Omega_\eta \text{ hits } \mao {f_{k_1,j_1,n_1}}(\xi-\eta);\\
	\gti M^{a,b,c,e} &\longrightarrow \gti M^{a,b,c,e+1}, & &\text{if } \Omega_\eta \text{ hits } \mao {g_{k_2,j_2,n_2}}(\eta);
	\end{alignat*}
	at each step of integration by parts. The desired conclusion follows immediately from the preceding discussion.
\end{proof}

Integration by parts plays a central role throughout this paper. For an oscillatory integral of the form
\[
\int \me^{\mi \lambda f(x)} g(x) \dif x,
\]
integration by parts in $x$ gives
\begin{align*}
&\int \me^{\mi \lambda f(x)} \ppd x \Brac{-\frac{g(x)}{\mi \lambda f'(x)}} \dif x\\
={}&\int \me^{\mi \lambda f(x)} \frac{- g'(x)}{\mi \lambda f'(x)} \dif x+\int \me^{\mi \lambda f(x)} \frac{g(x) \mi \lambda f''(x)}{(\mi \lambda f'(x))^2}	\dif x.
\end{align*}
Relative to the original integral, the gain is $(\mi \lambda f'(x))^{-1}$, while the losses are $ g'(x)/g(x)$ and $f''(x)/ f'(x)$. If each step of integration by parts yields an additional decay factor, such as $t^{-\delta^2}$, then iterating this procedure sufficiently many times will provide enough decay.

\subsection{Size Estimates}
We first record some elementary facts.
\begin{lemma}\label{cos}
\[
	\abs{\set{\theta\in [0, 2\pi) \colon \cos \theta \in (x_0-\epsl, x_0+\epsl)}} \lle \epsl^{1/2},
	\]
	where the constant in the inequality is independent of $x_0$.
\end{lemma}
\begin{lemma}\label{size}
	Suppose that $f \in C^\infty(a,b)$ has at most $N$ critical points with $N \ge 1$. Then
	\[
	\abs{\set{x \in (a,b) \colon \abs{f(x)} \le \epsl, \abs{f'(x)} \ge L}} \lle N\epsl/L,
	\]
	where the constant in the inequality is independent of $f$ and $(a,b)$.
\end{lemma}
\begin{lemma}\label{second}
	Suppose that $f \in C^\infty(a,b)$ and $\abs{f''} \ge 1$. Then
	\[
	\abs{\set{x \in (a,b) \colon \abs{f(x)} \le \epsl}} \lle \epsl^{1/2},
	\]
    where the constant in the inequality is independent of $f$ and $(a,b)$.
\end{lemma}
We next have the following estimates from Schur's test.
\begin{lemma}\label{sizeest}
Let $m \ge D^4$. Assume that $k,k_1,k_2$ satisfy $-\delta^2 m \leq \ul{k} \leq \cl{k} \le \delta^2 m$,  $l \geq -2m$, and $2^{-100} \leq r_1, r_2 \leq 2^{100}$. Let $\chi_{k,k_1,k_2}$ be the cutoff function $\psi_k(\xi)\psi_{k_1}(\xi-\eta)\psi_{k_2}(\eta)$.
\begin{enumerate}[label=(\roman*)]
	\item
	\begin{align*}
		&\Norm{\int \chi_{k,k_1,k_2} \fai_{l}(\Phi(\xi,\eta)) \fai_{-n_1}(\abs{\xi-\eta}-r_1) \abs{ f_1(\xi-\eta)}\abs{f_2(\eta)} \dif \eta}_{L^2_\xi}\\
		&\lle 2^{5\delta^2 m}2^{-n_1/2}2^{l/2}\norm{f_1}_{L^2(r \dif r)L^\infty_\theta} \norm{f_2}_2.
	\end{align*}
	\item
	\begin{align*}
		&\Norm{\int \chi_{k,k_1,k_2} \fai_{l}(\Phi(\xi,\eta)) \fai_{-n_1}(\abs{\xi-\eta}-r_1) \fai_{-n_2}(\abs{\eta}-r_2) \abs{ f_1(\xi-\eta)}\abs{f_2(\eta)} \dif \eta}_{L^2_\xi}\\
		&\lle 2^{5\delta^2 m}2^{-n_1/2}2^{l/2}(2^{-n_2/2}\wedge 2^{l/2})\norm{f_1}_{L^2(r \dif r)L^\infty_\theta} \norm{f_2}_2,
	\end{align*}
	if $l \le -\delta^2 m$, $n_1$, $n_2 \ge \delta^2m$, and
	\[
	\abs{\lambda(r_1) \pm  \lambda(r_2) \pm  \lambda(\abs{r_1\pm  r_2})} \ge 2^{-100}.
	\]
	\item
	\begin{gather}
		\sup_\xi \int \chi_{k,k_1,k_2} \fai_{l}(\Phi(\xi,\eta)) \fai_{\log_2 \kappa_{\theta}}(\Theta(\xi,\eta))\fai_{\ge p_1}(\nabla_\eta \Phi(\xi,\eta)) \dif \eta \lle 2^{5\delta^2 m}2^l2^{-p_1}\kappa_\theta, \label{eq:2.7.1}\\
		\sup_\eta \int \chi_{k,k_1,k_2} \fai_{l}(\Phi(\xi,\eta)) \fai_{\log_2 \kappa_{\theta}}(\Theta(\xi,\eta))\fai_{\ge p_2}(\nabla_\xi \Phi(\xi,\eta)) \dif \xi \lle 2^{5\delta^2 m}2^l2^{-p_2}\kappa_\theta, \label{eq:2.7.2}
	\end{gather}
	if $\kappa_\theta \le 2^{p_1 \wedge p_2-2\delta^2 m}$.
	Consequently,
	\begin{align*}
		&\Biggl\lVert\int \chi_{k,k_1,k_2} \fai_{l}(\Phi(\xi,\eta)) \fai_{\log_2 \kappa_\theta}(\Theta(\xi,\eta))\\
		&\times\fai_{\ge p_1}(\nabla_\eta \Phi(\xi,\eta)) \fai_{\ge p_2}(\nabla_\xi \Phi(\xi,\eta))\abs{ f_1(\xi-\eta)}\abs{f_2(\eta)} \dif \eta \Biggr\rVert_{L^2_\xi}\\
		\lle{} &2^{5\delta^2 m}2^{l}2^{-p_1/2}2^{-p_2/2}\kappa_\theta \norm{f_1}_\infty \norm{f_2}_2.
	\end{align*}
	\item
	\begin{gather}
		\sup_\xi \int \chi_{k,k_1,k_2} \fai_l(\Phi(\xi,\eta)) \dif \eta \lle 2^{15\delta^2 m}2^l, \label{eq:2.7.5}\\
		\sup_\eta \int \chi_{k,k_1,k_2} \fai_l(\Phi(\xi,\eta)) \dif \xi \lle 2^{15\delta^2 m}2^l. \label{eq:2.7.6}
	\end{gather}
\end{enumerate}
\end{lemma}
\begin{proof}
\begin{enumerate}[label=(\roman*)]
	\item Without loss of generality, we assume that $l \le -3\delta^2 m$. By means of Schur's test, it suffices to estimate
	\begin{gather*}
		\sup_\xi \int \chi_{k,k_1,k_2} \fai_{l}(\Phi(\xi,\eta)) \fai_{-n_1}(\abs{\xi-\eta}-r_1) \abs{ f_1(\xi-\eta)} \dif \eta,\\
		\sup_\eta \int \chi_{k,k_1,k_2} \fai_{l}(\Phi(\xi,\eta)) \fai_{-n_1}(\abs{\xi-\eta}-r_1) \abs{ f_1(\xi-\eta)} \dif \xi.
	\end{gather*}
	For the first integral, we have
	\begin{align*}
		&\int \chi_{k,k_1,k_2} \fai_{l}(\Phi(\xi,\eta)) \fai_{-n_1}(\abs{\xi-\eta}-r_1) \abs{ f_1(\xi-\eta)} \dif \eta\\
\stackrel{\xi-\eta=r\theta}{=}{} &\int \chi_{k,k_1,k_2} \fai_{l}(\Phi(\xi,\xi-r\theta)) \fai_{-n_1}(r-r_1) \abs{ f_1(r\theta)} r\dif r \dif \theta\\
\lle{} &\int \Brac{\int  \chi_{k,k_1,k_2} \fai_{l}(\Phi(\xi,\xi-r\theta)) \dif \theta} \fai_{-n_1}(r-r_1)  \sup_\theta \abs{ f_1(r\theta)} r\dif r.
	\end{align*}
	We next estimate
	\[
	\int  \chi_{k,k_1,k_2} \fai_{l}(\Phi(\xi,\xi-r\theta)) \dif \theta.
	\]
The bound	$\abs{\Phi} \lle 2^l$ implies that
	\[
	\abs{-\lambda(\abs{\xi})+\iota_1\lambda(r)+\iota_2\lambda(\abs{\xi-r\theta})} \lle 2^l.
	\]
	It follows that	$\abs{\xi-r\theta}$ lies in $B(r_0,2^l2^{\delta^2m})$. By the cosine formula,
	\[
	\theta \cdot \mao \xi=\frac{\abs{\xi}^2+r^2-\abs{\xi-r\theta}^2}{2\abs{\xi}r}.
	\]
	Hence $\theta \cdot \mao \xi$ lies in $B(\theta_0,2^l2^{4\delta^2m})$. \lemmaref{cos} implies that the range of $\theta$ has length $\lle 2^{l/2}2^{2\delta^2 m}$. Therefore,
	\[
	\int  \chi_{k,k_1,k_2} \fai_{l}(\Phi(\xi,\xi-r\theta)) \dif \theta \lle 2^{l/2}2^{2\delta^2 m}.
	\]
	The desired bound for the first integral then follows from the Cauchy\allowbreak--\allowbreak Schwarz inequality. The second integral can be handled in a similar way.
	\item To apply Schur's test, we aim to estimate
	\begin{gather*}
		\sup_\xi \int \chi_{k,k_1,k_2} \fai_{l}(\Phi(\xi,\eta)) \fai_{-n_1}(\abs{\xi-\eta}-r_1) \fai_{-n_2}(\abs{\eta}-r_2) \abs{ f_1(\xi-\eta)} \dif \eta,\\
		\sup_\eta \int \chi_{k,k_1,k_2} \fai_{l}(\Phi(\xi,\eta)) \fai_{-n_1}(\abs{\xi-\eta}-r_1) \fai_{-n_2}(\abs{\eta}-r_2) \abs{ f_1(\xi-\eta)} \dif \xi.
	\end{gather*}
	For the first integral, as in (i), the key step is to estimate, after the substitution $\xi-\eta=r\theta$,
	\[
	\fai_{-n_1}(r-r_1)\int \chi_{k,k_1,k_2} \fai_l(\Phi(\xi,\xi-r\theta)) \fai_{-n_2}(\abs{\xi-r\theta}-r_2) \dif \theta.
	\]
	We next show that
	\[
	\abs{\theta \cdot \mao{\xi}^{\perp}} \gge 2^{-\delta^2 m/2}.
	\]
	By the cosine formula, it suffices to prove that
	\[
	1 \pm \frac{r^2+\abs{\xi-r\theta}^2-\abs{\xi}^2}{2r\abs{\xi-r\theta}} \gge 2^{-\delta^2 m}.
	\]
	Hence we need to bound
	\[
	\abs{r \pm \abs{\xi-r\theta}}-\abs{\xi},
	\]
	which is equal to $\abs{r_1 \pm r_2}-\abs{\xi}+O(2^{-\delta^2m})$. Since $\abs{\Phi} \lle 2^l$, we have
	\[
	\abs{-\lambda(\abs \xi)+\iota_1\lambda(r)+\iota_2\lambda(\abs{\xi-r\theta})} \lle 2^l,
	\]
	which yields
	\[
	\abs{-\lambda(\abs \xi)+\iota_1\lambda(r_1)+\iota_2\lambda(r_2)} \lle 2^{-\delta^2m}.
	\]
	On the other hand,
	\[
	\abs{-\lambda(\abs{r_1 \pm r_2})+\iota_1\lambda(r_1)+\iota_2\lambda(r_2)} \ge 2^{-100}.
	\]
	It follows that
	\[
	\abs{\lambda(\abs{r_1 \pm r_2})-\lambda(\abs \xi)} \gge 2^{-100},
	\]
	which gives 
	\[
	\abs{\abs{r_1 \pm r_2}-\abs{\xi}} \gge 2^{-100}.
	\]
	Therefore, we obtain that $\abs{\theta \cdot \mao{\xi}^\perp} \gge 2^{-\delta^2 m/2}$.
	
	Because of the restrictions $\fai_l(\Phi)$ and $\fai_{-n_2}(\abs{\xi-r\theta}-r_2)$, $\abs{\xi-r\theta}$ lies in an interval of length $\lle 2^l \wedge 2^{-n_2}$. As in (i), \lemmaref{size} shows that the range of $\theta$ has length $\lle (2^l \wedge 2^{-n_2})2^{\delta^2 m/2}$. The first integral is then estimated as in (i). The second integral can be treated similarly.
	\item We only prove \eqref{eq:2.7.1} as \eqref{eq:2.7.2} is similar. Let $\eta=r\theta$. Note that
	\[
	(\ppd r\Phi(\xi,r\theta))^2+(r^{-1}\ppd \theta\Phi(\xi,r\theta))^2=\abs{\nabla_\eta \Phi(\xi,r\theta)}^2,
	\]
	and
	\[
	\Omega_\eta \Phi(\xi,r\theta)=\ppd \theta \Phi(\xi,r\theta).
	\]
	Hence
	\[
	\abs{\ppd r \Phi(\xi,r\theta)} \gge 2^{p_1}. 
	\]
	In polar coordinates, we aim to estimate
	\[
	\int \Brac{ \int\chi_{k,k_1,k_2} \fai_{l}(\Phi(\xi,r\theta)) \fai_{\log_2 \kappa_\theta}(\Theta(\xi,r\theta))\fai_{\ge p_1}(\nabla_\eta \Phi(\xi,r\theta)) r\dif r} \dif \theta.
	\]
	By \lemmaref{size} and the fact that $\Phi(\xi,r\theta)$ has at most $20$ critical points if $\xi$ and $\theta$ are fixed, we have
	\[
	\int\chi_{k,k_1,k_2} \fai_{l}(\Phi(\xi,r\theta)) \fai_{\log_2 \kappa_\theta}(\Theta(\xi,r\theta))\fai_{\ge p_1}(\nabla_\eta \Phi(\xi,r\theta)) r\dif r \lle 2^{\delta^2 m} 2^{l-p_1}.
	\]
	The conclusion follows immediately.
	\item
	As \eqref{eq:2.7.5} and \eqref{eq:2.7.6} are similar, we only prove \eqref{eq:2.7.5}. In polar coordinates, we write
	\[
	\int \chi_{k,k_1,k_2} \fai_l(\Phi(\xi,\eta)) \dif \eta
	\stackrel{\eta=r\theta}{=} \int \Brac{\int \chi_{k,k_1,k_2} \fai_l(\Phi(\xi,r\theta)) \dif \theta} r\dif r.
	\]
	For $\varrho_1, \varrho_2 \geq 0$, set
	\[
	R(\varrho_1,\varrho_2)=\min \set{ \abs{-\lambda(\varrho_1)+\iota_1\lambda(\varrho_1 + \varrho_2)+\iota_2 \lambda(\varrho_2)}, \abs{-\lambda(\varrho_1)+\iota_1\lambda(\abs{\varrho_1 - \varrho_2})+\iota_2 \lambda(\varrho_2)}}.
	\]
	We claim that
	\begin{equation}\label{eq:2.7.10}
		\int \chi_{k,k_1,k_2} \fai_l(\Phi(\xi,r\theta)) \dif \theta \lle \frac{2^l2^{8\delta^2m}}{2^{l/2}+\sqrt{R(\varrho,r)}},
	\end{equation}
	where $\varrho=\abs{\xi}$.
	Indeed, when $2^l \ll R(\varrho,r)$, as in (ii), the restriction $\fai_l(\Phi(\xi,r\theta))$ implies that $\abs{\xi-r\theta}$ lies in an interval of length $\lle 2^l 2^{\delta^2 m}$, and $\abs{\Phi} \lle 2^l$ and $2^l \ll R(\varrho,r)$ yield $\abs{\abs{\varrho \pm r}-\abs{\xi-r\theta}} \gge 2^{-\delta^2 m}R(\varrho,r)$. Therefore, it follows from the cosine formula that $\theta \cdot \mao \xi$ lies in an interval of length $\lle 2^l2^{4\delta^2 m}$ and $\abs{\theta \cdot \mao{\xi}^\perp} \gge 2^{-2\delta^2 m}\sqrt{R(\varrho,r)}$. \lemmaref{size} implies that the range of $\theta$ has length $\lle 2^l2^{6\delta^2 m}/\sqrt{R(\varrho,r)}$. When $2^l \gge R(\varrho,r)$, a similar argument shows that $\abs{\xi-r\theta}$ lies in an interval of length $\lle 2^l 2^{\delta^2 m}$ and $\theta \cdot \mao \xi$ lies in an interval of length $\lle 2^l2^{4\delta^2 m}$. Then it follows from \lemmaref{cos} that the range of $\theta$ has length $\lle 2^{l/2}2^{2\delta^2 m}$. Thus, \eqref{eq:2.7.10} is proved.
	
	We next estimate
	\[
	\int \frac{2^l2^{8\delta^2m}}{2^{l/2}+\sqrt{R(\varrho_1,r)}} r \dif r.
	\]
	It suffices to bound
	\[
	2^{9\delta^2 m}\sum_{p \colon p \gge l} \int_{R(\varrho_1,r) \simeq 2^p} 2^l2^{-p/2} \dif r
	\]
	and
	\[
	2^{l/2}2^{9\delta^2 m}\int_{R(\varrho_1,r) \lle 2^l} \dif r,
	\]
	which reduces to estimating
	\[
	\abs{\set{r \colon R(\varrho_1,r) \lle 2^p}}.
	\]
	Applying \lemmaref{size} if $\iota_1\iota_2=-$ and \lemmaref{second} if $\iota_1\iota_2=+$, we have
	\[
	\abs{\set{r \colon R(\varrho_1,r) \lle 2^p}} \lle 2^p 2^{\delta^2 m} \vee 2^{p/2}.
	\]
	Therefore, the desired estimate follows.
\end{enumerate}
\end{proof}
An argument similar to that in \lemmaref{sizeest}(iv) yields the following estimates.
\begin{lemma}\label{sizeangle}
Let $\chi_{k,k_1,k_2}(\xi,\eta)$ be the cutoff function $\psi_k(\xi)\psi_{k_1}(\xi-\eta)\psi_{k_2}(\eta)$.
\begin{enumerate}[label=(\roman*)]
	\item
\begin{gather*}
	\sup_{\xi,r} \int \chi_{k,k_1,k_2}(\xi,\xi-r\theta) \fai_l(\Phi(\xi,\xi-r\theta)) \dif \theta \lle 2^{(l-\cl k-\ul k)/2},\\
	\sup_{\eta,r} \int \chi_{k,k_1,k_2}(\eta+r\theta,\eta) \fai_l(\Phi(\eta+r\theta,\eta)) \dif \theta \lle 2^{(l-\cl k-\ul k)/2}.
\end{gather*}
\item
If $k_2 \ge k_1+5$ and $-2 \le k_1 \le 2$, then
\begin{gather*}
	\sup_{\xi,r} \int \chi_{k,k_1,k_2}(\xi,\xi-r\theta) \fai_l(\Phi(\xi,\xi-r\theta)) \dif \theta \lle 2^{l-\cl k},\\
	\sup_{\eta,r} \int \chi_{k,k_1,k_2}(\eta+r\theta,\eta) \fai_l(\Phi(\eta+r\theta,\eta)) \dif \theta \lle 2^{l-\cl k}.
\end{gather*}
\end{enumerate}
\end{lemma}

\subsection{Symbol Bounds}
In this subsection, we establish symbol bounds for cutoff functions of modulations, typically used when the modulation is not too small. These bounds and the averaging argument will serve as our main tools for handling cutoff functions of modulations. For the proof, see \cite[Section~2.7]{arXiv:2510.18655}. Assume that $\chi$ is a smooth function supported in $B(0,2)$, and that $\nabla \chi$ is supported in $A(0;1,2)$. For $k \in \bZ$, define $\chi_k(x) :=\chi(x/2^k)$.
\begin{lemma}\label{symbound}
	Suppose that the frequencies $-\xi$, $\xi-\eta$, and $\eta$ correspond, in some order, to localization parameters $k_1$, $k_2$, and $k_3$, with $k_1 \ge k_2 \ge k_3$. Moreover, assume that $\lambda \gge 1$.
	\begin{enumerate}[label=(\roman*)]
		\item If $2^l \gge 2^{k_1}\lambda^{-1}$, then $\norm{\chi_l(\Phi(\xi,\eta))}_{S_{k_1,k_2,k_3}} \lle 2^{20(k_3)_+}\lambda^5$.
		\item If $2^l \gge 2^{k_1+k_3}\lambda^{-1}$, then $\norm{\chi_l(\Phi(\xi,\eta))}_{S_{k_1,k_2,k_3}} \lle \lambda^5$.
		\item If $2^l \gge \lambda^{-1}$, then $\norm{\chi_l(\Phi(\xi,\eta))}_{S_{k_1,k_2,k_3}} \lle 2^{20(k_1)_+}\lambda^5$.
	\end{enumerate}
\end{lemma}
We have the following variants of \lemmaref{symbound}(i) and (ii) obtained by a change of variables.
\begin{lemma}\label{symbound1}
Suppose that $\abs{k-k_2} \le 2$, $k_2 \ge k_1+8$, $\tau \in[0,1]$, and $\lambda \gge 1$.
\begin{enumerate}[label=(\roman*)]
	\item If $2^l \gge 2^{k_2}\lambda^{-1}$, then
	\[
	\norm{\chi_l(\Phi(\xi-\tau\sigma,\eta+(1-\tau)\sigma))\psi_{k}(\xi-\tau\sigma)\psi_{k_2}(\eta+(1-\tau)\sigma)\psi_{k_1}(\xi-\eta-\sigma)\fai_{k_2-100}(\sigma)}_S \lle 2^{20(k_1)_+}\lambda^5.
	\]
	\item If $2^l \gge 2^{k_1+k_2}\lambda^{-1}$, then
	\[
	\norm{\chi_l(\Phi(\xi-\tau\sigma,\eta+(1-\tau)\sigma))\psi_{k}(\xi-\tau\sigma)\psi_{k_2}(\eta+(1-\tau)\sigma)\psi_{k_1}(\xi-\eta-\sigma)\fai_{k_2-100}(\sigma)}_S \lle \lambda^5.
	\]
\end{enumerate}
\end{lemma}
For the second and third resonance functions, we have the following bounds.
\begin{lemma}\label{symbound2}
		Suppose that the frequencies $-\xi$, $\xi-\eta$, $\eta-\sigma$, and $\sigma$ correspond, in some order, to localization parameters $k_1$, $k_2$, $k_3$, and $k_4$, with $k_1 \ge k_2 \ge k_3 \ge k_4$, and that $k_1 \le 20\delta^3 m$. Moreover, assume that $\lambda \gge 1$. Then for the second resonance function $\wan \Phi$, if $2^p \gge \lambda^{-1}$, then
		\[
		\norm{\chi_p(\wan \Phi(\xi,\eta,\sigma))}_{S_{k_1,k_2,k_3,k_4}} \lle 2^{1000\delta^3 m}\lambda^7.
		\]
\end{lemma}
\begin{lemma}\label{symbound3}
		Suppose that the frequencies $-\xi$, $\xi-\eta$, $\eta-\sigma$, $\sigma-\zeta$, and $\zeta$ correspond, in some order, to localization parameters $k_1$, $k_2$, $k_3$, $k_4$, and $k_5$, with $k_1 \ge k_2 \ge k_3 \ge k_4 \ge k_5$, and that $k_1 \le 20\delta^3 m$. Moreover, assume that $\lambda \gge 1$. Then for the third resonance function $\cl \Phi$, if $2^q \gge \lambda^{-1}$, then
		\[
		\norm{\chi_q(\cl \Phi(\xi,\eta,\sigma,\zeta))}_{S_{k_1,k_2,k_3,k_4,k_5}} \lle 2^{1000\delta^3 m}\lambda^7.
		\]
\end{lemma}

\subsection{Technical Lemmas}\label{techlemma}
We first establish a localized multiplier theorem for pseudo-differential multilinear operators. The treatment is similar to that in \cite[Section~3.1]{MR3665671}. Define
\[
\FF L_a(f_1,f_2)(\xi)=\int \mao a(\xi-\eta,\xi,\eta,\sigma) \mao{f_1}(\eta-\sigma) \mao{f_2}(\sigma) \dif \eta \dif \sigma,
\]
($a=a(x,\xi,\eta,\sigma)$ and $\mao a$ denotes the Fourier transform of the first variable), and set
\[
\norm{a}_{S'}=\norm{\FF^{-1}_{\xi,\eta,\sigma}a(x,y_1,y_2,y_3)}_{L^1_yL^\infty_x}.
\]
It follows immediately that
\[
\norm{ab}_{S'} \lle \norm{a}_{S'}\norm{b}_{S'}.
\]
\begin{lemma}\label{pseudomultithm}
Assume that $p,q,r \in [1,\infty]$ and $r^{-1}=p^{-1}+q^{-1}$. Then
\[
\norm{L_a(f_1,f_2)}_r \lle \norm{a}_{S'}\norm{f_1}_p \norm{f_2}_q.
\]
\end{lemma}
Moreover, we use the following lemma to estimate the $S'$-norm.
\begin{lemma}\label{pseudomulti}
	Let $a(x,\xi_1,\xi_2,\xi_3)$ be a symbol defined on $(\bR^2)^4$. If $a(x,\cdot)$ is supported in $B(0,2^{k_1}) \times B(0,2^{k_2}) \times B(0,2^{k_3})$ and $2^{|\alpha|k_i}\abs{\ppd {\xi_i}^\alpha a} \le M$ for $0 \le \abs \alpha \le 10$ and $1 \le i \le 3$, then $\norm{a}_{S'} \lle M$.
\end{lemma}

We next introduce a lemma, slightly different from that in \cite[Section~3.2]{MR3665671}, to handle the cases where all rotational vector fields act on low frequencies in the paralinearization. This lemma can be proved in a manner similar to that in \cite[Section~3.2]{MR3665671}. For $m,l \geq 0$ and $p \geq 1$, let
\[
\norm{a}^l_{m,p}=\sup_{\xi,\abs \alpha \le l}\spa{\xi}^{-m}\abs{\xi}^{\abs{\alpha}}\norm{D^{\alpha}_{\xi} a(x,\xi)}_{L^p_x}.
\]
\begin{lemma}\label{extrderisym}
	Assume that $r^{-1}=p^{-1}+q^{-1}$, $p,q,r \in [1,\infty]$, and $s \geq 0$. Then
	\[
	\norm{P_kT_af}_r \lle 2^{(m+s)k_+}\norm{\spa{D}^{-s}a}^{10}_{m,p}\norm{P_{[k-2,k+2]}f}_q.
	\]
\end{lemma}

We next introduce the averaging argument. A basic version is stated in the following lemma.
\begin{lemma}
	Let $\chi$ be a Schwartz function. Suppose that $l \ge -m+2\delta^2m$, $t \simeq 2^m$, $m \ge D$, and $2^{-1}=p^{-1}+q^{-1}$ with $p,q \in [1,\infty]$. Define $I(f,g)$ by
	\[
	\FF I(f,g)(\xi)=\fai_{\le 10 m}(\xi) \int \me^{\mi t \Phi_{\iota_1\iota_2}}\chi(\Phi_{\iota_1\iota_2}/2^l)\gti m(\xi, \eta) \mao{f}(\xi-\eta) \mao{g}(\eta) \dif \eta. 
	\]
	Then
	\[
	\norm{I(f,g)}_2 \lle \norm{\gti m}_S (\sup_{t_1 \in [t/2,2t],t_2 \in [t/2,2t]} \norm{E_{\iota_1 t_1}f}_p\norm{E_{\iota_2 t_2}g}_q+2^{-10m}\norm{f}_2\norm{g}_2).
	\]
Later we write $\norm{E_{[t/2,2t]}f}_p$ for $\sup_{s \in [t/2,2t]} \norm{E_{s}f}_p$.
\end{lemma}
	For a detailed proof, see \cite[Lemma~7.4]{MR3784694}. The key point is to write
	\[
	\FF T(f,g)(\xi)=\fai_{\le 10 m}(\xi) \int \mao{\chi}(\theta) \Brac{\int \me^{\mi (t+\theta/2^l) \Phi}\gti m(\xi,\eta) \mao{f}(\xi-\eta) \mao{g}(\eta) \dif \eta} \dif \theta,
	\]
	and consider two cases $\abs \theta \le 2^{\delta^2 m}$ and $\abs \theta >2^{\delta^2 m}$. The averaging argument can be adapted to deal with various operators involving localizations of modulation. It is worth noting that, in addition to multiplier theorems, the averaging argument can also be applied in the approximate-finite-speed-of-propagation argument and in integration by parts with respect to frequencies, especially in cases where the modulation is small. For multilinear operators, one additional restriction needs to be taken into account. Consider
	\begin{align*}
		\FF I(f_1,f_2,f_3)(\xi)={}&\int_{\xi_1+\xi_2+\xi_3=\xi} \me^{\mi s \wan \Phi_{+++}} \chi_1(\Phi^{1,23}_{+++}/2^l) \chi_2(\wan \Phi_{+++}/2^p)\mao{f_1}(\xi_1) \mao{f_2}(\xi_2) \mao{f_3}(\xi_3)\\
		&\times \dif \xi_1 \dif \xi_2 \dif \xi_3,
	\end{align*}
	where $\chi_1$ and $\chi_2$ are Schwartz functions (see Section~\ref{notations} for the definition of $\Phi^{1,23}_{+++}$). We then have
	\begin{align*}
		\FF I(f_1,f_2,f_3)(\xi)={}&\int \mao{\chi_1}(\theta_1) \mao{\chi_2}(\theta_2) \Big( \int_{\xi_1+\xi_2+\xi_3=\xi} \me^{\mi s \wan \Phi+\mi 2^{-l}\theta_1 \Phi^{1,23}_{+++}+\mi 2^{-p} \theta_2 \wan \Phi}\\
	 &\times\mao{f_1}(\xi_1) \mao{f_2}(\xi_2) \mao{f_3}(\xi_3) \dif \xi_1 \dif \xi_2 \dif \xi_3 \Big) \dif \theta_1 \dif \theta_2.
	\end{align*}
	\[
	\FF^{-1} \int_{\xi_1+\xi_2+\xi_3=\xi} \me^{\mi s \wan \Phi+\mi 2^{-l}\theta_1\Phi^{1,23}_{+++}+\mi 2^{-p} \theta_2 \wan \Phi} \mao{f_1}(\xi_1) \mao{f_2}(\xi_2) \mao{f_3}(\xi_3) \dif \xi_1 \dif \xi_2 \dif \xi_3
	\]
	can be written in a flag form:
	\[
	E_{-s-2^{-l}\theta_1-2^{-p}\theta_2}T_1(E_{s+2^{-l}\theta_1+2^{-p}\theta_2}f_1,E_{2^{-l}\theta_1}T_1(E_{s+2^{-p}\theta_2}f_2,E_{s+2^{-p}\theta_2}f_3)).
	\]
	To handle the deviation operator $E_{2^{-l}\theta_1}$, one is led to estimate it by
	\[
	\norm{E_{s+2^{-l}\theta_1+2^{-p}\theta_2}f_1}_\infty \norm{T_1(E_{s+2^{-p}\theta_2}f_2,E_{s+2^{-p}\theta_2}f_3)}_2,
	\]
	and this will introduce an additional restriction on the H\"older exponents. To employ the averaging argument, we need to write multilinear operators in a flag form. If the operator has a non-trivial multiplier, then a different averaging argument is required. For instance, for
	\begin{align*}
		\FF I(f_1,f_2,f_3)(\xi)={}&\int_{\xi_1+\xi_2+\xi_3=\xi} \me^{\mi s \wan \Phi_{+++}} \gti m(\xi_1,\xi_2,\xi_3)\chi_1(\Phi^{1,23}_{+++}/2^l) \chi_2(\wan \Phi_{+++}/2^p)\\ &\times\mao{f_1}(\xi_1)\mao{f_2}(\xi_2) \mao{f_3}(\xi_3) \dif \xi_1 \dif \xi_2 \dif \xi_3,
	\end{align*}
	where $\chi_1$ and $\chi_2$ are Schwartz functions, we can write $\gti m(\xi_1,\xi_2,\xi_3)$ as
	\[
	\int \check {\gti m}(y_1,y_2,y_3) \me^{-\mi (y_1 \cdot \xi_1+y_2 \cdot \xi_2+y_3 \cdot \xi_3)} \dif y_1 \dif y_2 \dif y_3,
	\]
	so that the resulting operator can again be written in a flag form.

We next present some basic properties of the $Z$-norm.
\begin{lemma}\label{zproperty}
Suppose that $\chi$ is a Schwartz function, and for any $k \in \bZ$, define $\chi_k(x):=\chi(x/2^k)$.
\begin{enumerate}[label=(\roman*)]
	\item If $k \ge 10$ or $k \le -10$, then $\norm{Q^{(k)}_jP_kf}_2 \lle 2^{-(1-20\delta)(j+k)}2^{-\delta k}2^{-10k_+} \norm{f}_Z$.
	\item For $\gti m \in \gti S^0$, $\norm{\gti m(D)f}_Z \lle \norm{f}_Z$.
	\item $\norm{A^{(j)}_n\chi_j(x)Q^{(k)}_jP_kf}_2 \lle  2^{-(1-20\delta)(j+k)}2^{(1/2-18\delta)n}2^{-\delta k}2^{-10k_+}\norm{f}_Z$.
	\item $\sup_{0 \le n \le j+1} 2^{-\alpha n}\norm{A^{(j)}_n\chi_j(x)f}_2 \lle \sup_{0 \le n \le j+1} 2^{-\alpha n}\norm{A^{(j)}_n f}_2$ for any $\alpha>0$.
\end{enumerate}
\end{lemma}
\begin{proof}
\mbox{}
\begin{enumerate}[label=(\roman*)]
	\item Assume that $\norm{f}_Z=1$. We have
\[
\norm{Q^{(k)}_jP_kf}_2 \lle \sum_n \norm{A^{(j)}_nQ^{(k)}_jP_kf}_2.
\]
It suffices to consider the case $n>0$. We have
\[
\FF (A^{(j)}_n Q^{(k)}_j P_k f)(\xi)=\wan{\fai}^{-j-1,0}_{-n}(2^{100}(\abs\xi-2\gamma_0))\int 2^{2j}\mao{\chi_0}(2^j(\xi-\eta)) \psi_k(\eta) \mao f(\eta) \dif \eta
\]
for some Schwartz function $\chi_0$. The separation between the frequencies $\xi$ and $\eta$ leads to the estimates
\[
\norm{A^{(j)}_n Q^{(k)}_j P_k f}_2 \lle
\begin{cases}
	2^{-100j}\norm{P_kf}_2 & \text{if } k \le -10;\\
	2^{-100(j+k)}\norm{P_kf}_2 & \text{if } k \ge 10.
\end{cases}
\]
(Such an argument will be used repeatedly in this lemma.) The conclusion then follows immediately.
\item Assume that $\norm{f}_Z=1$. It suffices to prove that
\[
2^{10k_+}2^{\delta k}2^{(1-20\delta)(j+k)}2^{-(1/2-18\delta)n}\norm{A^{(j)}_nQ^{(k)}_j\wan{P}_k P_k f}_2 \lle 1,
\]
where $\wan{P}_k$ is another frequency localization operator satisfying $\wan{P}_k=\wan{\psi}(D/2^k)$ with $\abs{\FF\wan \psi(\xi)} \lle \spa{\xi}^{-200}$. We decompose $A^{(j)}_nQ^{(k)}_j\wan{P}_kP_k f$ as
\[
\sum_{j' \colon j' \ge k_-} A^{(j)}_nQ^{(k)}_j\wan{P}_kQ^{(k)}_{j'}P_k f.
\]
When $k \ge 10$ or $k \le -10$, for the terms with $j' \gge j$,
\begin{align*}
&2^{10k_+}2^{\delta k}2^{(1-20\delta)(j+k)}2^{-(1/2-18\delta)n}\sum_{j' \colon j' \gge j} \norm{A^{(j)}_nQ^{(k)}_j\wan{P}_kQ^{(k)}_{j'}P_k f}_2 \\
\lle{} &2^{10k_+}2^{\delta k}2^{(1-20\delta)(j+k)}2^{-(1/2-18\delta)n}\sum_{j' \colon j' \gge j} \norm{Q^{(k)}_{j'}P_k f}_2\\
\lle{} &\sum_{j' \colon j' \gge j}2^{(1-20\delta)j}2^{-(1/2-18\delta)n}2^{-(1-20\delta)j'},
\end{align*}
while for the terms with $j' \ll j$, we write $\wan{P}_k$ as a convolution operator and the separation between $Q_j^{(k)}$ and $Q_{j'}^{(k)}$ yields the estimate
\begin{align*}
	&2^{10k_+}2^{\delta k}2^{(1-20\delta)(j+k)}2^{-(1/2-18\delta)n}\sum_{j' \colon j' \ll j} \norm{A^{(j)}_nQ^{(k)}_j\wan{P}_kQ^{(k)}_{j'}P_k f}_2 \\
	\lle{} &2^{10k_+}2^{\delta k}2^{(1-20\delta)(j+k)}2^{-(1/2-18\delta)n}\sum_{j' \colon j' \ll j} 2^{-100(j+k)}\norm{Q^{(k)}_{j'}P_k f}_2 \\
	\lle{} &\sum_{j' \colon j' \ll j}2^{(1-20\delta)j}2^{-(1/2-18\delta)n}2^{-100(j+k)}2^{-(1-20\delta)j'}.
\end{align*}
When $-10<k<10$, we also decompose $A^{(j)}_nQ^{(k)}_j\wan{P}_kP_k f$ as
\[
\sum_{j' \colon j' \ge k_-} A^{(j)}_nQ^{(k)}_j\wan{P}_kQ^{(k)}_{j'}P_k f.
\]
For the terms with $j' \ll j$, as in the cases $k \ge 10$ or $k \le -10$, we exploit the separation between $Q^{(k)}_j$ and $Q^{(k)}_{j'}$, despite the relatively weak bound $\norm{Q^{(k)}_{j'}P_k f}_2 \lle 2^{-(1/2-2\delta)j'}$ in these cases. For the terms with $j' \gge j$, we further decompose $\sum_{j' \colon j' \gge j} A^{(j)}_nQ^{(k)}_j\wan{P}_kQ^{(k)}_{j'}P_k f$ as
\[
\sum_{j',n' \colon j' \gge j,0 \le n' \le j'+1} A^{(j)}_nQ^{(k)}_j\wan{P}_kA^{(j')}_{n'}Q^{(k)}_{j'}P_k f.
\]
For the terms with $n' \gg n$, we regard $Q^{(k)}_j$ as a convolution operator, and the separation between $A^{(j)}_n$ and $A^{(j')}_{n'}$ leads to the estimate
\begin{align*}
	&2^{(1-20\delta)j}2^{-(1/2-18\delta)n}\sum_{j',n' \colon j' \gge j,n' \gg n} \norm{A^{(j)}_nQ^{(k)}_j\wan{P}_kA^{(j')}_{n'}Q^{(k)}_{j'}P_k f}_2 \\
\lle{} &2^{(1-20\delta)j}2^{-(1/2-18\delta)n}\sum_{j',n' \colon j' \gge j,n' \gg n} 2^{-100(j-n)}\norm{A^{(j')}_{n'}Q^{(k)}_{j'}P_k f}_2\\
\lle{} &\sum_{j',n' \colon j' \gge j,n' \lle j'}2^{(1-20\delta)j}2^{-(1/2-18\delta)n}2^{-100(j-n)}2^{(1/2-18\delta)n'}2^{-(1-20\delta)j'}.
\end{align*}
For the terms with $n' \lle n$, we have
\begin{align*}
	&2^{(1-20\delta)j}2^{-(1/2-18\delta)n}\sum_{j',n' \colon j' \gge j,n' \lle n} \norm{A^{(j)}_nQ^{(k)}_j\wan{P}_kA^{(j')}_{n'}Q^{(k)}_{j'}P_k f}_2 \\
\lle{} &2^{(1-20\delta)j}2^{-(1/2-18\delta)n}\sum_{j',n' \colon j' \gge j,n' \lle n} \norm{A^{(j')}_{n'}Q^{(k)}_{j'}P_k f}_2\\
\lle{} &\sum_{j',n' \colon j' \gge j,n' \lle n}2^{(1-20\delta)j}2^{-(1/2-18\delta)n}2^{(1/2-18\delta)n'}2^{-(1-20\delta)j'}.
\end{align*}
\item We decompose $A^{(j)}_n\chi_j(x)Q^{(k)}_jP_kf$ as $\sum_{n' \colon 0 \le n' \le j+1} A^{(j)}_n\chi_j(x)A^{(j)}_{n'}Q^{(k)}_jP_kf$. We estimate the terms with $n' \lle n$ directly. For the terms with $n' \gg n$, as in (ii), we exploit the separation between $A^{(j)}_{n'}$ and $A^{(j)}_n$.
\item We decompose $A^{(j)}_n\chi_j(x)f$ as $\sum_{n' \colon 0 \le n' \le j+1} A^{(j)}_n\chi_j(x)A^{(j)}_{n'}f$, and the desired estimate can be proved in a way similar to that in (iii).
\end{enumerate}
\end{proof}

To handle the remainders in the dispersive estimates, we need the following properties of the simpler $Z$-norm:
\[
\norm{f}_{\cl Z^s}:=\sup_{j,k \colon j \ge k_-} 2^{sj}\norm{Q_{jk}f}_2,
\]
where $0<s<1$. 
\begin{lemma}\label{simplez}
\mbox{}
	\begin{enumerate}[label=(\roman*)]
		\item Suppose that $\gti m \in \gti S^0$. Then $\norm{\gti m(D)f}_{\cl Z^s} \lle \norm{f}_{\cl Z^s}$.
		\item $\norm{f}_{\cl Z^s} \lle \norm{\spa{x}^{s} f}_2 \lle \sum_{j,k \colon j \ge k_-} 2^{sj}\norm{Q_{jk}f}_2$.
		\item $\norm{\me^{\mi t \Lambda(D)}f}_{\cl Z^s} \lle t^s \norm{\Lambda^s f}_2+\norm{f}_{\cl Z^s}$.
		\item $\norm{{\spa x}^\alpha T_{\gti m}(f,g)}_2 \lle \norm{ \check{\gti m}}_1 \norm{\spa x^{\alpha}f}_2 \norm{g}_\infty+\norm{ \spa{x}^{\alpha}\check{\gti m}}_1 \norm{f}_p \norm{g}_q$,
		where $\alpha \ge 0$, $1 \le p, q \le \infty$, and $1/p+1/q=1/2$.
		\end{enumerate}
\end{lemma}
For the proof, see \cite[Lemma~2.19]{arXiv:2510.18655}.
The following lemma will be used in the remainder estimates in Section~\ref{sec:remainder}.
\begin{lemma}\label{weightedinfty}
\[
	\norm{\spa x^{\alpha}\me^{\mi t \Lambda(D)} \Omega^{[0,N_2]} V_k(s)}_\infty \lle \epsl_1 \spa{t}^{\alpha-1+50\delta}2^{-5k_+}2^{-(1-\delta) k},
\]
where $(1/10) \spa t \le \spa s \le  10\spa t$ and $0<\alpha<1-20\delta$.
\end{lemma}
\begin{proof}
	We assume that $\spa x \simeq 2^j$ and $\spa t \simeq 2^m$, where $j, m \ge 0$. The case $j \lle (m+k) \vee 0$ is easy to handle using \eqref{decayestb} and \eqref{decayestbr}. Hence we can assume that $j \gg (m+k) \vee 0$. We decompose $\me^{\mi t \Lambda(D)} V_k$ into $\me^{\mi t \Lambda(D)} V_{k,j',n'}$. If $j' \gge j$, then the estimate follows from \eqref{zest}. If $j' \ll j$, then integration by parts yields the desired estimate.
\end{proof}

\subsection{Time Derivative Estimates}
Assume that $t \in [2^{m-1},2^{m+1}]$ and $m \ge 10$.
\begin{lemma}\label{timederien}
	For $W \in \gti W^+$,
	\begin{equation}\label{eq:paratimederi1}
	(\ppd t-\mi \Lambda(D))W=T_{\mi \Sigma_{\ge 1}}W-T_{R\psi \cdot \mi \xi}W+T_{\gti m'}(G,W^1)+R,		
	\end{equation}
	where $\gti m'$ and $G$ satisfy the same conditions as in \eqref{energymain1},
	\[
	W^1 \in \gti W'=\set{T_{\Sigma^{N_0/2}} \hti U^{\pm},T_{\Sigma^{\delta/2}}\Omega^{N_1} \hti U^{\pm}, \spa{\nabla}^{300} \Omega^{[0,N_1-1]} U^{\pm}},
	\]
	and
	\[
	\norm{P_{k} R(t)}_2 \lle \epsl_1 2^{-(1-1/100)m} (\sum_{W \in \gti W}\norm{P_kW(t)}_2+\epsl_1 2^{-k})
	\]
	for $k \ge -100$. (The Taylor expansion of $\Sigma$ with respect to $\Lambda \rho$ is $\Sigma=\Sigma_0+\Sigma_1+\Sigma_{\ge 2}$; note that $\Sigma_1$ does not lose derivatives.) Moreover,
	\begin{equation}\label{eq:paratimederi2}
	\norm{P_{<0}(\ppd t-\mi \Lambda(D))W(t)}_2 \lle \epsl_1^2 2^{-m/3}.
	\end{equation}
\end{lemma}
\begin{proof}
	\eqref{eq:paratimederi1} is an immediate consequence of \eqref{energymain1} and \eqref{dispereq}. Let
	\[
	\hti E \in \set{T_{\Sigma^{N_0/2}}\hti U-T_{\Sigma_0^{N_0/2}}U,T_{\Sigma^{\delta/2}}\Omega^{N_1}\hti U-T_{\Sigma_0^{\delta/2}}\Omega^{N_1}U}.
	\]
	It then suffices to estimate $P_{<0}(\ppd t-\mi \Lambda(D))\hti E(t)$. Since there is no loss of derivative, we can handle $P_{<0}\ppd t \hti E(t)$ and $P_{<0} \Lambda(D) \hti E(t)$ separately. The desired estimate follows from a routine but lengthy verification.
\end{proof}
\begin{lemma}\label{timederil2}
	\[
	\norm{\ppd t \Omega^{[0,N_2]} V_k(t)}_2 \lle \varepsilon_1^2 2^{-(1-8\delta)m}.
	\]
	Moreover,
	\[
	\norm{I_{\gti m_q,m}(\va(t),\vb(t))}_2 \lle \varepsilon_1^2 2^{-(1-8\delta)m},
	\]
	where $0 \le p_1, p_2 \le N_2$.
\end{lemma}
\begin{proof}
	This follows directly from \cororef{decaydecom}.
\end{proof}
\begin{lemma}\label{timederil2r}
\[
\norm{\ppd t \Omega^{[0,N_2]}V_k(t)}_{L^2(r \dif r)L^\infty_\theta} \lle \varepsilon_1^2 2^{-(1-8\delta)m}.
\]
	Moreover,
	\[
	\norm{I_{\gti m_q,m}(\va(t),\vb(t))}_{L^2(r \dif r)L^\infty_\theta} \lle \varepsilon_1^2 2^{-(1-8\delta)m},
	\]
	where $0 \le p_1, p_2 \le N_2$.
	\end{lemma}
\begin{proof}
	It follows from the following inequality:
	\[
	\norm{f}_{L^2(r \dif r)L^\infty_\theta} \lle \norm{f}_{L^2(r \dif r)L^2_\theta}^{1-1/(2N)}\norm{\Omega^{[0,N]}f}_{L^2(r \dif r)L^2_\theta}^{1/(2N)},
	\]
	where $N=1/\delta^2$.
\end{proof}
\begin{lemma}\label{timederifli}
Suppose that $k \ge -2\delta^2 m$. Then
	\[
	\norm{\FF \ppd t \Omega^{[0,N_2]} V_{k}(t)}_\infty \lle \varepsilon_1^2 2^{-(1/2-5\delta)m}. 	
	\]
Moreover,
\[
	\norm{\FF \ppd t \Omega^{[0,N_2]} V_{k,n}(t)}_\infty \lle \varepsilon_1^2 2^{-(1-25\delta)m},
\]
if $n \le D$.
\end{lemma}
\begin{proof}
We only prove the first estimate, as the second one follows from a similar argument. We have
	\[
	\ppd t V^p_{k}(t)=P_{k}\sum _{k_1,k_2} I_{\gti m_q,t}(\va(t),\vb(t))+ C,
	\]
where $C$ denotes cubic terms. 	The terms with $k_1 \vee k_2 \ge \delta^3 m$ or $k_1 \wedge k_2 \le -2m$ are easy to handle, and so are the cubic terms. Indeed, the terms with $k_1 \vee k_2 \ge \delta^3 m$ or $k_1 \wedge k_2 \le -2m$ are estimated using
	\begin{align*}
	&\norm{\FF P_kI_{\gti m_q,t}(\va(t),\vb(t))}_\infty \\
	\lle{} &2^{\cl k}2^{4\cl k_+}(\norm{\va(t)}_2\norm{\vb(t)}_2 \wedge 2^{2(k_1 \wedge k_2)}\norm{\mao{\va(t)}}_\infty \norm{\mao{\vb(t)}}_\infty) ;		
	\end{align*}
	for the terms with $k_1 \vee k_2 \ge \delta^3 m$, we additionally use \eqref{energyest}. We may assume that $k_1 \ge k_2$. We continue to decompose $I_{\gti m_q,t}(\va(t),\vb(t))$ into $I_{\gti m_q,t}(\Vaa(t),\Vbb(t))$ and the case $j_1 \vee j_2 \ge 2m$ is easy to handle. We first consider the case $k_2 \le -3\delta^2 m$. For the case $j_1 \ge (1-5\delta^2)m$, we estimate
	\begin{align*}
	\norm{\FF P_k I_{\gti m_q,t}(\Vaa(t),\Vbb(t))}_\infty 
	&\lle 2^{\cl k}2^{10\cl k_+} \norm{\Vaa(t)}_2 \norm{\Vbb(t)}_2\\ 
	&\lle \varepsilon_1^2 2^{\delta^2 m}2^{-(1-20\delta)j_1}2^{(1/2-18\delta)n_1}2^{-\delta k_2},
	\end{align*}
	while for the case $j_2 \ge (1-5\delta^2)m$, we have
	\begin{align*}
	\norm{\FF I_{\gti m_q,t}(\Vaa(t),\Vbb(t))}_\infty \lle{}& 2^{\cl k}2^{10\cl k_+} 2^{k_2}\norm{\mao{\Vaa}(t)}_\infty \norm{\Vbb(t)}_2\\
	 \lle{}& \varepsilon_1^2 2^{-(1-25\delta)m}.
	\end{align*}
	If $j_1 \vee j_2 \le (1-5\delta^2)m$, then integration by parts in $\eta$ shows that these terms are negligible. It remains to treat the case $k_2 \ge -3\delta^2 m$. For the case $j_1 \wedge j_2 \ge (1-5\delta^2)m$, we estimate
	\[
		\norm{\FF P_k I_{\gti m_q,t}(\Vaa(t),\Vbb(t))}_\infty \lle 2^{\cl k}2^{10\cl k_+} \norm{\Vaa(t)}_2\norm{\Vbb(t)}_2.
	\]
	When $j_1 \wedge j_2 \le (1-5\delta^2)m$, we insert a cutoff function $\fai_{>\log_2 \kappa_\theta}(\Theta)$ with $\kappa_\theta=2^{-(1/2-2\delta^2)m}$, and \lemmaref{rotainte} shows that the $>\kappa_\theta$-part is negligible. We next consider the case $j_1 \vee j_2 \ge (1/2)m$. We only deal with the case $j_2 \ge (1/2)m$ as the case $j_1 \ge (1/2)m$ is similar. We estimate
	\begin{align*}
		&\abs{\FF P_kI_{\gti m_q,t}(\Vaa(t),\Vbb(t))} \\
		\lle{} &2^{\delta^2m}\int \fai_{\le \log_2\kappa_\theta}(\Theta)\abs{\mao{\Vaa}(\xi-\eta,t)}\abs{\mao{\Vbb}(\eta,t)} \dif \eta\\
		\lle{} &\varepsilon_1 2^{5\delta m} \int \fai_{\le \log_2 \kappa_\theta}(\Theta) \abs{\mao{\Vbb}(\eta,t)} \dif \eta\\
		\lle{} &\varepsilon_1 2^{5\delta m} \int \fai_{\le \log_2 \kappa_\theta}(\Theta) \abs{\mao{\Vbb}(r\theta,t)} r\dif r \dif \theta\\
		\lle{} &\varepsilon_1 2^{5\delta m+\delta^2 m} \kappa_\theta \norm{\mao{\Vbb}(t)}_{L^2(r\dif r)L^\infty_\theta} 2^{-n_2/2}\\
		\lle{} &\varepsilon_1^2 2^{5\delta m+10\delta^2 m}2^{-(1-21\delta)(1/2)m}\kappa_\theta\\
		\lle{} &\varepsilon_1^2 2^{-(1-25\delta)m}.
	\end{align*}
Finally, if $j_1 \vee j_2 \le (1/2)m$, then we insert a cutoff function $\fai_{>-(1/2-\delta)m}(\Xi)$, and integration by parts in $\eta$ shows that the $>-(1/2-\delta)m$-part is negligible. For the $\le -(1/2-\delta)m$-part, a direct size estimate leads to the desired estimate.
\end{proof}
\begin{lemma}\label{timederi2}
Suppose that $0 \le p \le N_2$. Then
\[
	\ppd t \Omega^p V_{k}(t)=I^p_{k}(t)+II^p_{k}(t)+III^p_{k}(t),
\]
where
\begin{gather*}
I^p_{k}(t)=\sum_{\substack{p_1,k_1,p_2,k_2 \colon \\ p_1+p_2=p, -2m \le k_1,k_2 \le \delta^3 m}} P_k I_{\gti m_q,t,>-3\delta^2 m}(\va(t),\vb(t)),\\
\norm{II^p_{k}(t)}_2 \lle \varepsilon_1^2 2^{-(3/2-12\delta)m},\\
\mao{III^p_k}(\xi,t)=\me^{\mi t\Phi^0} G(\xi,t) \fai_{<-2\delta^2 m}(\abs{\xi}-2\gamma_0),
\end{gather*}
where $\Phi^0=-\Lambda(\xi)+2\Lambda(\xi/2)$ and $G$ satisfies the following estimates:
\begin{gather}
	\abs{\partial^\alpha_\xi G(\xi,t)}\lle \varepsilon_1^2 2^{-(1-5\delta)m}2^{\abs \alpha (1/2+3\delta)m}, \label{eq:2.8.1}\\
	\abs{\ppd t G(\xi,t)}\lle \varepsilon_1^2 2^{-(1-5\delta)m}2^{-(1-25\delta)m}. \label{eq:2.8.2}
\end{gather}

\end{lemma}
There are several possible choices of the term $I_{\gti m_q,t,>-3\delta^2 m}(\va(t),\vb(t))$ in the sum defining $I^p_k$; for brevity, we display only one representative term.
\begin{proof}
	We have
	\[
	\ppd t V^p_{k}(t)=\sum _{k_1,k_2} P_k I_{\gti m_q,t}(\va(t),\vb(t))+C.
	\]
	The terms with $k_1 \vee k_2 \ge \delta^3 m$ or $k_1 \wedge k_2 \le -2m$ can be absorbed into $II^p_k$, and so can the cubic terms. Indeed, the terms with $k_1 \vee k_2 \ge \delta^3 m$ or $k_1 \wedge k_2 \le -2m$ can be estimated using
	\[
	\norm{P_kI_{\gti m_q,t}(\va(t),\vb(t))}_2 \lle 2^{\cl k}2^{4\cl k_+}2^{\ul k}\norm{\va(t)}_2\norm{\vb(t)}_2;
	\]
	for the terms with $k_1 \vee k_2 \ge \delta^3 m$, we additionally use \eqref{energyest}. Then we decompose $I_{\gti m_q,t}(\va(t),\vb(t))$ into
	\[
	I_{\gti m_q,t,>-3\delta^2 m}(\va(t),\vb(t)) \text{ and } I_{\gti m_q,t,\le -3\delta^2 m}(\va(t),\vb(t))
	\]
	according to the modulation. For $I_{\gti m_q,t,\le -3\delta^2 m}(\va(t),\vb(t))$, we continue to decompose it into $I_{\gti m_q,t,\le -3\delta^2 m}(\Vaa(t),\Vbb(t))$, and the terms with $j_1 \vee j_2 \ge 4m$ are easy to deal with. We first consider the terms with $n_1 \vee n_2 \ge D$. The terms with $j_1 \vee j_2 \ge (1-2\delta)m$ can be easily estimated by an averaging argument and an $L^2 \times L^\infty$ argument, where we use \eqref{zest} for the $L^2$-part and \cororef{decaydecom} for the $L^\infty$-part. For the terms with $j_1 \vee j_2 \le (1-2\delta)m$, we insert a cutoff function $\fai_{>-\delta m}(\Xi)$, and integration by parts in $\eta$ shows that the $>-\delta m$-part is negligible. For the $\le -\delta m$-part, \lemmaref{resoregion} implies that $\abs \xi$, $\abs{\xi-\eta}$, and $\abs \eta$ are close to $2\gamma_0$, $\gamma_0$, and $\gamma_0$. Hence the $\le -\delta m$-part vanishes. It remains to handle the terms with $n_1 \vee n_2 < D$. As in the treatment of the terms with $n_1 \vee n_2 \ge D$, an averaging argument and an $L^2 \times L^\infty$ argument show that the terms with $j_1 \vee j_2 \ge (1/2)m$ can be absorbed into $II^p_k$. For the terms with $j_1 \vee j_2 < (1/2)m$, we insert a cutoff function $\fai_{>-(1/2-2\delta)m}(\Xi)$, and integration by parts in $\eta$ implies that the $>-(1/2-2\delta)m$-part is negligible.
	
	The Fourier transform of the remaining part (which is denoted by $\mao{III^p_{k}}$) is
	\[
	\sum_{\hti B}\psi_k(\xi)\int \me^{\mi t \Phi} \fai_{\le -3\delta^2 m}(\Phi)\fai_{\le -(1/2-2\delta)m}(\Xi) \gti m_q \mao{\Vaa}(\xi-\eta,t)\mao{\Vbb}(\eta,t) \dif \eta,
	\]
	where $\hti B=\set{(k_1,k_2,j_1,j_2,n_1,n_2) \colon -2m<k_1,k_2<\delta^3 m,\quad j_1,j_2<(1/2)m,\quad n_1,n_2<D}$.
	Decompose $\Phi$ into $\Phi^0$ and $\cl{\Phi}$ with $\Phi^0=-\Lambda(\xi)+2\Lambda(\xi/2)$. Set
	\[
	G(\xi,t)=\psi_k(\xi)\sum_{\hti B} \int \me^{\mi t \cl{\Phi}} \fai_{\le -3\delta^2 m}(\Phi)\fai_{\le -(1/2-2\delta)m}(\Xi) \gti m_q \mao{\Vaa}(\xi-\eta,t)\mao{\Vbb}(\eta,t) \dif \eta.
	\]
 We next verify that $G$ satisfies \eqref{eq:2.8.1}-\eqref{eq:2.8.2}. A simple calculation shows that
	\begin{gather*}
		\abs{\eta-\xi/2} \lle 2^{-(1/2-2\delta)m},\\
		\abs{\cl \Phi} \lle 2^{-(1-4\delta)m},\\
		\abs{\nabla_\xi \cl \Phi} \lle 2^{-(1/2-2\delta)m}. 
	\end{gather*}
	Hence
	\[
	\abs{\partial^\alpha G(\xi,t)} \lle \epsl_1^2 2^{-(1-5\delta)m}2^{\abs \alpha (1/2+3\delta)m}.
	\]
	Moreover, it follows from \lemmaref{timederifli} that
	\[
	\abs{\ppd t G(\xi,t)} \lle \epsl_1^2 2^{-(1-5\delta)m}2^{-(1-25\delta)m}.
	\]
\end{proof}

\section{Energy Estimates}
In this section, we derive the energy estimates, handling $P_{>100} U$ and $P_{\le 100} U$ differently. For $P_{>100}U$, we estimate $P_{>100}\hti U$ via the energy estimate method (by estimating $\re (\hti W,\hti W)_{t}$). For $P_{\le 100}U$, since there is no derivative loss, we instead use the dispersive formulation to obtain the required bounds. The first three subsections are devoted to the estimates of $P_{>100}U$. By \eqref{energymain1}, it suffices to show that
\[
\abs{\gti T_{\wan{\gti m}}(G,W^1,W^2)} \lle \epsl_1^3 \spa{t}^{2\delta^4}, \quad \wan{\gti m} \in\{\gti d,\gti m,\gti m'\},
\]
where $G \in \gti G^{\iota}$, $W^1 \in \gti W^{\iota_1}$, and $W^2 \in \cl{\gti W^{0}}$ (recall the definitions in \eqref{energymain1} and the notation in Section~\ref{multilinear}). In what follows, we use $\gti T_{\gti d | \gti m|\gti m'}$ to denote any one of $\gti T_{\gti d}$, $\gti T_{\gti m}$, and $\gti T_{\gti m'}$; analogous conventions will be used for similar notation throughout the paper. Moreover, we denote the profiles of $G$, $W^1$, and $W^2$ by $g$, $V^1$, and $V^2$, respectively; for example, $G=E^\iota g$. We decompose $\gti T_{\gti d |\gti m | \gti m'}(G,W^1,W^2)$ into $\gti T_{\gti d |\gti m | \gti m',m}(G,W^1,W^2)$, and the cases $0 \le m \le D^4$ and $m=L+1$ are easy to deal with. We classify frequencies into ``good frequencies'' and ``bad frequencies''. Define the bad frequency cutoff:
\[
\chi_{\text{ba}}(\xi)=\fai(2^{10000}(\abs{\xi}-2\gamma_0))+\fai(2^{10000}\abs{\xi}),
\]
and set
\[
\chi_{\text{go}}(\xi)=1-\chi_{\text{ba}}(\xi).
\]
Inserting cutoff functions $\chi_{\text{ba}}(\xi-\eta)$ and $\chi_{\text{go}}(\xi-\eta)$ into $\gti T_{\gti d |\gti m | \gti m',m}(G,W^1,W^2)$, we obtain $\gti T^{\text{ba}|\text{go}}_{\gti d |\gti m | \gti m',m}(G,W^1,W^2)$. Note that if the frequency of $G$ is greater than $2^{-D_1}$, then we can treat $\gti m$ and $\gti m'$ in the same way by assuming that $\norm{\gti m}_{S_{k,k_1,k_2}} \lle (1+2^{6(k_1 \wedge k_2)})/(1+2^{k_1 \vee k_2})$. To deal with the case when the input frequencies are ``bad frequencies'' and the modulation is small, we will use the following $L^2$ lemma.
\begin{lemma}\label{l2lemma}
Suppose that $t \simeq 2^m \ge 2^{D^4}$, $l-k+k_1/3 \ge -(1-\delta)m$, and $k \ge 50$. Let
\[
	Tg(\xi)=\int \me^{\mi t\Phi_{\iota_1\iota_2}(\xi,\eta)} \fai_l(\Phi_{\iota_1\iota_2}) \chi_{\text{ba}}(\xi-\eta) a(\xi,\eta) \abs{\xi-\eta}\mao{f_{k_1}}(\xi-\eta) \psi_k(\eta)g(\eta) \dif \eta,
	\]
	where $(\iota_1,\iota_2) \in \{+,-\}^2$, $f=\Omega^{[0,N_1/2]}V$, and
	\[
	\abs{\partial^\alpha a(\xi,\eta)} \lle \abs{\xi-\eta}^{-\abs \alpha},
	\]
	for all multi-indices \(\alpha\).
	Then
	\[
	\norm{Tg}_2 \lle \epsl_1 2^{15\delta m}(2^{3/2(l-k+k_1/3)}+2^{1/2(l-k+k_1/3)-2m/3}) \norm{g}_2.
		\] 
\end{lemma}
We will prove it in Section~\ref{sec:l2lemma}. The intermediate- and low-frequency part $P_{\le 100}U$ will be handled in Section~\ref{interlowfre}.

\subsection{Good Frequencies}\label{sec:3.1gofre}
In this subsection, we address $\gti T^{\text{go}}_{\gti d |\gti m | \gti m',m}(G,W^1,W^2)$ (we will omit $\gti m'$ below). We call them ``good frequencies'' because they exhibit optimal decay. We decompose $\gti T^{\text{go}}_{\gti d |\gti m,m}(G,W^1,W^2)$ into $\gti T^{\text{go}}_{\gti d |\gti m,m}(G_{k},W^1_{k_1},W^2_{k_2})$ where $k_2 \gge 0$. We may assume that $k \le \delta^3 m$ (see the proof of \lemmaref{timederi2}). If the multiplier is $\gti m$, then the factor $2^{-\cl{k}_+}$ in $\gti m$ allows us to exclude the cases $\cl k \ge 10 m$ and $\ul k \le -10m$, so the summations over $k$, $k_1$, and $k_2$ are harmless except in the absence of extra decay. If the multiplier is $\gti d$, then, since $k_1 \simeq k_2$, the summation over $k_1$ and $k_2$ can be handled by the Cauchy--Schwarz inequality. We insert localizations of the modulation $\Phi$ with parameter $l \ge l_0=\flo{\cl k-100 \delta m}$ into it. In view of \lemmaref{symbound}(i), we have $\norm{\fai^{l_0}_l(\Phi(\xi,\eta))}_{S_{k_2,k,k_1}} \lle 2^{500\delta m}$.

We next deal with the case $l>l_0$. Integration by parts in time gives
\begin{align*}
&\gti T^{\text{go}}_{\gti d |\gti m,m,l}(G_{k},W^1_{k_1},W^2_{k_2})=\gti R^{\text{go},\prime}_{\gti d |\gti m,m,l}(G_{k},W^1_{k_1},W^2_{k_2})+\gti T^{\text{go},\prime}_{\gti d |\gti m,m,l}((\ppd s-\mi \iota\Lambda(D))G_{k},W^1_{k_1},W^2_{k_2})\\
&+\gti T^{\text{go},\prime}_{\gti d |\gti m,m,l}(G_{k},(\ppd s-\mi \iota_1\Lambda(D))W^1_{k_1},W^2_{k_2})+\gti T^{\text{go},\prime}_{\gti d |\gti m,m,l}(G_{k},W^1_{k_1},(\ppd s+\mi\Lambda(D))W^2_{k_2}).
\end{align*}
The terms $\gti R^{\text{go},\prime}_{\gti d |\gti m,m,l}(G_{k},W^1_{k_1},W^2_{k_2})$ and $\gti T^{\text{go},\prime}_{\gti d |\gti m,m,l}((\ppd s-\mi \iota\Lambda(D))G_{k},W^1_{k_1},W^2_{k_2})$ are easy to handle. For $\gti T^{\text{go},\prime}_{\gti d |\gti m,m,l}(G_{k},(\ppd s-\mi \iota_1\Lambda(D))W^1_{k_1},W^2_{k_2})$ and $\gti T^{\text{go},\prime}_{\gti d |\gti m,m,l}(G_{k},W^1_{k_1},(\ppd s+\mi\Lambda(D))W^2_{k_2})$, the difficulty is the loss of derivatives from $T_{\Sigma_{\ge 2}}$-part and $T_{R \psi \cdot \mi \xi}$-part in the time derivatives (see \lemmaref{timederien}). If the multiplier is $\gti m$, or if the multiplier is $\gti d$ and $\iota_1=-$ (in which case $2^l \gge 2^{2k_1}$), then this difficulty can be handled easily. If the multiplier is $\gti d$ and $\iota_1=+$, then we use the symmetrization argument (see \cite[Lemma~5.3]{MR3784694}). We treat only the $T_{\Sigma_{\ge 2}}$-part. We write $\wan{\psi}_l(x)$ for $\psi(2^{-l}x)/(2^{-l}x)$.
We have
\begin{align*}
	&\gti T^{\text{go},\prime}_{\gti d,m,l}(G_{k},P_{k_1}T_{\mi \Sigma_{\ge 2}}W^1,W^2_{k_2})-\gti T^{\text{go},\prime}_{\gti d,m,l}(G_{k},W^1_{k_1},P_{k_2}T_{\mi \Sigma_{\ge 2}}W^2)\\
	={}&-2^{-l}\int q_m(s)\gti d(\xi,\eta) \wan{\psi}_l(\Phi) {\varphi_{\leq -10}}\Brac{\frac{\abs{\sigma}}{\abs{2\eta-\sigma}}}\psi_{k_1}(\eta)\psi_{k_2}(-\xi)\mao{G_{k}}(\xi-\eta,s) \mao{\Sigma_{\ge 2}}(\sigma,\eta-\sigma/2)\\
	&\times \mao{W^1_{k_1}}(\eta-\sigma,s) \mao{W^2_{k_2}}(-\xi,s) \dif \xi \dif \eta \dif \sigma \dif s\\
	&+2^{-l}\int  q_m(s)\gti d(\xi,\eta) \wan{\psi}_l(\Phi) {\varphi_{\leq -10}}\Brac{\frac{\abs{\sigma}}{\abs{-2\xi-\sigma}}}\psi_{k_1}(\eta)\psi_{k_2}(-\xi)\mao{G_{k}}(\xi-\eta,s)\\ \\
	&\times \mao{\Sigma_{\ge 2}}(\sigma,-\xi-\sigma/2) \mao{W^1_{k_1}}(\eta,s) \mao{W^2_{k_2}}(-\xi-\sigma,s)  \dif \xi \dif \eta \dif \sigma \dif s\\
	={}&-2^{-l}\int q_m(s)\\
	&\Bigl[\gti d(\xi,\eta+\sigma) \wan{\psi}_l(\Phi(\xi,\eta+\sigma)) {\varphi_{\leq -10}}\Brac{\frac{\abs{\sigma}}{\abs{2\eta+\sigma}}}\psi_{k_1}(\eta+\sigma)\psi_{k_2}(-\xi) \mao{\Sigma_{\ge 2}}(\sigma,\eta+\sigma/2)\\
	&-\gti d(\xi-\sigma,\eta) \wan{\psi}_l(\Phi(\xi-\sigma,\eta)) {\varphi_{\leq -10}}\Brac{\frac{\abs{\sigma}}{\abs{-2\xi+\sigma}}} \psi_{k_1}(\eta)\psi_{k_2}(-\xi+\sigma)\mao{\Sigma_{\ge 2}} (\sigma,-\xi+\sigma/2)\Bigr]\\
	&\times\mao{G_{k}}(\xi-\eta-\sigma,s)\mao{W^1_{k_1}}(\eta,s) \mao{W^2_{k_2}}(-\xi,s)	  \dif \xi \dif \eta \dif \sigma \dif s.
\end{align*}
We decompose
\begin{align*}
	&\gti d(\xi,\eta+\sigma) \wan{\psi}_l(\Phi(\xi,\eta+\sigma)) {\varphi_{\leq -10}}\Brac{\frac{\abs{\sigma}}{\abs{2\eta+\sigma}}}\psi_{k_1}(\eta+\sigma)\psi_{k_2}(-\xi) \mao{\Sigma_{\ge 2}}(\sigma,\eta+\sigma/2)\\
	&-\gti d(\xi-\sigma,\eta) \wan{\psi}_l(\Phi(\xi-\sigma,\eta)) {\varphi_{\leq -10}}\Brac{\frac{\abs{\sigma}}{\abs{-2\xi+\sigma}}} \psi_{k_1}(\eta)\psi_{k_2}(-\xi+\sigma)\mao{\Sigma_{\ge 2}} (\sigma,-\xi+\sigma/2)\\
	 &=\underbrace{(\gti m_0(\xi,\eta+\sigma)-\gti m_0(\xi-\sigma,\eta))\wan{\psi}_l(\Phi(\xi,\eta+\sigma)) {\varphi_{\leq -10}}\Brac{\frac{\abs{\sigma}}{\abs{2\eta+\sigma}}} \mao{\Sigma_{\ge 2}}(\sigma,\eta+\sigma/2)}_{(1)}\\
	&+\underbrace{\gti m_0(\xi-\sigma,\eta)(\wan{\psi}_l(\Phi(\xi,\eta+\sigma))-\wan{\psi}_l(\Phi(\xi-\sigma,\eta))){\varphi_{\leq -10}}\Brac{\frac{\abs{\sigma}}{\abs{2\eta+\sigma}}} \mao{\Sigma_{\ge 2}}(\sigma,\eta+\sigma/2)}_{(2)}\\
	&+\underbrace{\gti m_0(\xi-\sigma,\eta)\wan{\psi}_l(\Phi(\xi-\sigma,\eta))\Brac{\varphi_{\leq -10}\Brac{\frac{\abs{\sigma}}{\abs{2\eta+\sigma}}}-{\varphi_{\leq -10}}\Brac{\frac{\abs{\sigma}}{\abs{-2\xi+\sigma}}}}}_{(3)}\\
	&\underbrace{\times\mao{\Sigma_{\ge 2}}(\sigma,\eta+\sigma/2)}_{(3)}\\
	&+\underbrace{\gti m_0(\xi-\sigma,\eta)\wan{\psi}_l(\Phi(\xi-\sigma,\eta)){\varphi_{\leq -10}}\Brac{\frac{\abs{\sigma}}{\abs{-2\xi+\sigma}}}(\mao{\Sigma_{\ge 2}}(\sigma,\eta+\sigma/2)-\mao{\Sigma_{\ge 2}} (\sigma,-\xi+\sigma/2))}_{(4)},
\end{align*}
where $\gti m_0(\xi,\eta)=\gti d(\xi,\eta)\psi_{k_1}(\eta)\psi_{k_2}(-\xi)$. Terms $(1)$ and $(2)$ provide a factor $\sigma$, and terms $(3)$ and $(4)$ provide a factor $\xi-\eta-\sigma$. It is clear that terms $(1)$, $(3)$, and $(4)$ provide a gain $2^{-k_1}$. For $(2)$, note that $\Phi(\xi,\eta)=-\Lambda(\xi)+\Lambda(\eta)+\iota\Lambda(\xi-\eta)=(\xi-\eta) \cdot G(\xi,\eta)+\iota\Lambda(\xi-\eta)$ where $\abs{\ppd {\xi,\eta}^{\alpha} G} \lle 2^{-\abs \alpha k_1+k_1}$. Hence term $(2)$ provides a gain $2^k2^{-l}$. The desired estimate then follows from \lemmaref{symbound1}(i) and \lemmaref{pseudomultithm}.

We now deal with the case $l=l_0$. We first exclude the case $\cl k \ge 50\delta m$. Instead of the symbol bounds, we use the averaging argument to deal with this case (see Section~\ref{techlemma} for a discussion of the averaging argument). If the multiplier is $\gti m$, then the desired bound follows immediately since $\norm{\gti m}_{S_{k_2,k,k_1}} \lle 2^{6k_+}2^{-\cl k_+}$. If the multiplier is $\gti d$, then the condition $l=l_0$ implies that $\iota_1=+$. We next exploit the depletion factors. Noting that
\begin{equation}\label{depletion}
(\xi-\eta) \cdot (\xi+\eta)=\abs{\xi}^2-\abs{\eta}^2=\abs{\xi}^2-\Lambda(\xi)-(\abs{\eta}^2-\Lambda(\eta))-\Phi+\iota\Lambda(\xi-\eta),
\end{equation}
we see that the depletion factor $\mao{\xi-\eta} \cdot \mao{\xi+\eta}$ provides extra decay.

We now assume that $\cl k \le 50\delta m$. Then we have $l_0 \ll 0$, and hence $\cl k +\ul k \gge 0$. It follows that $\norm{\gti m}_{S_{k_2,k,k_1}} \lle 2^{6k_+}2^{\wan{k_1}}$. We now decompose $G$ into $G_{k,j,n}$ where $G_{k,j,n}=E^\iota g_{k,j,n}$. In view of \lemmaref{decaylemma}, we can obtain the desired estimates if $j \le (1-\delta)m$. If $j \ge (1-\delta)m$, then Schur's test and \lemmaref{sizeangle}(i) yield
\begin{equation}\label{eq:3.2.20}
	\abs{\bbT^{\text{go}}_{\gti m,l_0}(G_{k,j,n}(s),W^1_{k_1}(s),W^2_{k_2}(s))} \lle 2^{\delta^2 m}2^{l_0/2}\norm{\mao{g_{k,j,n}}(s)}_{L^2(r\dif r)L^\infty_\theta}\norm{W^1_{k_1}(s)}_2\norm{W^2_{k_2}(s)}_2.
\end{equation}

\subsection{Bad Frequencies}\label{sec:badfre}
In this subsection, we estimate $\gti T^{\text{ba}}_{\gti d |\gti m | \gti m',m}(G,W^1,W^2)$. We will deal with $\gti d$, $\gti m$ and $\gti m'$ in different ways due to the low-frequency components of the bad frequency.

\noindent\underline{1. The multipliers $\gti d$ and $\gti m$.}\\
We decompose $\TTT^{\text{ba}}_{\gti d | \gti m,m}(G,W^1,W^2)$ into $\TTT^{\text{ba}}_{\gti d | \gti m,m,l}(G_k,W^1_{k_1},W^2_{k_2})$ where $l \ge l_0=\flo{\cl k-k/3-2m/3-m/1000}$. We can assume that $k \ge -2m$. The case $l=l_0$ follows from \lemmaref{l2lemma}. Our goal is to treat the case $l>l_0$. The difficult cases typically occur when $k_1$ is low. If the multiplier is $\gti d$ and $\iota_1=-$, then the bound $l \gge 2k_1$ allows us to estimate it easily as there is no issue of loss of derivatives after performing integration by parts in time. We first treat the case \ul{$l \ge -m/100$}. If the multiplier is $\gti m$, then we can first assume that $\cl k \le 100\delta m$ due to the factor $2^{-\cl k_+}$ in $\gti m$. Integration by parts in time then leads to the desired bound as a consequence of \lemmaref{symbound}(iii). Hence we focus on the case that the multiplier is $\gti d$ and $\iota_1=+$. Because of the depletion factors, we can assume that $l \ge \cl k+\ul k-m/50$. \lemmaref{symbound}(ii) then yields $\norm{\fai^{l_0}_l(\Phi(\xi,\eta))}_{S_{k_2,k,k_1}} \lle 2^{m/10}$. Integration by parts in time gives
\begin{align*}
	&\gti T^{\text{ba}}_{\gti d,m,l}(G_{k},W^1_{k_1},W^2_{k_2})=\gti R^{\text{ba},\prime}_{\gti d,m,l}(G_{k},W^1_{k_1},W^2_{k_2})+\gti T^{\text{ba},\prime}_{\gti d,m,l}((\ppd s-\mi \iota\Lambda(D))G_{k},W^1_{k_1},W^2_{k_2})\\
&+\gti T^{\text{ba},\prime}_{\gti d,m,l}(G_{k},(\ppd s-\mi \iota_1 \Lambda(D))W^1_{k_1},W^2_{k_2})+\gti T^{\text{ba},\prime}_{\gti d,m,l}(G_{k},W^1_{k_1},(\ppd s+\mi\Lambda(D))W^2_{k_2}).
\end{align*}
The terms $\gti R^{\text{ba},\prime}_{\gti d,m,l}(G_{k},W^1_{k_1},W^2_{k_2})$ and $\gti T^{\text{ba},\prime}_{\gti d,m}((\ppd s-\mi \iota\Lambda(D))G_{k},W^1_{k_1},W^2_{k_2})$ can be estimated easily. For $\gti T^{\text{ba},\prime}_{\gti d,m,l}(G_{k},(\ppd s-\mi \iota_1 \Lambda(D))W^1_{k_1},W^2_{k_2})$ and $\gti T^{\text{ba},\prime}_{\gti d,m,l}(G_{k},W^1_{k_1},(\ppd s+\mi\Lambda(D))W^2_{k_2})$, we can use the symmetrization argument as in Section~\ref{sec:3.1gofre} to deal with them; the worst term is the $T_{R\psi \cdot \mi \xi}$-part, where the multipliers give gains $2^{-\cl k}$ or $2^k2^{-l}$.

We can now assume that \ul{$l \le -m/100$}. We can then further assume that $0 \lle \ul k+\cl k \lle 100\delta m$ and $\cl k \lle m/2-m/200$.
Integration by parts in time yields
\begin{align*}
&\gti T^{\text{ba}}_{\gti d |\gti m,m,l}(G_{k},W^1_{k_1},W^2_{k_2})=\gti R^{\text{ba},\prime}_{\gti d |\gti m,m,l}(G_{k},W^1_{k_1},W^2_{k_2})+\gti T^{\text{ba},\prime}_{\gti d |\gti m,m,l}((\ppd s-\mi \iota\Lambda(D))G_{k},W^1_{k_1},W^2_{k_2})\\
&+\gti T^{\text{ba},\prime}_{\gti d |\gti m,m,l}(G_{k},(\ppd s-\mi \iota_1 \Lambda(D))W^1_{k_1},W^2_{k_2})+\gti T^{\text{ba},\prime}_{\gti d |\gti m,m,l}(G_{k},W^1_{k_1},(\ppd s+\mi\Lambda(D))W^2_{k_2}).	
\end{align*}
The term $\gti R^{\text{ba},\prime}_{\gti d |\gti m,m,l}(G_{k},W^1_{k_1},W^2_{k_2})$ is straightforward to estimate. The two terms \\ $\gti T^{\text{ba},\prime}_{\gti d |\gti m,m,l}(G_{k},(\ppd s-\mi \iota_1\Lambda(D))W^1_{k_1},W^2_{k_2})$ and $\gti T^{\text{ba},\prime}_{\gti d |\gti m,m,l}(G_{k},W^1_{k_1},(\ppd s+\mi\Lambda(D))W^2_{k_2})$ are similar, so we treat only $\gti T^{\text{ba},\prime}_{\gti d |\gti m,m,l}(G_{k},(\ppd s-\mi \iota_1\Lambda(D))W^1_{k_1},W^2_{k_2})$. In view of \lemmaref{timederien}, the $R$-part and the $T_{\mi \Sigma_{\ge 1}}$-part are easy to handle. The $T_{R\psi \cdot \mi \xi}$-part and the $T_{\gti m'}$-part are similar, so we only deal with the former one. We write it as $2^{-l}\gti T^{\text{ba}}_{m,l}(G_{k},U^{\iota_3},W^1_{k_1},W^2_{k_2})$. Then we decompose it into $\gti T^{\text{ba}}_{m,l}(G_{k,j,n},U^{\iota_3}_{k_3,j_3,n_3},W^1_{k_1},W^2_{k_2})$, and we can assume that $-2m \le  k_3 \le \delta^3 m$ and $j , j_3 \le 10m$. For the case $k_3 \gge k$, we have $k_3 \gge k \gge -m/2+m/200$. It then follows from an averaging argument, \cororef{decayestcom}, and \eqref{energyest} that
\begin{align*}
	&2^{-l}\abs{\gti T^{\text{ba}}_{m,l}(G_{k,j,n},U^{\iota_3}_{k_3,j_3,n_3},W^1_{k_1},W^2_{k_2})}\\
	\lle{} &2^{\delta^2 m}2^m2^{-k_1}2^{k/3}2^{2m/3+m/1000}2^k2^{k_1} \sup_{s \simeq 2^m}\norm{E_{[-2^{(1-\delta^2)m},2^{(1-\delta^2)m}]}G_{k,j,n}(s)}_\infty \norm{U^{\iota_3}_{k_3,j_3,n_3}(s)}_\infty \\
&\times \norm{W^1_{k_1}(s)}_2 \norm{W^2_{k_2}(s)}_2+\epsl_1^4 2^{-10m}\\
	\lle{} &\epsl_1^4 2^{2\delta^2 m}2^m2^{k_3/3}2^{2m/3+m/1000}2^{-(1-21\delta)m}2^{-(1-21\delta)m}2^{-k_3} +\epsl_1^4 2^{-10m}\lle{} \epsl_1^4 2^{-\delta^4 m}.
\end{align*}
When $k_3 \ll k$, we aim to exploit the absence of space resonances between $G_{k,j,n}$ and $U^{\iota_3}_{k_3,j_3,n_3}$. The case $j_3 \ge (1-\delta^2)m+k$ can be handled in a manner similar to the case $k_3 \gge k$. For the case $j \ge (1-\delta^2)m+k$, when $k \ll 0$, using Schur's test together with $\abs{\nabla_\eta \Phi}, \abs{\nabla_{\xi} \Phi} \gge 2^{k_1}$, and applying \eqref{zest}, \cororef{decayestlow}, and \eqref{energyest}, we estimate
\begin{align*}
&\abs{\gti T^{\text{ba}}_{m,l}(G_{k,j,n},U^{\iota_3}_{k_3,j_3,n_3},W^1_{k_1},W^2_{k_2})}\\
\lle{} &2^{\delta^2 m}2^m2^k2^{k_1} 2^{(l-k_1)/2}2^{k/2}\sup_{s \simeq 2^m}\norm{G_{k,j,n}}_2\norm{U^{\iota_3}_{k_3,j_3,n_3}}_\infty \norm{W^1_{k_1}}_2\norm{W^2_{k_2}}_2\\
\lle{} &\epsl_1^4 2^{2\delta^2 m}2^m2^{k_1} 2^{l/2+k}2^{-(1-20\delta)j}2^{-(1/2-2\delta)m},
\end{align*}
while when $k \simeq 0$, using Schur's test, \lemmaref{sizeangle}(ii), \eqref{fl2r}, \cororef{decayestlow}, and \eqref{energyest}, we have
\begin{align*}
	&\abs{\gti T^{\text{ba}}_{m,l}(G_{k,j,n},U^{\iota_3}_{k_3,j_3,n_3},W^1_{k_1},W^2_{k_2})}\\
	\lle{} &2^{\delta^2 m}2^m2^{k_1}2^{l-k_1}2^{-n/2}\sup_{s \simeq 2^m}\norm{\mao{g_{k,j,n}}}_{L^2(r \dif r)L^\infty_\theta}\norm{U^{\iota_3}_{k_3,j_3,n_3}}_\infty \norm{W^1_{k_1}}_2\norm{W^2_{k_2}}_2\\
	\lle{} &\epsl_1^4 2^{2\delta^2 m}2^m2^{l} 2^{-n/2}2^{-(1-21\delta)j}2^{(1/2-18\delta)n}2^{-(1/2-2\delta)m}.
\end{align*}
Integration by parts in the frequency variables, together with an averaging argument, then concludes this part.

Finally, we address $\gti T^{\text{ba},\prime}_{\gti d |\gti m,m,l}((\ppd s-\mi \iota\Lambda(D))G_{k},W^1_{k_1},W^2_{k_2})$. By \lemmaref{timederi2}, for $s \simeq 2^m$, we have
\[
	\ppd s  g_{k}(s)=I_k(s)+II_k(s)+III_k(s),
\]
where
\begin{gather*}
I_k(s)=\sum_{\substack{k_3,p_3,k_4,p_4 \colon \\ p_3+p_4 \le N_1/2,-2m \le k_3,k_4 \le \delta^3 m}} P_k I_{\gti m_0,s,>-3\delta^2 m}(\vx 3(s),\vx 4(s)), \\
\text{where the phase function is } -\iota\Lambda(\xi)+\iota_3\Lambda(\xi-\eta)+\iota_4\Lambda(\eta) ,\\
\norm{II_{k}(s)}_2 \lle \varepsilon_1^2 2^{-(3/2-12\delta)m},\\
\mao{III_k}(\xi,s)=\me^{\mi s \iota \Phi^0} G(\xi,s) \fai_{<-2\delta^2 m}(\abs{\xi}-2\gamma_0),
\end{gather*}
where $\Phi^0=-\Lambda(\xi)+2\Lambda(\xi/2)$ and $G$ satisfies the following estimates:
\[
	\abs{\partial^\alpha_\xi G(\xi,s)}\lle \varepsilon_1^2 2^{-(1-5\delta)m}2^{\abs \alpha (1/2+3\delta)m},
\]
where $s \simeq 2^m$.
For the $I_k$-part, our treatment is analogous to that in Section~\ref{sec:dispersiveinter}. We write the $I_k$-part as a multilinear oscillatory integral for profiles $2^{-l}\gti I_{m,l,q,p}(\vx 3,\vx 4,V^1_{k_1},V^2_{k_2})$ (see Section~\ref{multilinear}). Here the superscripts $\text{ba}$ and the multiplier are suppressed. We denote by $\xi_3$ and $\xi_4$ the frequency variables of $\vx 3$ and $\vx 4$. We insert localizations of the second resonance function $\wan \Phi$ and $\Phi^{3,4}_{\iota\iota_3\iota_4}$ with parameters $p, q \gge -3\delta^2 m$.  Integration by parts in time yields
\begin{align*}
	&\gti I_{m,l,q,p}(\vx 3,\vx 4,V^1_{k_1},V^2_{k_2})=\gti J'_{m,l,q,p}(\vx 3,\vx 4,V^1_{k_1},V^2_{k_2})\\&+\gti I'_{m,l,q,p}(\ppd s \vx 3,\vx 4,V^1_{k_1},V^2_{k_2})
	+\gti I'_{m,l,q,p}(\vx 3,\ppd s\vx 4,V^1_{k_1},V^2_{k_2})\\&+\gti I'_{m,l,q,p}(\vx 3,\vx 4,\ppd s V^1_{k_1},V^2_{k_2})
	+\gti I'_{m,l,q,p}(\vx 3,\vx 4,V^1_{k_1},\ppd s V^2_{k_2}).
\end{align*}
All these terms can be handled, together with an averaging argument, in a manner similar to that used for the $II_k$-part, so we estimate only $\gti I'_{m,l,q,p}(\ppd s \vx 3,\vx 4,V^1_{k_1},V^2_{k_2})$. By the averaging argument, we can write it as
\begin{align*}
	&\int_{\abs{\lambda_1},\abs{\lambda_2} \le 2^{\delta^2 m}} \mao{\fai_1}(\lambda_1)\mao{\fai_2}(\lambda_2) \Biggl( \int q_m(s) \me^{\mi (s\wan \Phi+\lambda_12^{-q} \Phi^{3,4}_{\iota\iota_3\iota_4}+\lambda_2 2^{-p} \wan \Phi)} \fai^{l_0}_l(\Phi) (\gti d | \gti m)(\xi,\eta) \gti m_q(\xi-\eta,\sigma)\\
	&\chi_{\text{ba}}(\xi-\eta)\psi_k(\xi-\eta)
	 \mao{\ppd s \vx 3}(\xi-\eta-\sigma,s) \mao{\vx 4}(\sigma,s) \mao{V^1_{k_1}}(\eta,s) \mao{V^2_{k_2}}(-\xi,s) \dif \xi \dif \eta \dif \sigma \dif s \Biggr)\\ 
	 &\times \dif \lambda_1 \dif \lambda_2+O(\epsl_1^32^{-10m}),
\end{align*}
which can be rearranged as
\begin{align*}
	&\int_{\abs{\lambda_1},\abs{\lambda_2} \le 2^{\delta^2 m}} \mao{\fai_1}(\lambda_1)\mao{\fai_2}(\lambda_2) \Biggl( \int q_m(s) \me^{\mi (s+\lambda_22^{-p}) \Phi} \fai^{l_0}_l(\Phi) (\gti d | \gti m)(\xi,\eta) \chi_{\text{ba}}(\xi-\eta)\psi_k(\xi-\eta) \\
	&\FF I_{\gti m_q,(s+\lambda_1 2^{-q}+\lambda_2 2^{-p})}(\ppd s \vx 3,\vx 4) \mao{V^1_{k_1}}(\eta,s) \mao{V^2_{k_2}}(-\xi,s) \dif \xi \dif \eta  \dif s \Biggr) \dif \lambda_1 \dif \lambda_2+O(\epsl_1^32^{-10m}).
\end{align*}
We can use Schur's test, together with the fact that $\abs{\nabla_\eta \Phi} \gge 2^{k_1}$, to bound the main contribution by
\[
\epsl_1^4 2^{\delta^2 m}2^m 2^{(l-k_1)/2}2^{k/2}2^{-(3/2-1/100)m}.
\]
Finally, we handle the $III_k$-part. The key is the following fact:
\[
\norm{E_{\iota t+\rho}III_k(t)}_\infty \lle \epsl_1^2 2^{-(2-1/100)m},
\]
where $t \simeq 2^m$ and $\abs{\rho} \le 2^{(1-\delta^2)m}$. We prove it via the stationary phase method. We have
\[
E_{\iota t+\rho}III_k(t)(x)=\int \me^{2\mi \iota t \Lambda(\xi/2)+\mi \rho \Lambda(\xi)+\mi x \cdot \xi} G(\xi,t) \fai_{<-2\delta^2 m}(\abs \xi-2\gamma_0) \dif \xi.
\]
Set $\Psi(\xi,t,x)=[\iota \lambda'(\abs{\xi}/2)+(\rho/t) \lambda'(\abs \xi)]\mao \xi+x/t$ and insert a cutoff function \\ $\fai_{\ge -(1/2-100\delta)m}(\Psi)$ into it to divide the integral into $I_{\ge 2^{-(1/2-100\delta)m}}$ and $I_{< 2^{-(1/2-100\delta)m}}$. For $I_{\ge 2^{-(1/2-100\delta)m}}$, integration by parts in $\xi$ shows that this part is negligible. For $I_{< 2^{-(1/2-100\delta)m}}$, we use polar coordinates to bound its size by
\[
\epsl_1^2 2^{-(1-5\delta)m}2^{-(1/2-100\delta)m}2^{-(1/2-100\delta)m}.
\]
An averaging argument now yields 
\begin{align*}
\abs{\gti I^{\text{ba}}_{\gti d |\gti m,m,l}(III_k,V^1_{k_1},V^2_{k_2})} \lle{}& 2^m\sup_{ s \in [2^{m-1},2^{m+1}]}\norm{E_{\iota s+O(2^{(1-\delta^2)m})} III_k(s)}_\infty \norm{V^1_{k_1}(s)}_2\norm{V^2_{k_2}(s)}_2\\
&+\epsl_1^3 2^{-10m},	
\end{align*}
which completes the proof.

\noindent\underline{2. The multiplier $\gti m'$.}\\
We decompose $\TTT_{\gti m',m}^{\text{ba}}(G,W^1,W^2)$ into $\TTT_{\gti m',m}^{\text{ba}}(G_k,W_{k_1}^1,W_{k_2}^2)$. First, it is easy to exclude the case $\cl k \ge  m/50$ using the bound $\norm{\gti m'}_{S_{k,k_1,k_2}} \lle 2^{90k_+}2^{-80 \cl k_+}$. Hence the issue of the loss of derivatives plays a minor role. Moreover, it is straightforward to discard the case $\ul k \le -2m$. We only consider the case $k \ll 0$, since the case $k \gge 0$ can be addressed in the same way as $\TTT_{\gti m,m}^{\text{ba}}(G_k,W_{k_1}^1,W_{k_2}^2)$. If $\ul k+\cl k \gge 0$, then, since $2^{-\cl k_+} \lle 2^{\ul k}$, we gain enough low-frequency factors and $G$ exhibits optimal decay. Hence the treatment is similar to that in Section~\ref{sec:3.1gofre}. We decompose $\TTT_{\gti m',m}^{\text{ba}}(G_k,W_{k_1}^1,W_{k_2}^2)$ into $\TTT_{\gti m',m,l}^{\text{ba}}(G_k,W_{k_1}^1,W_{k_2}^2)$ where $l \ge l_0=\flo{-m/100}$, and the case $l>l_0$ can be handled by integrating by parts in time. For the case $l=l_0$, we continue to decompose $\TTT_{\gti m',m,l}^{\text{ba}}(G_k,W_{k_1}^1,W_{k_2}^2)$ into $\TTT_{\gti m',m,l}^{\text{ba}}(G_{k,j,n},W_{k_1}^1,W_{k_2}^2)$. Then Schur's test excludes the case $j \ge (1-1/1000)m$, and the case $j \leq (1-1/1000)m$ exhibits optimal decay.

When $\ul k+\cl k \ll 0$, we use \eqref{paraerror} to replace $W^1$ and $W^2$ with expressions involving $U$. Then the proof is similar to that in Section~\ref{subsec:3.4.2}, and hence we omit it.

\subsection{The $L^2$ Lemma}\label{sec:l2lemma}
In this subsection, we prove \lemmaref{l2lemma}. We will employ a $TT^*$-argument to prove it.
We decompose $f_{k_1}(t)$ into $f_{k_1,j_1,n_1}(t)$.
We first exclude the case $j_1 \ge 2m/3$. Using Schur's test, we have
\[
\norm{Tg}_2 \lle 2^{(l-k)/2}2^{k_1/2}2^{k_1}\norm{f_{k_1,j_1,n_1}(t)}_{2}\norm{g}_2
\]
if $k_1 \ll 0$, and
\[
\norm{Tg}_2 \lle  2^{l-k}2^{-n_1/2}\norm{\mao{f_{k_1,j_1,n_1}}(t)}_{L^2(r\dif r)L^\infty_\theta}\norm{g}_2
\]
(\lemmaref{sizeangle}(ii) is applied) if $k_1 \simeq 0$.
Hence we can assume that $j_1 \le 2m/3$.
Take a cutoff function $\chi \in C_0^\infty((-2,2)^2)$ such that $\sum_{v \in \bZ^2} \chi(\xi+v)=1$ and let $\chi_v(\xi)=\chi(\xi/(2^{k_1}2^{-10})+v)$. We then decompose the operator $T$ as $\sum_{v} T_{v} = \sum_{v} \chi_vT$. Because of the restriction $\abs{\xi-\eta} \lle 2^{k_1}$, $T$ is a local operator, and $T_v\chi_w$ vanishes unless $\abs{v-w} \lle 1$. Therefore it suffices to prove the boundedness of $\sup_v \norm{T_v}$. The kernel of $T_vT_v^*$ is
\begin{align*}
K_v(\xi,\xi')={}&\chi_v(\xi)\chi_v(\xi')\int \me^{\mi t (\Phi(\xi,\eta)-\Phi(\xi',\eta))} \fai_l(\Phi(\xi,\eta))\fai_l(\Phi(\xi',\eta))A(\xi,\eta) \cl A(\xi',\eta)\\
&\times \abs{\xi-\eta}\abs{\xi'-\eta}\mao {f_{k_1,j_1,n_1}}(\xi-\eta,t)\cl{\mao {f_{k_1,j_1,n_1}}(\xi'-\eta,t)}\psi_k^2(\eta) \dif \eta,
\end{align*}
where
\[
A(\xi,\eta)=a(\xi,\eta) \chi_{\text{ba}}(\xi-\eta).
\]
We aim to show that
\[
\int \abs{K_v(\xi,\xi')} \dif \xi  + \int \abs{K_v(\xi,\xi')}  \dif \xi' \lle \epsl_1^2 2^{29\delta m}(2^{3(l-k)}2^{k_1}+2^{3(l-k)/2}2^{k_1/2}2^{-m})
\]
(we only deal with the integral with respect to $\xi$ as the other is similar). Take
\[
L=2^{\delta m}(2^{l-k}+2^{j_1-m}).
\]
We first estimate
\begin{align*}
	&\int_{\abs{\xi-\xi'} \le L} \abs{K_v(\xi,\xi')} \dif \xi\\
\lle{} &\epsl_1 2^{5\delta m}2^{k_1} \int_{\abs{\xi-\xi'} \le L} \fai_l(\Phi(\xi,\eta))\fai_l(\Phi(\xi',\eta))\\
  &\times\psi_{k_1}(\xi-\eta) \psi_{k_1}(\xi'-\eta) \abs{\mao {f_{k_1,j_1,n_1}}(\xi'-\eta,t)}\psi_k^2(\eta) \dif \eta \dif \xi\\
\lle{} &\epsl_1 2^{5\delta m}2^{k_1} \int \biggl(\int_{\abs{\xi-\xi'} \le L} \fai_l(\Phi(\xi,\eta))\fai_l(\Phi(\xi',\eta))\\
  &\times\psi_{k_1}(\xi-\eta) \psi_{k_1}(\xi'-\eta) \abs{\mao {f_{k_1,j_1,n_1}}(\xi'-\eta,t)} \psi_k^2(\eta) \dif \xi \biggr) \dif \eta.
\end{align*}
It follows that
\[
\int_{\abs{\xi-\xi'} \le L} \abs{K_v(\xi,\xi')} \dif \xi \lle{} \epsl_1^2 2^{10\delta m} \cdot 2^{l-k}L \cdot (2^{(l-k)/2}2^{k_1/2}2^{-(1-20\delta)j_1} \wedge 2^{l-k}2^{k_1})
\]
if $k_1 \ll 0$, and
\[
		\int_{\abs{\xi-\xi'} \le L} \abs{K_v(\xi,\xi')} \dif \xi
\lle{} \epsl_1^2 2^{10\delta m} \cdot 2^{l-k}L \cdot (2^{l-k}2^{-n_1/2}2^{-(1-21\delta)j_1}2^{(1/2-18\delta)n_1} \wedge 2^{l-k})
\]
if $k_1 \simeq 0$.
For the region $\abs{\xi-\xi'} \ge L$, we use integration by parts to show that this part is negligible. Let
\[
V=(\mao{\nabla_\eta \Phi(\xi,\eta)})^\perp.
\]
We will integrate by parts with respect to the vector field $V$. We first analyze $V \cdot \nabla_\eta(\Phi(\xi,\eta)-\Phi(\xi',\eta))$. It follows from the localization of modulation that
\begin{equation}\label{eq:moduest}
2^l \gge \abs{\Phi(\xi,\eta)-\Phi(\xi',\eta)}=\abs{\nabla_\xi \Phi(\xi,\eta) \cdot (\xi'-\xi)+O(\abs{\xi'-\xi}^2)}.	
\end{equation}
Suppose that $\xi'-\xi=x_1 \mao{\nabla_\xi\Phi(\xi,\eta)}+x_2 (\mao{\nabla_\xi\Phi(\xi,\eta)})^\perp$. Then
\[
2^k\abs{x_1} \lle 2^l+\abs{\xi'-\xi}^2,
\]
which gives
\[
\abs{x_1} \lle 2^{-10}\abs{\xi'-\xi}.
\]
We now have
\begin{align*}
&V \cdot \nabla_\eta(\Phi(\xi',\eta)-\Phi(\xi,\eta))\\
={} &V \cdot [\nabla^2_{\xi,\eta} \Phi(\xi,\eta) \cdot (\xi'-\xi)+O(2^{-k_1}\abs{\xi'-\xi}^2)]\\
={} &x_2 \mao{\Upsilon}(\xi,\eta)+O(\abs{x_1}+2^{-k_1}\abs{\xi'-\xi}^2),
\end{align*}
which, together with \lemmaref{upsilon}, yields
\[
\abs{V \cdot \nabla_\eta(\Phi(\xi',\eta)-\Phi(\xi,\eta))} \gge \abs{\xi'-\xi}.
\]
Then we perform integration by parts with respect to $V$. Since the gains and losses are $2^m \abs{\xi'-\xi}$ and $2^{-k_1}, 2^{j_1}, 2^{-l}\abs{\xi'-\xi}$, the integration by parts can be performed effectively. (See Section~\ref{integrationbyparts} for an explanation of integration by parts.) Indeed, when $V$ hits $\fai_l(\Phi(\xi',\eta))$,
\begin{align*}
V \cdot \nabla_\eta \fai_l(\Phi(\xi',\eta))&=2^{-l}\fai'(\Phi(\xi',\eta)/2^l) V \cdot \nabla_\eta \Phi(\xi',\eta)\\
&=2^{-l}\fai'(\Phi(\xi',\eta)/2^l)V \cdot (\nabla_\eta \Phi(\xi',\eta)-\nabla_\eta \Phi(\xi,\eta))\\
&=O(2^{-l}\abs{\xi'-\xi}).
\end{align*}

\subsection{Intermediate and Low Frequencies}\label{interlowfre}
In this subsection, we aim to estimate
\[
\norm{P_{\le 100} (\Lambda^\delta+\Lambda^{N_0}) \Omega^{[0,N_1]} U}_2.
\]
Since there is no issue of loss of derivatives, we use the dispersive formulation. By \eqref{disperformula}, it suffices to prove that
\[
\sum_{k \le 100,m \geq 0} 2^{\delta k}\norm{P_k\BB_m(V^{\iota_1,p_1},V^{\iota_2,p_2})}_2 \lle \epsl_1^22^{\delta^4 L},
\]
where the cubic terms and the remainders are straightforward to handle (recall the notation in Section~\ref{multilinear}). The cases $m \le D^4$ and $m=L+1$ are easy to deal with. We decompose $\BB_m(V^{\iota_1,p_1},V^{\iota_2,p_2})$ into $\BB_m(\va,\vb)$ and we may assume that $p_1 \le p_2$. The cases $\cl k \ge \delta^3 m$ and $\ul k \le -2m$ are easy to discard. We next consider two cases: $k_1>-200$ and $k_1 \le -200$.

\subsubsection{Case 1: $k_1>-200$}\label{subsec:3.4.1}
We continue to decompose $\BB_m(\va,\vb)$ into\\ $\BB_{m,l}(\va,\vb)$, where $l \ge l_0:=\flo{-m/50}$. When $l>l_0$, integration by parts in time gives the desired estimate via \lemmaref{symbound}(iii). Therefore, we focus on $\BB_{m,l_0}(\va,\vb)$. We next decompose $\va$ into $\Vaa$, and we aim to exclude the case $j_1 \ge (1-1/200)m$. By Schur's test, \lemmaref{sizeangle}(i) yields
\[
\norm{P_k\BB_{m,l_0}(\Vaa,\vb)}_2 \lle 2^m2^{\delta^2 m}2^{l_0/2}2^{-n_1/2}\sup_{s \simeq 2^m}\norm{\mao{\Vaa}(s)}_{L^2(r \dif r)L^\infty_\theta}\norm{\vb(s)}_2,
\]
which gives the desired bound. We can now assume that $j_1 \le (1-1/200)m$. When $n_1 \le D$, \lemmaref{decaylemma} yields the optimal decay
\[
\norm{E_{\iota_1s_2}\Vaa(s_1)}_\infty \lle \epsl_1 2^{-m}2^{-5k_1}2^{-\delta j_1}
\]
for $s_2 \simeq s_1 \simeq 2^m$, which concludes this case.

When $n_1>D$, we will use almost orthogonality in time to get extra decay. Note that we now have $-300 \le \ul k \le \cl k \le 200$. Set
\[
Tf(\xi)=\psi_k(\xi)\int \me^{\mi s \Phi} q_m(s) \fai_{l_0}(\Phi)\gti m_q(\xi,\eta)\mao{\Vaa}(\xi-\eta,s)\psi_{k_2}(\eta)f(\eta,s) \dif \eta \dif s,
\]
where $f=\mao{\Vb}$. It suffices to show that
\[
\norm{Tf}_2 \lle \epsl_1^2 2^{-\delta j_1}2^{\delta^4 m}.
\]
Take a cutoff function $\fai_{>\log_2 \kappa_0}(\abs{\xi}-\rho_{\iota_1\iota_2})$, where $\kappa_0=2^{-400\delta n_1}$, to divide $T$ into $T_{>\kappa_0}$ and $T_{\le \kappa_0}$.

\noindent\ul{1. $T_{>\kappa_0}$-part.}\\
As a consequence of \lemmaref{propertyp} and \lemmaref{reso:almostortho}, we have $(\iota_1,\iota_2)=(+,+)$ or $(-,+)$, $\abs{\xi}-\rho_{\iota_1\iota_2} \ge 0$, and $\abs{\mao \xi^\perp \cdot \mao \eta} \ge C_0\kappa_0^{1/2}$. Choose a smooth function $\zeta$ on $\bR$ such that $\zeta$ is supported in $[1/2,\infty)$ and $\zeta=1$ on $[3/2,\infty)$. Let $\zeta_{\gge \kappa_0^{1/2}}(x):=\zeta(\frac{x}{C_0\kappa_0^{1/2}/10})$ and $\zeta_{\lle -\kappa_0^{1/2}}(x):=\zeta(-\frac{x}{C_0\kappa_0^{1/2}/10})$.
We insert cutoff functions $\zeta_{\gge \kappa_0^{1/2}}(\mao \xi^\perp \cdot \mao \eta)$ and $\zeta_{\lle -\kappa_0^{1/2}}(\mao \xi^\perp \cdot \mao \eta)$ into $T_{>\kappa_0}$ to obtain $T^{\pm}_{>\kappa_0}$, and we will address only $T^+_{>\kappa_0}$. Take a cutoff function $\chi \in C_0^\infty((-2,2))$ such that $\sum_{v \in \bZ} \chi(s+v)=1$ and let $\chi_v(s):=\chi(s/(2^m \kappa)+v)$, where $\kappa=2^{-100\delta n_1}$. We then decompose $T^+_{>\kappa_0}f$ as $\sum_v T^+_{>\kappa_0,v}f$, where
\begin{align*}
T^+_{>\kappa_0,v}f(\xi)={}&\fai_{>\log_2 \kappa_0}(\abs{\xi}-\rho_{\iota_1\iota_2})\psi_k(\xi)\int \me^{\mi s \Phi} q_m(s) \chi_v(s) \fai_{l_0}(\Phi)\gti m_q(\xi,\eta)\mao{\Vaa}(\xi-\eta,s)\\
&\times \zeta_{\gge \kappa_0^{1/2}}(\mao \xi^\perp \cdot \mao \eta)\psi_{k_2}(\eta)f(\eta,s) \dif \eta \dif s.
\end{align*}
We aim to show that $T^+_{>\kappa_0,v}f$ are almost orthogonal, that is, $(T^+_{>\kappa_0,v_i}f,T^+_{>\kappa_0,v_j}f)$ is negligible when $\abs{v_i-v_j} \ge 100$. Define
\begin{align*}
	S^+_{>\kappa_0,s}g(\xi)={}&\fai_{>\log_2 \kappa_0}(\abs{\xi}-\rho_{\iota_1\iota_2})\psi_k(\xi)\int \me^{\mi s \Phi} \fai_{l_0}(\Phi)\gti m_q(\xi,\eta)\mao{\Vaa}(\xi-\eta,s)\\
&\times \zeta_{\gge \kappa_0^{1/2}}(\mao \xi^\perp \cdot \mao \eta)\psi_{k_2}(\eta)g(\eta) \dif \eta.
\end{align*}
Then
\[
T^+_{>\kappa_0,v}f=\int q_m(s) \chi_v(s) S^+_{>\kappa_0,s} f(s) \dif s.
\]
The kernel of $(S^+_{>\kappa_0,s'})^*S^+_{>\kappa_0,s}$ is
\begin{align*}
K^+_{>\kappa_0,s,s'}(\eta,\eta')={}&\int \me^{\mi (s\Phi(\xi,\eta)-s'\Phi(\xi,\eta'))}  \fai^2_{>\log_2 \kappa_0}(\abs{\xi}-\rho_{\iota_1\iota_2})\psi_k^2(\xi) \fai_{l_0}(\Phi(\xi,\eta))\fai_{l_0}(\Phi(\xi,\eta'))\\
& \times  \gti m_q(\xi,\eta) \cl{\gti m_q(\xi,\eta')} \mao{\Vaa}(\xi-\eta,s)\cl{\mao{\Vaa}(\xi-\eta',s')} \\
&\times \zeta_{\gge \kappa_0^{1/2}}(\mao \xi^\perp \cdot \mao \eta)\zeta_{\gge \kappa_0^{1/2}}(\mao \xi^\perp \cdot \mao{\eta'})\psi_{k_2}(\eta) \psi_{k_2}(\eta') \dif \xi,
\end{align*}
where $s$ (resp. $s'$) lies in the support of $q_m \chi_{v_i}$ (resp. $q_m \chi_{v_j}$). We have, in view of \lemmaref{reso:almostortho},
\[
	s\nabla_\xi \Phi(\xi,\eta)-s'\nabla_\xi \Phi(\xi,\eta')
	=(s-s')\nabla_\xi \Phi(\xi,p^+_{\iota_1\iota_2}(\xi))+O[2^m(2^{l_0/2}+2^{-n_1/2})],
\]
and hence
\[
\abs{s\nabla_\xi \Phi(\xi,\eta)-s'\nabla_\xi \Phi(\xi,\eta')} \gge 2^m \kappa.
\]
We can now perform integration by parts in $\xi$ since the gains and losses are $2^m\kappa$ and $\kappa^{-1}, \kappa_0^{-1}, 2^{-l_0}, 2^{j_1}$. Hence we have
\[
\abs{(T^+_{>\kappa_0,v_i}f,T^+_{>\kappa_0,v_j}f)} \lle \epsl_1^4 2^{-100m}.
\]
Consequently,
\begin{equation}\label{eq:3.4.10}
	\norm{T^+_{>\kappa_0}f}_2^2 \lle \sum_v \norm{T^+_{>\kappa_0,v}f}_2^2+\epsl_1^4 2^{-50m}.
\end{equation}
To estimate $T^+_{>\kappa_0,v}f$, it suffices to deal with $S^+_{>\kappa_0,s}$. For later use, we define
\begin{align*}
	S_{>\kappa_0,s}g(\xi)={}&\fai_{>\log_2 \kappa_0}(\abs{\xi}-\rho_{\iota_1\iota_2})\psi_k(\xi)\int \me^{\mi s \Phi} \fai_{l_0}(\Phi)\gti m_q(\xi,\eta)\mao{\Vaa}(\xi-\eta,s)\\
&\times \psi_{k_2}(\eta)g(\eta) \dif \eta.
\end{align*}
The cutoff function $\zeta_{\gge \kappa_0^{1/2}}(\mao \xi^\perp \cdot \mao \eta)$ obstructs a direct application of the multiplier theorem. In contrast, we can apply the multiplier theorem to $S_{>\kappa_0,s}$. We will relate $	S^+_{>\kappa_0,s}$ to $S_{>\kappa_0,s}$ by exploiting the almost orthogonality between $S^+_{>\kappa_0,s}g$ and $S^-_{>\kappa_0,s}g$. Our goal is to prove that $(S^+_{>\kappa_0,s}g,S^-_{>\kappa_0,s}g)$ is negligible. It suffices to estimate the kernel of $(S^-_{>\kappa_0,s})^*S^+_{>\kappa_0,s}$:
\begin{align*}
	K^{-+}_{>\kappa_0,s}(\eta,\eta')={}&\int \me^{\mi s(\Phi(\xi,\eta)-\Phi(\xi,\eta'))} \fai^2_{>\log_2\kappa_0}(\abs{\xi}-\rho_{\iota_1\iota_2})\psi_k^2(\xi) \fai_{l_0}(\Phi(\xi,\eta))\fai_{l_0}(\Phi(\xi,\eta'))\\ &\times \gti m_q(\xi,\eta) \cl{\gti m_q(\xi,\eta')} \mao{\Vaa}(\xi-\eta,s)\cl{\mao{\Vaa}(\xi-\eta',s)}\\  &\times \zeta_{\gge \kappa_0^{1/2}}(\mao \xi^\perp \cdot \mao \eta)\zeta_{\lle -\kappa_0^{1/2}}(\mao \xi^\perp \cdot \mao{\eta'})\psi_{k_2}(\eta) \psi_{k_2}(\eta') \dif \xi.
\end{align*}
It follows from \lemmaref{propertyp} and \lemmaref{reso:almostortho} that
\[
\abs{p^+_{\iota_1\iota_2}(\xi)-p^-_{\iota_1\iota_2}(\xi)} \gge \kappa_0,
\]
and that
\[
s(\nabla_\xi \Phi(\xi,\eta)-\nabla_\xi\Phi(\xi,\eta'))= s(\nabla_\xi \Phi(\xi,p^+_{\iota_1\iota_2}(\xi))-\nabla_\xi\Phi(\xi,p^-_{\iota_1\iota_2}(\xi)))+O[2^m(2^{l_0/2}+2^{-n_1/2})].
\]
Therefore,
\[
\abs{s(\nabla_\xi \Phi(\xi,\eta)-\nabla_\xi\Phi(\xi,\eta'))} \gge 2^m\kappa_0^{1/2}.
\]
Integration by parts in $\xi$ now shows that the kernel is negligible as the gains and losses are $2^m \kappa_0$ and $\kappa_0^{-1}, 2^{-l_0}, 2^{j_1}$. Thus, we obtain
\[
\norm{S_{>\kappa_0,s}g}_2^2=\norm{S^+_{>\kappa_0,s}g}_2^2+\norm{S^-_{>\kappa_0,s}g}_2^2+O(\epsl_1^22^{-100m}\norm{g}_2^2).
\]
By an averaging argument and \lemmaref{decaylemma}, we have
\[
\norm{S_{>\kappa_0,s}g}_2 \lle \epsl_1 2^{-m}2^{-\delta j_1}2^{5\delta n_1}\norm{g}_2,
\]
which gives
\begin{equation}\label{eq:3.4.11}
	\norm{S^+_{>\kappa_0,s}g}_2 \lle \epsl_1 2^{-m}2^{-\delta j_1}2^{5\delta n_1}\norm{g}_2.
\end{equation}

Using \eqref{eq:3.4.11}, we have
\[
\norm{T^+_{>\kappa_0,v}f}_2^2 \lle \epsl_1^4 2^{2m}\kappa^2 2^{-2m}2^{-2\delta j_1}2^{10\delta n_1}2^{2\delta^4 m} \lle \epsl_1^4 \kappa^2 2^{-2\delta j_1}2^{10\delta n_1}2^{2\delta^4 m}.
\]
Then it follows from \eqref{eq:3.4.10} that
\[
\norm{T^+_{>\kappa_0}f}_2^2 \lle \epsl_1^4 \kappa^{-1} \kappa^2 2^{-2\delta j_1}2^{10\delta n_1} 2^{2\delta^4 m} \lle \epsl_1^4 2^{-2\delta j_1} 2^{2\delta^4 m},
\]
which completes the proof of the $T_{>\kappa_0}$-part.

\noindent\ul{2. $T_{\le \kappa_0}$-part.}\\
For the $T_{\le \kappa_0}$-part, we do not distinguish between $p^+_{\iota_1\iota_2}$ and $p^-_{\iota_1\iota_2}$, so we do not decompose $T_{\le \kappa_0}$ into $T^+_{\le \kappa_0}$ and $T^-_{\le \kappa_0}$. As in the $T_{>\kappa_0}$-part, take $\kappa=2^{-100\delta n_1}$, and we decompose $T_{\le \kappa_0}f$ into $T_{\le \kappa_0,v}f$. We aim to prove that $T_{\le \kappa_0,v}f$ are almost orthogonal. The kernel of $(S_{\le \kappa_0,s'})^*S_{\le\kappa_0,s}$ is
\begin{align*}
K_{\le \kappa_0,s,s'}(\eta,\eta')={}&\int \me^{\mi (s\Phi(\xi,\eta)-s'\Phi(\xi,\eta'))}  \fai^2_{\le \log_2 \kappa_0}(\abs{\xi}-\rho_{\iota_1\iota_2})\psi_k^2(\xi) \fai_{l_0}(\Phi(\xi,\eta))\fai_{l_0}(\Phi(\xi,\eta'))\\
& \times  \gti m_q(\xi,\eta) \cl{\gti m_q(\xi,\eta')} \mao{\Vaa}(\xi-\eta,s)\cl{\mao{\Vaa}(\xi-\eta',s')} \\
&\times  \psi_{k_2}(\eta) \psi_{k_2}(\eta') \dif \xi,
\end{align*}
where $s$ (resp. $s'$) lies in the support of $q_m \chi_{v_i}$ (resp. $q_m \chi_{v_j}$) and $\abs{v_i-v_j} \ge 100$. As a consequence of \lemmaref{propertyp} and \lemmaref{reso:almostortho}, we have
\[
s\nabla_\xi\Phi(\xi,\eta)-s'\nabla_\xi\Phi(\xi,\eta')=(s-s')\nabla_\xi \Phi(\xi,p^+_{\iota_1\iota_2}(\xi))+O[2^m(2^{l_0/2}+2^{-n_1/2}+\kappa_0^{1/2})],
\]
which yields
\[
\abs{s\nabla_\xi\Phi(\xi,\eta)-s'\nabla_\xi\Phi(\xi,\eta')} \gge 2^m \kappa.
\]
Hence we can perform integration by parts in $\xi$ effectively; the remaining arguments are analogous to those for the $T_{>\kappa_0}$-part.

\subsubsection{Case 2: $k_1 \le -200$}\label{subsec:3.4.2}
It is clear that $\abs{\Phi} \gge 1$. Hence we can insert localizations of $\Phi$ with parameter $l \gge 0$. Note that \lemmaref{symbound}(iii) yields $\norm{\psi_l(\Phi(\xi,\eta))}_{S_{k,k_1,k_2}} \lle 1$. Integration by parts in time yields
\[
\BB_{m,l}(\va,\vb)=\II'_{m,l}(\va,\vb)+\BB'_{m,l}(\ppd s \va,\vb)+\BB'_{m,l}(\va,\ppd s \vb).
\]
The terms $\II'_{m,l}(\va,\vb)$ and $\BB'_{m,l}(\ppd s \va,\vb)$ are easy to handle. Neglecting the cubic terms and the remainders, we consider the trilinear operator
\[
2^{-l}\TT_{\gti n,m,l}(\va,\vvc,\vx 4),
\]
where $\gti n$ is its associated multiplier, and where $p_3 \ge p_4$ and $p_1, p_4 \le N_1/2$. In what follows, we omit the multiplier $\gti n$ from the notation. It is easy to exclude the cases $k_3 \vee k_4 \ge 2\delta^3 m$ and $k_3 \wedge k_4 \le -2m$. Moreover, by \eqref{decaycor} and Bernstein's inequality, we obtain
\[
\norm{E_{\iota_1s}\va(s)}_\infty \lle \epsl_1 [2^{-(1-20\delta)m}2^{-(1-39\delta)k_1} \wedge 2^{(1-\delta)k_1}]
\]
for $s \simeq 2^m$.
Therefore, the cases $k_1 \le -(1/2+100\delta)m$ and $k_1 \ge -(1/2-100\delta)m$ can be treated. Analogously, we can assume that $-(1/2+100\delta)m \le k_4 \le -(1/2-100\delta)m$. Furthermore, we can assume that $\cl k \ge -100\delta m$ with $\cl k=k_1 \vee k_3 \vee k_4$.

We next insert localizations of the second resonance function $\wan \Phi$ with parameter $p \ge p_0=\flo{ -(1/2+1000\delta)m}$. For the case $p=p_0$, note that
\begin{align*}
\wan \Phi &=-\Lambda(\xi)+\iota_1 \Lambda(\xi-\eta)+\iota_3 \Lambda(\eta-\sigma)+\iota_4 \Lambda(\sigma)\\
&=-\Lambda(\xi)+\iota_3 \Lambda(\eta-\sigma)+\iota_1+\iota_4+O(2^{-2m/3}).
\end{align*}
It follows that
\[
\abs{-\Lambda(\xi)+\iota_3 \Lambda(\eta-\sigma)+\iota_1+\iota_4} \lle 2^{p_0}.
\]
Then using Schur's test, we estimate
\[
\norm{P_k\TT_{m,l,p_0}(\va,\vvc,\vx 4)}_2 \lle \epsl_1^2 2^m2^{5\delta m} 2^{p_0-\cl k}2^{k_{14}-\cl k} \sup_{s \simeq 2^m}\norm{\vvc}_2.
\]
When $p>p_0$, integration by parts in time yields 
\begin{align*}
&\TT_{m,l,p}(\va,\vvc,\vx 4)=\II'_{m,l,p}(\va,\vvc,\vx 4)+\TT'_{m,l,p}(\ppd s\va,\vvc,\vx 4)\\
&+\TT'_{m,l,p}(\va,\ppd s\vvc,\vx 4)+\TT'_{m,l,p}(\va,\vvc,\ppd s\vx 4).
\end{align*}
The term $\II'_{m,l,p}(\va,\vvc,\vx 4)$ is easy to handle. For $\TT'_{m,l,p}(\ppd s\va,\vvc,\vx 4)$, applying an averaging argument to $\fai^{p_0}_p(\wan \Phi)$, we estimate
\begin{align*}
&\norm{P_k\TT'_{m,l,p}(\ppd s\va,\vvc,\vx 4)}_2\\ \lle{} &2^m2^{-p}(2^{k_1}\sup_{s \simeq 2^m}\norm{\ppd s \va}_2\norm{\vvc}_2\norm{E_{[2^{m-2},2^{m+2}]}\vx 4}_\infty+\epsl_1^3 2^{-10m}).
\end{align*}
The term $\TT'_{m,l,p}(\va,\vvc,\ppd s\vx 4)$ can be treated in a similar way. For $\TT'_{m,l,p}(\va,$ $\ppd s\vvc,\vx 4)$, by neglecting the cubic terms and the remainders, we can write it as a quadrilinear operator $2^{-p}\QQ_{\gti p,m,l,p}(\va,\vx 5,\vx 6,\vx 4)$, where $\gti p$ is its associated multiplier, and where $p_5 \ge p_6$ and $p_1, p_6, p_4 \le N_1/2$. We use $\xi_1$, $\xi_5$, $\xi_6$, and $\xi_4$ to denote the frequency variables of $\va$, $\vx 5$, $\vx 6$, and $\vx 4$, respectively. It is straightforward to discard the cases $k_5 \vee k_6 \ge 3\delta^3 m$ and $k_5 \wedge k_6 \le -2m$. If $k_6 \ge -m/3$, then it can be handled by applying an averaging argument to $\fai^{p_0}_p(\wan \Phi)$. Moreover, we may assume that $p \le -10\delta m$. If $k_6 \le -m/3$, then the third resonance function $\Phi^{1,5,6,4}_{+\iota_1\iota_5\iota_6\iota_4}=\Phi^{1,56,4}_{+\iota_1\iota_3\iota_4}+\Phi^{5,6}_{\iota_3\iota_5\iota_6}$ satisfies $\abs{\Phi^{1,5,6,4}_{+\iota_1\iota_5\iota_6\iota_4}} \gge 1$, and integration by parts in time, together with averaging arguments, concludes the proof.

\subsection{General Korteweg Tensors}\label{gekotensors}
In this subsection, we sketch the proof of the energy estimate for general Korteweg tensors. Under the bootstrap assumptions analogous to \eqref{energyest}-\eqref{zest}, we prove the corresponding improved energy estimate.

The estimate for the low-frequency component $P_{\le 100}U$ follows from the formulation for dispersive estimates derived in Section~\ref{dispersiveformu}. We only consider the bilinear operator $\BB_{\gti m_q}(\va,\vb)$. If $\ul k > -200$, then an $L^2 \times L^\infty$ argument yields the desired bound. It remains to deal with the case $\ul k \le -200$. If $k_1 \simeq k_2$, then it can be handled by an $L^2 \times L^\infty$ argument. If $k_1 \not\simeq k_2$, then $\cl k+\ul k \ll 0$, and integration by parts in time completes the argument.

The estimate for the high-frequency component $P_{> 100}U$ follows from a paralinearization similar to that in Section~\ref{paralinear}, but without requiring the depletion factors. We can obtain
\begin{equation}\label{energymain10}
	\hti U_t-\mi T_\Sigma \hti U+T_{R\psi \cdot \mi \xi} \hti U=\hti Q+\hti C,
\end{equation}
where
\begin{itemize}
	\item $\hti Q=T_{\gti m}(U^{\iota_1},U^{\iota_2})$;
	\item $\Omega_{\xi,\eta}\gti m=0$;
	\item $\norm{\gti m}_{S_{k,k_1,k_2}} \lle 2^{k_1 \wedge k_2}2^{5(k_1 \wedge k_2)_+}$;
	\item $\norm{P_{\ge 0}\Lambda^{N_0}\hti C(t)}_2+\norm{P_{\ge 0} \Lambda^\delta \Omega^{N_1} \hti C(t)}_2 \lle \varepsilon_1^3 \spa{t}^{-1-1/100}$.
\end{itemize}
After commuting \eqref{energymain10} with $P_{\ge 50}T_{\Sigma^{N_0/2}}$ and $P_{\ge 50}T_{\Sigma^{\delta/2}}\Omega^{N_1}$, we have
\begin{equation}\label{energymain11}
	\hti W_t-\mi T_\Sigma \hti W+T_{R\psi \cdot \mi \xi} \hti W=\hti Q_{\hti W}+\hti C_{\hti W},
\end{equation}
where
\begin{itemize}
	\item $\hti W \in\set{P_{\ge 50}T_{\Sigma^{N_0/2}} \hti U,P_{\ge 50}T_{\Sigma^{\delta/2}}\Omega^{N_1} \hti U}$;
	\item $\hti Q_{\hti W}=T_{\gti m}(G,W^1)+T_{\gti m'}(G,W^1)$ where
	\[
	G \in \set{\spa{\nabla}^{10/\delta} \Omega^{[0,N_1/2]} U^{\pm}}
	\]
	and
	\[
	W^1 \in \set{T_{\Sigma^{N_0/2}} \hti U^{\pm},T_{\Sigma^{\delta/2}}\Omega^{N_1} \hti U^{\pm}, \spa{\nabla}^{10/\delta} \Omega^{[0,N_1-1]} U^{\pm}};
	\]
	\item $\norm{\gti m}_{S_{k,k_1,k_2}} \lle2^{k_1 \wedge k_2}2^{5(k_1 \wedge k_2)_+}$;
	\item $\norm{\gti m'}_{S_{k,k_1,k_2}} \lle 2^{-2\delta^{-1}(k_1 \vee k_2)_+}$;
	\item $\norm{\hti C_{\hti W}(t)}_2 \lle \varepsilon_1^3 \spa{t}^{-1-1/100}$.
\end{itemize}
The $\gti m$-terms can be handled by a direct energy estimate. For the $\gti m'$-terms, as in Section~\ref{sec:badfre}, we can assume $\cl k \le \delta m$. If $\ul k+\cl k \gge 0$, then $2^{-\cl k_+} \lle 2^{\ul k}$ shows that the factor $2^{-\cl k_+}$ provides the null structure needed for the decay estimates. If $\ul k+\cl k \ll 0$, then integration by parts in time gives the desired estimate since the derivative loss is harmless.

\section{Dispersive Estimates}
In this section, we estimate the $Z$-norms. By \eqref{disperformula}, we aim to prove that
\begin{gather*}
\sum_m 2^{10k_+}2^{\delta k}2^{(1-20\delta)(j+k)}2^{-(1/2-18\delta)n}\norm{A^{(j)}_nQ_{jk}\BB_m(V^{\iota_1,p_1},V^{\iota_2,p_2})}_2 \lle \epsl_1^2,\\
\sum_m  2^{10k_+}2^{\delta k}2^{(1-20\delta)(j+k)}2^{-(1/2-18\delta)n}\norm{A^{(j)}_nQ_{jk}\TT_m(V^{\iota_1,p_1},V^{\iota_2,p_2},V^{\iota_3,p_3})}_2 \lle \epsl_1^2,
\end{gather*}
and the remainder estimates (recall the notation in Section~\ref{multilinear}).
The cases $m \le D^4$ and $m=L+1$ are easy to deal with. It suffices to prove that
\begin{gather*}
	  2^{10k_+}2^{\delta k}2^{(1-20\delta)(j+k)}2^{-(1/2-18\delta)n}\norm{A^{(j)}_nQ_{jk}\BB_m(V^{\iota_1,p_1},V^{\iota_2,p_2})}_2 \lle \epsl_1^2 2^{-\delta^4 m},\\
  2^{10k_+}2^{\delta k}2^{(1-20\delta)(j+k)}2^{-(1/2-18\delta)n}\norm{A^{(j)}_nQ_{jk}\TT_m(V^{\iota_1,p_1},V^{\iota_2,p_2},V^{\iota_3,p_3})}_2 \lle \epsl_1^2 2^{-\delta^4 m}.
\end{gather*}
We first decompose the quadratic terms into $\BB_{m}(\va,\vb)$. The cases $k_1 \vee k_2 \ge \delta^3(m+j)$ and $k_1 \wedge k_2 \le -3(m+j)$ are easy to treat (see the proof of \lemmaref{timederi2}). Then the case $j \ge 10m$ can be handled by the approximate-finite-speed-of-propagation argument (which will be introduced in Section~\ref{finitespeed}). We continue to decompose $\BB_{m}(\va,\vb)$ into $\BB_{m,l}(\va,\vb)$ 
where $l \ge l_0=\flo{-(1-\delta/2)m}$ is the localization parameter of the modulation $\Phi$. If $l>-100\delta m$, then integration by parts in time yields
\[
\BB_{m,l}(\va,\vb)=\II'_{m,l}(\va,\vb)+\BB'_{m,l}(\ppd s \va,\vb)+\BB'_{m,l}(\va,\ppd s \vb).
\]
The terms $\BB'_{m,l}(\ppd s \va,\vb)$ and $\BB'_{m,l}( \va, \ppd s\vb)$ become cubic terms and higher-order terms and will be handled in Section~\ref{section:cubic} and Section~\ref{sec:remainder}.

\subsection{Approximate Finite Speed of Propagation}\label{finitespeed}
In this subsection, we aim to exclude the case $j \ge m+\cl k+D$ for $\BB_{m,l}(\va,\vb)$ with $l_0 \le l \le -100\delta m$ and $\II'_{m,l}(\va,\vb)$ with $l>-100\delta m$. The term $\II'_{m,l}(\va,\vb)$ is easy to handle. For $\BB_{m,l}(\va,\vb)$, we note that $\ul k \ge -\delta^2 m$. We decompose $\BB_{m,l}(\va,\vb)$ into $\BB_{m,l}(\Vaa,\Vbb)$, and the case $j_1 \vee j_2 \ge 30m$ is easy to exclude. The approximate-finite-speed-of-propagation argument is based on writing $P_k\BB_{m,l}(\Vaa,\Vbb)$ in the form
\[
\int \me^{\mi s \Phi+\mi x \cdot \xi} q_m(s) \fai_l^{l_0}(\Phi) \psi_k(\xi)\gti m_q(\xi,\eta) \mao{\Vaa}(\xi-\eta,s) \mao{\Vbb}(\eta,s) \dif \xi\dif \eta \dif s.
\]
We may omit the cutoff function $\fai_l^{l_0}(\Phi)$ since it can be handled by an averaging argument. We then integrate by parts in $\xi$. Observe that $\nabla_\xi(s\Phi+x \cdot \xi)=x+s \Phi_{\xi}$ is bounded below whenever $\abs{x} \gg s \abs{\Phi_{\xi}}$. Thus, the integration by parts is effective provided $j_1 \wedge j_2 \le (1-\delta^2)j$ and $-k \le (1-\delta^2)j$. Hence, we now exclude the complementary cases.

\noindent\fbox{1. $j_1 \wedge j_2 \ge (1-\delta^2)j$.}\\
By \lemmaref{sizeest}(i) and (ii), we have
\begin{equation}\label{4.2.1}
	\begin{aligned}
&\norm{P_k\BB_{m,l}(\Vaa,\Vbb)}_2 \\
\lle{} &2^{10\delta^2 m}2^{\cl k}2^{10\cl k_+}2^m2^{-n_1/2}2^{-n_2/2}2^{l/2}\sup_{s \simeq 2^m}\norm{\mao{\Vaa}(s)}_{L^2(r\dif r)L_{\theta}^{\infty}} \norm{\Vbb(s)}_2,		
	\end{aligned}
\end{equation}
which yields the desired estimate.

\noindent\fbox{2. $-k \ge (1-\delta^2)j$.}\\
It is easy to see that this case contradicts $\ul k \ge -\delta^2 m$.

We can now assume that $j \le m+\cl k+D$. We next exclude the case $l\geq -\delta m$ for $\II'_{m,l}(\va,\vb)$. We may assume that $k_1 \ge k_2$. By \cororef{decaydecom}, for any $s_1 \simeq s \simeq 2^m$, $E_{\iota_1 s_1} \vx 1(s) = I^{p_1}_{k_1}(s,s_1) + II^{p_1}_{k_1}(s,s_1)$, where $\norm{I^{p_1}_{k_1}(s,s_1)}_{\infty} \lle \varepsilon_1 2^{-(1+\delta)k_1}2^{-(1-5\delta)m}$ and $\norm{II^{p_1}_{k_1}(s,s_1)}_{2} \lle \varepsilon_1 2^{-(5/6)k_1}2^{(-1/2-1/100)m}$. Hence, in view of an averaging argument, \eqref{energyest}, and \cororef{decayestlow},
\begin{align*}
	&2^{10k_+}2^{\delta k}2^{(1-20\delta)(j+k)}2^{-(1/2-18\delta)n}\norm{A^{(j)}_nQ_{jk}\II'_{m,l}(\va,\vb)}_2\\
	\lle{} & 2^{\delta^2 m} 2^{(1-20\delta)m}2^{-l}2^{\cl k}(\sup_{s \simeq s_1 \simeq 2^m}\norm{I^{p_1}_{k_1}(s,s_1)}_{\infty}\norm{\vb(s)}_{2}+\sup_{s \simeq s_1 \simeq s_2 \simeq 2^m}\norm{II^{p_1}_{k_1}(s,s_1)}_2\\
	&\times\norm{E_{\iota_2 s_2}\vb(s)}_{\infty} + \varepsilon_1^2 2^{-10m})\\
	\lle{} & \varepsilon_1^2 2^{\delta^2 m} 2^{(1-19\delta)m} (2^{-(1-5\delta)m} 2^{-\delta k_2} + 2^{k_1/6}2^{(-1/2-1/100)m} 2^{-(1/2-2\delta)m}+2^{-5m})\\
	\lle{} & \varepsilon_1^2 2^{-10\delta^4 m},
\end{align*}
(see Section~\ref{techlemma} for a discussion of the averaging argument).

We next exclude the non-parallel parts of $\BB_{m,l}(\va,\vb)$ and $\II'_{m,l}(\va,\vb)$. For $\BB_{m,l}(\va,\vb)$ with $l_0 \le l \le -100\delta m$, we insert a cutoff function $\fai_{>\log_2\kappa_\theta}(\Theta)$, with $\kappa_\theta=2^{-(1/2-2\delta^2)m}$, into it to obtain $\BB_{m,l,>\kappa_\theta}(\va,\vb)$, and we aim to exclude this type of term. We can use \lemmaref{rotainte} together with an averaging argument to deal with it. To apply \lemmaref{rotainte}, we need to exclude the case $j_1 \wedge j_2 \ge (1-2\delta^2)m$. As in \eqref{4.2.1}, this case can be handled by \lemmaref{sizeest}(i) and (ii). 
The term $\II'_{m,l,>\kappa_{\theta}}(\va,\vb)$ with $-100\delta m<l<-\delta m$ can be excluded via a similar argument. 

Thus, it remains to deal with $\BB_{m,l,\le \kappa_\theta}(\va,\vb)$ with $l_0 \le l \le -100\delta m$ and\\
$\II'_{m,l,\le \kappa_{\theta}}(\va,\vb)$ with $-100\delta m<l<-\delta m$. For $\BB_{m,l,\le \kappa_\theta}(\va,\vb)$, integration by parts in time yields
\begin{align*}
	&\BB_{m,l,\le \kappa_\theta}(\va,\vb)\\
	&=\II'_{m,l,\le \kappa_\theta}(\va,\vb)+\BB'_{m,l,\le \kappa_\theta}(\ppd s \va,\vb)+\BB'_{m,l,\le \kappa_\theta}(\va,\ppd s \vb).
\end{align*}
Therefore, it suffices to address
\begin{itemize}
	\item Small Modulation Cases and Boundary Terms:
	\[
	\BB_{m,l_0,\le \kappa_\theta}(\va,\vb) \text{ and } \II'_{m,l,\le \kappa_\theta}(\va,\vb), \ l_0<l<-\delta m;
	\]
	\item Intermediate Modulation Cases:
	\[
	\BB'_{m,l,\le \kappa_\theta}(\ppd s \va,\vb) \text{ and } \BB'_{m,l,\le \kappa_\theta}(\va,\ppd s \vb), \ l_0<l \le -100\delta m;
	\]
	\item Cubic Terms and Remainders.
\end{itemize}

\subsection{Small Modulation Cases and Boundary Terms}\label{smallboundary}
In this subsection, we address $\BB_{m,l_0,\le \kappa_\theta}(\va,\vb)$ and $\II'_{m,l,\le \kappa_\theta}(\va,\vb)$ with $l_0<l<-\delta m$.
 It suffices to show that
\begin{multline*}
2^{10k_+}2^{\delta k}2^{(1-20\delta)(j+k)}2^{-(1/2-18\delta)n}2^{\delta m/2}2^{-l}\norm{A^{(j)}_nQ_{jk}I_{\gti m_q,s,l,\le \kappa_\theta}(\va(s),\vb(s))}_2 \\ \lle \epsl_1^2 2^{-10\delta^4 m},	
\end{multline*}
where $s \in [2^{m-1},2^{m+1}]$ and $l_0 \le l<-\delta m$. Here the subscript $\gti m_q$ will be suppressed. Note that $\ul k \ge -\delta^2 m$.

\noindent\textbf{Step 1:} For $I_{s,l,\le \kappa_\theta}(\va(s),\vb(s))$, we decompose it into
\[
I_{s,l,\le \kappa_\theta}(\Vaa(s),\Vbb(s)).
\]
The case $j_1 \vee j_2 \ge (5/2)m$ is easy to treat. We next exclude the case $j_1 \vee j_2 \ge (1-3\delta)m$. We may assume that $j_2 \ge (1-3\delta)m$. We insert localizations of $\nabla_\xi \Phi$ with parameter $r_2 \ge r_2^0=\lfloor -(1/2-\delta)m \rfloor$ into $I_{s,l,\le \kappa_\theta}(\Vaa(s),\Vbb(s))$ to obtain $I_{s,l,\le \kappa_\theta,r_2}(\Vaa(s),$ $\Vbb(s))$. We first deal with the case $r_2 \ge -2\delta m$. In view of \lemmaref{sizeest}(iii) and (iv), we estimate
\begin{align*}
	&2^{10k_+}2^{\delta k}2^{(1-20\delta)(j+k)}2^{-(1/2-18\delta)n}2^{\delta m/2}2^{-l}\norm{A^{(j)}_nQ_{jk}I_{s,l,\le \kappa_\theta,r_2}(\Vaa(s),\Vbb(s))}_2\\
	\lle{} &2^{16\delta^2m}2^{(1-20\delta)j}2^{\delta m/2}2^{-l}2^l 2^{-r_2/2}\kappa_\theta^{1/2} \norm{\mao{\Vaa(s)}}_\infty \norm{\Vbb(s)}_2.
\end{align*}
The desired estimate follows if $n_2 \le \delta m$. If $n_2>\delta m$, then we have $\abs{\nabla_\eta \Phi} \gge 1$. Indeed, the condition $n_2>\delta m$ implies that $\abs \eta$ is close to $2\gamma_0$. Suppose, for a contradiction, that $\abs{\nabla_\eta \Phi} \ll 1$. Then $\abs{\xi-\eta}$ is close to $2\gamma_0$. The bound $\abs{\Theta} \lle \kappa_\theta$ implies that $\abs{\xi}$ is close to $4\gamma_0$ or $0$, which contradicts the restriction $l<-\delta m$. Using the bounds $\abs{\nabla_\eta \Phi} \gge 1$ and $r_2 \ge -2\delta m$, together with \lemmaref{sizeest}(iii), we obtain
\begin{align*}
	&2^{10k_+}2^{\delta k}2^{(1-20\delta)(j+k)}2^{-(1/2-18\delta)n}2^{\delta m/2}2^{-l}\norm{A^{(j)}_nQ_{jk}I_{s,l,\le \kappa_\theta,r_2}(\Vaa(s),\Vbb(s))}_2\\
	\lle{} &2^{16\delta^2m}2^{(1-20\delta)j}2^{\delta m/2}2^{-l}2^l 2^{-r_2/2}\kappa_\theta \norm{\mao{\Vaa(s)}}_\infty \norm{\Vbb(s)}_2.
\end{align*}
This gives the desired bound.

If $r_2 \le -2\delta m$, then \lemmaref{resoregion} implies that $|\xi|$, $|\xi-\eta|$, and $|\eta|$ are close to $\gamma_0$, $\gamma_0$, and $2\gamma_0$, respectively. We next employ the approximate-finite-speed-of-propagation argument, together with an averaging argument, to exclude the case $j \ge m+r_2+D$. To do so, it suffices to exclude the case $j_1 \ge (1-\delta^2)j$. Indeed, it follows from \lemmaref{sizeest}(i) that
\begin{align*}
	&2^{10k_+}2^{\delta k}2^{(1-20\delta)(j+k)}2^{-(1/2-18\delta)n}2^{\delta m/2}2^{-l}\norm{A^{(j)}_nQ_{jk}I_{s,l,\le \kappa_\theta,r_2}(\Vaa(s),\Vbb(s))}_2\\
\lle{} &2^{6\delta^2m}2^{(1-20\delta)j}2^{\delta m/2}2^{-l} 2^{-n_2/2}2^{l/2}\norm{\mao{\Vbb}(s)}_{L^2(r\dif r)L_{\theta}^{\infty}} \norm{\Vaa(s)}_2.
\end{align*}
Hence we can assume that $j \lle m+r_2$. If $r_2>r_2^0$, then, by \lemmaref{sizeest}(iii), we estimate
\begin{align*}
	&2^{10k_+}2^{\delta k}2^{(1-20\delta)(j+k)}2^{-(1/2-18\delta)n}2^{\delta m/2}2^{-l}\norm{A^{(j)}_nQ_{jk}I_{s,l,\le \kappa_\theta,r_2}(\Vaa(s),\Vbb(s))}_2\\
\lle{} &2^{6\delta^2m}2^{(1-20\delta)(m+r_2)}2^{\delta m/2}2^{-l} 2^l2^{-r_2/2}\kappa_\theta \norm{\mao{\Vaa}(s)}_\infty \norm{\Vbb(s)}_2,
\end{align*}
while if $r_2=r_2^0$, then, by \lemmaref{sizeest}(iii) and (iv), we have
\begin{align*}
	&2^{10k_+}2^{\delta k}2^{(1-20\delta)(j+k)}2^{-(1/2-18\delta)n}2^{\delta m/2}2^{-l}\norm{A^{(j)}_nQ_{jk}I_{s,l,\le \kappa_\theta,r_2}(\Vaa(s),\Vbb(s))}_2\\
\lle{} &2^{6\delta^2m}2^{(1-20\delta)(m+r_2)}2^{\delta m/2}2^{-l} 2^l\kappa_\theta^{1/2} \norm{\mao{\Vaa}(s)}_\infty \norm{\Vbb(s)}_2.
\end{align*}

\noindent\textbf{Step 2:} We can now assume that $j_1 \vee j_2 \le (1-3\delta)m$. Recall that we aim to deal with $I_{s,l,\le \kappa_\theta}(\Vaa(s),\Vbb(s))$. We insert a cutoff function $\fai_{>-\delta m}(\Xi)$ into it, and integration by parts in $\eta$, together with an averaging argument, shows that $I_{s,l,\le \kappa_\theta,>-\delta m}(\Vaa(s),\Vbb(s))$ is negligible. Hence in the following we focus on $I_{s,l,\le \kappa_\theta,\le -\delta m}(\Vaa(s),\Vbb(s))$. By \lemmaref{resoregion}, we deduce that $|\xi|$, $|\xi-\eta|$, and $|\eta|$ are close to $2\gamma_0$, $\gamma_0$, and $\gamma_0$. In view of \lemmaref{zproperty}(iv), our goal is to prove that
\begin{align*}
&2^{10k_+}2^{\delta k}2^{(1-20\delta)(j+k)}2^{\delta m/2}2^{-l}\\
&\times\sup_{0 \le n \le j+1}2^{-(1/2-18\delta)n}\norm{A^{(j)}_nP_kI_{s,l,\le \kappa_\theta,\le -\delta m}(\Vaa(s),\Vbb(s))}_2 \lle \epsl_1^2 2^{-10\delta^4 m}.
\end{align*}

We further exclude the case $j_1 \vee j_2 \ge (1/2+\delta^2)m$. We may assume that $j_2 \ge (1/2+\delta^2)m$. We insert localizations of $\nabla_\eta \Phi$ with parameter $r_1 \ge r_1^0=\flo{-(1/2-\delta)m}$ into $I_{s,l,\le \kappa_\theta,\le -\delta m}(\Vaa(s),\Vbb(s))$ to obtain $I_{s,l,\le \kappa_\theta,r_1}(\Vaa(s),\Vbb(s))$. As a consequence of \lemmaref{resoregion}, we have $2^{-n} \lle 2^{-j} \vee (2^l+2^{2r_1+2\delta^2m})$. If $l=-j \vee l \vee (2r_1+2\delta^2 m)$, then it follows from \lemmaref{sizeest}(iii) and \lemmaref{resoregion} that
\begin{align*}
		&2^{10k_+}2^{\delta k}2^{(1-20\delta)(j+k)}2^{-(1/2-18\delta)n}2^{\delta m/2}2^{-l}\norm{A^{(j)}_nP_kI_{s,l,\le \kappa_\theta,r_1}(\Vaa(s),\Vbb(s))}_2\\
\lle{} &2^{6\delta^2 m} 2^{(1-20\delta)j}2^{-(1/2+18\delta)l}2^{\delta m/2} 2^{l/2}2^{r_1/2}\kappa_\theta \norm{\mao{\Vaa}(s)}_\infty \norm{\Vbb(s)}_2.
\end{align*}
If $-j=-j \vee l \vee (2r_1+2\delta^2 m)$, then, in view of \lemmaref{sizeest}(iii) and (iv), we have
\begin{align*}
	&2^{10k_+}2^{\delta k}2^{(1-20\delta)(j+k)}2^{-(1/2-18\delta)n}2^{\delta m/2}2^{-l}\norm{A^{(j)}_nP_kI_{s,l,\le \kappa_\theta,r_1}(\Vaa(s),\Vbb(s))}_2\\
\lle{} &2^{16\delta^2m}2^{(1/2-2\delta)j}2^{\delta m/2}2^{-l}2^l \kappa_\theta^{1/2} \norm{\mao{\Vaa}(s)}_\infty \norm{\Vbb(s)}_2.
\end{align*}
For the case $2r_1+2\delta^2 m=-j \vee l \vee (2r_1+2\delta^2 m)$, if $r_1>r_1^0$, then, applying \lemmaref{sizeest}(iii), we estimate
\begin{align*}
	&2^{10k_+}2^{\delta k}2^{(1-20\delta)(j+k)}2^{-(1/2-18\delta)n}2^{\delta m/2}2^{-l}\norm{A^{(j)}_nP_kI_{s,l,\le \kappa_\theta,r_1}(\Vaa(s),\Vbb(s))}_2\\
\lle{} &2^{6\delta^2m} 2^{(1-20\delta)j}2^{(1/2-18\delta)(2r_1+2\delta^2 m)}2^{\delta m/2}2^{-l} 2^l2^{-r_1/2}\kappa_\theta \norm{\mao{\Vaa}(s)}_\infty \norm{\Vbb(s)}_2,
\end{align*}
while if $r_1=r_1^0$, then, by \lemmaref{sizeest}(iii) and (iv), we have
\begin{align*}
	&2^{10k_+}2^{\delta k}2^{(1-20\delta)(j+k)}2^{-(1/2-18\delta)n}2^{\delta m/2}2^{-l}\norm{A^{(j)}_nP_kI_{s,l,\le \kappa_\theta,r_1}(\Vaa(s),\Vbb(s))}_2\\
\lle{} &2^{16\delta^2m} 2^{(1-20\delta)j}2^{(1/2-18\delta)(-m+2\delta m+2\delta^2 m)}2^{\delta m/2}2^{-l} 2^l\kappa_\theta^{1/2} \norm{\mao{\Vaa}(s)}_\infty \norm{\Vbb(s)}_2.
\end{align*}

We can now assume that $j_1 \vee j_2 \le (1/2+\delta^2)m$. We insert a cutoff function \\$\fai_{>-(1/2-5\delta^2)m}(\Xi)$ into $I_{s,l,\le \kappa_\theta,\le -\delta m}(\Vaa(s),\Vbb(s))$. Integration by parts in $\eta$, together with an averaging argument, shows that the $>-(1/2-5\delta^2)m$-part is negligible. Finally, we estimate $I_{s,l,\le \kappa_\theta,\le -(1/2-5\delta^2)m}(\Vaa(s),\Vbb(s))$. In view of \lemmaref{resoregion}, we have $\abs{\abs{\xi}-2\gamma_0} \lle 2^l+2^{-2(1/2-5\delta^2)m+2\delta^2 m} \lle 2^l$, and hence $2^{-n} \lle 2^{-j} \vee 2^{l}$. The bound $\abs \Xi \lle 2^{-(1/2-5\delta^2)m}$ implies that $\abs{\eta-\xi/2} \lle 2^{-(1/2-5\delta^2)m}$. It then follows from \eqref{vinftyest1} that
\begin{align*}
	&2^{10k_+}2^{\delta k}2^{(1-20\delta)(j+k)}2^{-(1/2-18\delta)n}2^{\delta m/2}2^{-l}\\
	&\times \norm{A^{(j)}_nP_kI_{s,l,\le \kappa_\theta,\le -(1/2-5\delta^2)m}(\Vaa(s),\Vbb(s))}_2\\
\lle{} &2^{\delta^2m}(2^{(1/2-2\delta)j}\vee 2^{(1-20\delta)j}2^{(1/2-18\delta)l})2^{\delta m/2}2^{-l}2^{l/2} 2^{-2(1/2-5\delta^2)m} \\
	   &\times\norm{\mao{\Vaa}(s)}_\infty \norm{\mao{\Vbb}(s)}_\infty\\
\lle{} &\epsl_1^2 2^{2\delta^2m}(2^{(1/2-2\delta)m} \vee 2^{(1-20\delta)m+(1/2-18\delta)l})2^{\delta m/2}2^{-l/2} 2^{-(1-10\delta^2)m} \lle \epsl_1^2 2^{-10\delta^4 m}.
\end{align*}
This case accounts for the factor $2^{-(1/2-18\delta)n}$ in the definition of the $Z$-norm.

\subsection{Intermediate Modulation Cases}\label{sec:dispersiveinter}
We first prove the following lemma, whose proof is similar to that in Section~\ref{smallboundary}.
\begin{lemma}\label{interfirstpart}
	Assume that $m \ge D^4$. Suppose that one of the following two conditions holds. \\ \emph{Case 1.}
	\begin{gather*}
			\sup_{(k,j,n) \in \hti A} \Brac{2^{-3\delta^2 m}2^{-4\delta n} \norm{\mao{W^1_{k,j,n}}}_{\infty} + 2^{(1-21\delta)j}2^{-(1/2-18\delta)n}2^{-3\delta^2 m}\norm{\mao{W^1_{k,j,n}}}_{L^2(r\dif r)L_{\theta}^{\infty}} }\le 1 \\
		 \text{and } \norm{W^2}_2 \le 2^{-(3/2-13\delta)m}.	
	\end{gather*}
	\emph{Case 2.}
	\[
	2^{(1/2-5\delta)m}\norm{\mao{W^1}}_\infty+2^{(1-8\delta)m}\norm{\mao{W^1}}_{L^2(r\dif r)L^\infty_\theta} \le 1 \text{ and } \norm{\mao{W^2}}_{L^2(r\dif r)L^\infty_\theta} \le 2^{-(1-8\delta)m}.
	\]
	Then
	\[
	2^m2^{10 k_+}2^{\delta k}2^{(1-20\delta)(j+k)}2^{-(1/2-18\delta)n}2^{-l}\norm{A^{(j)}_nQ_{jk}I_{\gti m_q,m,l,\le \kappa_\theta}(W^1_{k_1},W^2_{k_2})}_2 \lle 2^{-10\delta^4 m},
	\]
	where the phase in $I$ is allowed to be any of $\Phi_{\iota_1\iota_2}$ with $(\iota_1,\iota_2) \in \set{+,-}^2$, and where $-\delta^2 m \le \ul k \le \cl k \le 2\delta^3 m$, $j \lle m$, $l_0<l \le -100\delta m$, and $\kappa_\theta=2^{-(1/2-2\delta^2)m}$.
\end{lemma}
\begin{proof}
In the following, we omit the subscript $\gti m_q$ in $I_{\gti m_q,m,l,\le \kappa_\theta}$.

\noindent\ul{Case 1.}\\
	We insert localizations of $\nabla_\eta \Phi$ and $\nabla_\xi \Phi$ with parameters $r_1 \ge r_1^0=\flo{-(1/2-\delta)m}$ and $r_2 \ge r_2^0=\flo{-(1/2-\delta)m}$ to obtain $I_{m,l,\le \kappa_\theta,r_1,r_2}(W^1_{k_1},W^2_{k_2})$. Then we decompose $I_{m,l,\le \kappa_\theta,r_1,r_2}(W^1_{k_1},W^2_{k_2})$ into $I_{m,l,\le \kappa_\theta,r_1,r_2}(W^1_{k_1,j_1,n_1},W^2_{k_2})$, and the case $j_1 \ge (4/3) m$ is easy to exclude. 

	If $r_1,r_2 \ge -2\delta^2 m$, then, by \lemmaref{sizeest}(iii), we estimate
	\begin{align*}
		&2^m2^{10 k_+}2^{\delta k}2^{(1-20\delta)(j+k)}2^{-(1/2-18\delta)n}2^{-l}\norm{A^{(j)}_nQ_{jk}I_{m,l,\le \kappa_\theta,r_1,r_2}(W^1_{k_1,j_1,n_1},W^2_{k_2})}_2\\
		\lle{} &2^m2^{6\delta^2 m}2^{(1-20\delta)j}2^{-l}2^l2^{-r_1/2}2^{-r_2/2}\kappa_\theta\norm{\mao{W^1_{k_1,j_1,n_1}}}_\infty \norm{W^2_{k_2}}_2.
	\end{align*}
	
	If $r_1 \le -2\delta^2 m$, then it suffices to prove that
	\[
	2^m2^{\delta^2 m}2^{(1-20\delta)j}\sup_{0 \le n \le j+1} 2^{-(1/2-18\delta)n} 2^{-l}\norm{A^{(j)}_nP_kI_{m,l,\le \kappa_\theta,r_1,r_2}(W^1_{k_1,j_1,n_1},W^2_{k_2})}_2 \lle 2^{-10\delta^4 m}.
	\]
	\lemmaref{resoregion} implies that $|\xi|$, $|\xi-\eta|$, and $|\eta|$ are close to $2\gamma_0$, $\gamma_0$, and $\gamma_0$, respectively, and that $2^{-n} \lle 2^{-j} \vee 2^l \vee 2^{2r_1+2\delta^2 m}$. If $l=-j \vee l \vee (2r_1+2\delta^2 m)$, then it follows from \lemmaref{sizeest}(iii) and \lemmaref{resoregion} that
	\begin{align*}
		&2^m2^{\delta^2 m}2^{(1-20\delta)j} 2^{-(1/2-18\delta)n} 2^{-l}\norm{A^{(j)}_nP_kI_{m,l,\le \kappa_\theta,r_1,r_2}(W^1_{k_1,j_1,n_1},W^2_{k_2})}_2\\
		\lle{} &2^m2^{7\delta^2 m}2^{(1-20\delta)j}2^{(1/2-18\delta)l} 2^{-l}2^{l/2}2^{r_1/2}\kappa_\theta \norm{\mao{W^1_{k_1,j_1,n_1}}}_\infty \norm{W^2_{k_2}}_2.
	\end{align*}
	If $-j=-j \vee l \vee (2r_1+2\delta^2 m)$, then, in view of \lemmaref{sizeest}(iii) and (iv), we estimate
	\begin{align*}
		&2^m2^{\delta^2 m}2^{(1-20\delta)j} 2^{-(1/2-18\delta)n} 2^{-l}\norm{A^{(j)}_nP_kI_{m,l,\le \kappa_\theta,r_1,r_2}(W^1_{k_1,j_1,n_1},W^2_{k_2})}_2\\
	\lle{} &2^m2^{17\delta^2 m}2^{(1/2-2\delta)j}2^{-l}2^l\kappa_\theta^{1/2} \norm{\mao{W^1_{k_1,j_1,n_1}}}_\infty \norm{W^2_{k_2}}_2.
	\end{align*}
	If $2r_1+2\delta^2 m=-j \vee l \vee (2r_1+2\delta^2 m)$, then \lemmaref{sizeest}(iii) yields 
	\begin{align*}
	&2^m2^{\delta^2 m}2^{(1-20\delta)j} 2^{-(1/2-18\delta)n} 2^{-l}\norm{A^{(j)}_nP_kI_{m,l,\le \kappa_\theta,r_1,r_2}(W^1_{k_1,j_1,n_1},W^2_{k_2})}_2\\
	\lle{} &2^m2^{10\delta^2 m}2^{(1-20\delta)j} 2^{(1-36\delta)r_1} 2^{-l}2^l 2^{-r_1/2}\kappa_\theta \norm{\mao{W^1_{k_1,j_1,n_1}}}_\infty \norm{W^2_{k_2}}_2
	\end{align*}
	if $r_1>r_1^0$, while \lemmaref{sizeest}(iii) and (iv) give the estimate
	\begin{align*}
	&2^m2^{\delta^2 m}2^{(1-20\delta)j} 2^{-(1/2-18\delta)n} 2^{-l}\norm{A^{(j)}_nP_kI_{m,l,\le \kappa_\theta,r_1,r_2}(W^1_{k_1,j_1,n_1},W^2_{k_2})}_2\\
	\lle{} &2^m2^{25\delta^2 m}2^{(1-20\delta)j} 2^{(1-36\delta)r_1} 2^{-l} 2^{l}\kappa_\theta^{1/2} \norm{\mao{W^1_{k_1,j_1,n_1}}}_\infty \norm{W^2_{k_2}}_2
	\end{align*}
	if $r_1=r_1^0$.
	
	If $r_2 \le -2\delta^2m$, then it follows from \lemmaref{resoregion} that $|\xi|$, $|\xi-\eta|$, and $|\eta|$ are close to $\gamma_0$, $\gamma_0$, and $2\gamma_0$, respectively. We first improve the approximate finite speed of propagation bound. When $j \ge m+r_2+D$, the approximate-finite-speed-of-propagation argument, together with an averaging argument, can handle the term if $j_1 \le (1-\delta^2)j$. When $j_1 \ge (1-\delta^2)j$, we apply Schur's test, proceeding as in the proof of \lemmaref{sizeest}(i): we control the radial range using the lower bound $\abs{\nabla_\eta \Phi} \gge 1$ and the angular range by $\kappa_\theta$. Therefore,
	\begin{equation}\label{eq:4.3.1}
		\begin{aligned}
		&2^m2^{10 k_+}2^{\delta k}2^{(1-20\delta)(j+k)}2^{-(1/2-18\delta)n}2^{-l}\norm{A^{(j)}_nQ_{jk}I_{m,l,\le \kappa_\theta,r_1,r_2}(W^1_{k_1,j_1,n_1},W^2_{k_2})}_2\\
	\lle{} &2^m2^{\delta^2 m}2^{(1-20\delta)j}2^{-l} 2^{l/4}
	\kappa_\theta\norm{\mao{W^1_{k_1,j_1,n_1}}}_{L^2(r\dif r)L^\infty_\theta}\norm{W^2_{k_2}}_2.
	\end{aligned}
	\end{equation}
		Hence we can assume that $j \lle m+r_2$. By \lemmaref{sizeest}(iii), we have
	\begin{align*}
		&2^m2^{10 k_+}2^{\delta k}2^{(1-20\delta)(j+k)}2^{-(1/2-18\delta)n}2^{-l}\norm{A^{(j)}_nQ_{jk}I_{m,l,\le \kappa_\theta,r_1,r_2}(W^1_{k_1,j_1,n_1},W^2_{k_2})}_2\\
	\lle{} &2^m2^{10\delta^2 m}2^{(1-20\delta)(m+r_2)} 2^{-l} 2^l2^{-r_2/2}\kappa_\theta\norm{\mao{W^1_{k_1,j_1,n_1}}}_\infty \norm{W^2_{k_2}}_2
	\end{align*}
	if $r_2>r_2^0$, and by \lemmaref{sizeest}(iii) and (iv), we estimate 
	\begin{align*}
		&2^m2^{10 k_+}2^{\delta k}2^{(1-20\delta)(j+k)}2^{-(1/2-18\delta)n}2^{-l}\norm{A^{(j)}_nQ_{jk}I_{m,l,\le \kappa_\theta,r_1,r_2}(W^1_{k_1,j_1,n_1},W^2_{k_2})}_2\\
	\lle{} &2^m2^{25\delta^2m}2^{(1-20\delta)(m+r_2)}2^{-l}2^l\kappa_\theta^{1/2}\norm{\mao{W^1_{k_1,j_1,n_1}}}_\infty \norm{W^2_{k_2}}_2
	\end{align*}
	if $r_2=r_2^0$.
	
\noindent\ul{Case 2.}\\
The proof is similar since $\norm{\mao{W^1}}_\infty \norm{W^2}_2$ is also bounded by $2^{-(3/2-13\delta)m}$. The main distinction arises in the case $r_2 \le -2\delta^2m$. In this case, analogous to \eqref{eq:4.3.1}, we estimate
\begin{align*}
	&2^m2^{10 k_+}2^{\delta k}2^{(1-20\delta)(j+k)}2^{-(1/2-18\delta)n}2^{-l}\norm{A^{(j)}_nQ_{jk}I_{m,l,\le \kappa_\theta,r_1,r_2}(W^1_{k_1},W^2_{k_2})}_2\\
\lle{} &2^m2^{6\delta^2m}2^{(1-20\delta)j}2^{-l} 2^{l/2}\kappa_\theta\norm{\mao{W^2_{k_2}}}_{L^2(r\dif r)L^\infty_\theta}\norm{W^1_{k_1}}_2.
\end{align*}
\end{proof}

We now address the intermediate modulation cases. We aim to estimate\\ $\BB'_{m,l,\le \kappa_\theta}(\ppd s \va,\vb)$ and $\BB'_{m,l,\le \kappa_\theta}(\va,\ppd s \vb)$ for $l_0<l \le -100\delta m$. We only deal with $\BB'_{m,l,\le \kappa_\theta}(\va,\ppd s \vb)$ as $\BB'_{m,l,\le \kappa_\theta}(\ppd s \va,\vb)$ is similar. By \lemmaref{timederi2}, for any $s \simeq 2^m$, we have
	\[
	\ppd s V^{\iota_2,p_2}_{k_2}(s)=I^{p_2}_{k_2}(s)+II^{p_2}_{k_2}(s)+III^{p_2}_{k_2}(s),
\]
where
\begin{gather*}
I^{p_2}_{k_2}=\sum_{\substack{k_3,p_3,k_4,p_4 \colon \\ p_3+p_4=p_2,-2m \le k_3,k_4 \le \delta^3 m}} P_{k_2} I_{\gti m_q,>-3\delta^2 m}(\vvc,V^{\iota_4,p_4}_{k_4}),\\
\text{where the phase function is } -\iota_2\Lambda(\xi)+\iota_3\Lambda(\xi-\eta)+\iota_4\Lambda(\eta),\\
\norm{II^{p_2}_{k_2}(s)}_2 \lle \varepsilon_1^2 2^{-(3/2-12\delta)m},\\
\mao{III^{p_2}_{k_2}}(\xi,s)=\me^{\mi \iota_2 s\Phi^0} G(\xi,s) \fai_{<-2\delta^2 m}(\abs{\xi}-2\gamma_0),
\end{gather*}
where $\Phi^0=-\Lambda(\xi)+2\Lambda(\xi/2)$ and $G$ satisfies the following estimates:
\begin{gather*}
	\abs{\partial^\alpha_\xi G(\xi,s)}\lle \varepsilon_1^2 2^{-(1-5\delta)m}2^{\abs \alpha (1/2+3\delta)m}, \\
	\abs{\ppd s G(\xi,s)}\lle \varepsilon_1^2 2^{-(1-5\delta)m}2^{-(1-25\delta)m}.
\end{gather*}
Hence the term $\BB'_{m,l,\le \kappa_\theta}(\va,\ppd s \vb)$ is divided into three parts.

The $I^{p_2}_{k_2}$-part can be written as $2^{-l}\TT_{\gti n,m,l,\le \kappa_\theta,>-3\delta^2 m}(\va,\vvc,V^{\iota_4,p_4}_{k_4})$. We use $\xi_1$, $\xi_3$, and $\xi_4$ to denote the frequency variables of $\va$, $\vvc$, and $V^{\iota_4,p_4}_{k_4}$, respectively. The multiplier $\gti n$ is of the form $\gti m_q(\xi_1,\xi_{34})\gti m_q(\xi_3,\xi_4)\psi_{k_{34}}(\xi_{34})$ (See Section~\ref{notations})
 and will be omitted below. 
Inserting localizations of the second resonance function $\wan \Phi$ with parameter $p$, decomposing the cutoff function $\fai_{>-3\delta^2 m}(\Phi^{3,4}_{\iota_2\iota_3\iota_4})$ into $\psi_q(\Phi^{3,4}_{\iota_2\iota_3\iota_4})$, and integrating by parts in time yield 
\begin{align*}
	&\TT_{m,l,\le \kappa_\theta,q,p}(\va,\vvc,V^{\iota_4,p_4}_{k_4})\\
	&=\II'_{m,l,\le \kappa_\theta,q,p}(\va,\vvc,V^{\iota_4,p_4}_{k_4})+\TT'_{m,l,\le \kappa_\theta,q,p}(\ppd s \va,\vvc,V^{\iota_4,p_4}_{k_4})\\
	&+\TT'_{m,l,\le \kappa_\theta,q,p}( \va,\ppd s\vvc,V^{\iota_4,p_4}_{k_4})+\TT'_{m,l,\le \kappa_\theta,q,p}(\va,\vvc,\ppd sV^{\iota_4,p_4}_{k_4}).
\end{align*}
For $\II'_{m,l,\le \kappa_\theta,q,p}(\va,\vvc,V^{\iota_4,p_4}_{k_4})$, by an averaging argument, it can be expressed as
\begin{multline*}
	2^{-p}\int_{\abs{\lambda_1},\abs{\lambda_2} \le 2^{\delta^2 m}} \mao{\fai_1}(\lambda_1)\mao{\fai_2}(\lambda_2)
	\times \Bigl(\FF^{-1}\int q'_m(s)\me^{\mi (s\wan \Phi+\lambda_12^{-q}\Phi^{3,4}_{\iota_2\iota_3\iota_4}+\lambda_2 2^{-p}\wan \Phi)}\gti n \fai^{l_0}_l(\Phi)\\
\times \fai_{\le \log_2\kappa_\theta}(\Theta) \mao{ \va}(\xi-\eta,s) \mao{\vvc}(\eta-\sigma,s)\mao{V^{\iota_4,p_4}_{k_4}}(\sigma,s) \dif \eta \dif \sigma \dif s \Bigr) \dif \lambda_1 \dif \lambda_2 +O_{L^2}(\epsl_1^2 2^{-10m}).
\end{multline*}
Here $\norm{O_{L^2}(\epsl_1^2 2^{-10m})}_2 \lle \epsl_1^2 2^{-10m}$. 
Now applying (Case 1 of) \lemmaref{interfirstpart}  to $V^{\iota_1,p_1}$ and $2^{-2m/3}\FF^{-1}\int \me^{\mi(s+\lambda_1 2^{-q}+\lambda_2 2^{-p})\Phi^{3,4}_{\iota_2\iota_3\iota_4}} \mao{\vvc}(\eta-\sigma,s)\mao{V^{\iota_4,p_4}_{k_4}}(\sigma,s)  \dif \sigma$ (whose bound is given by \lemmaref{timederil2}) yields the desired estimate.
For $\TT'_{m,l,\le \kappa_\theta, q,p}(\ppd s \va,\vvc,V^{\iota_4,p_4}_{k_4})$, by neglecting the cubic terms and the remainders and using an averaging argument, we can write it as
\begin{multline*}
	2^{-p}\int_{\abs \lambda_1,\abs \lambda_2 \le 2^{\delta^2 m}} \mao{\fai_1}(\lambda_1)\mao{\fai_2}(\lambda_2) \Bigl(\FF^{-1}\int q_m(s)\me^{\mi (s\wan \Phi+\lambda_12^{-q}\Phi^{3,4}_{\iota_2\iota_3\iota_4}+\lambda_2 2^{-p}\wan \Phi)}\gti n \fai^{l_0}_l(\Phi) \\
	\times \fai_{\le \log_2 \kappa_\theta}(\Theta) \mao{\ppd s \va}(\xi-\eta,s) \mao{\vvc}(\eta-\sigma,s)\mao{V^{\iota_4,p_4}_{k_4}}(\sigma,s) \dif \eta \dif \sigma \dif s \Bigr) \dif \lambda_1 \dif \lambda_2 +O_{L^2}(\epsl_1^2 2^{-10m}).
\end{multline*}
(Case 2 of) \lemmaref{interfirstpart} can give the desired bound.
For $\TT'_{m,l,\le \kappa_\theta,q,p}( \va,\ppd s\vvc,V^{\iota_4,p_4}_{k_4})$, by neglecting the cubic terms and the remainders and applying an averaging argument, we can write it as
\begin{multline*}
		2^{-p}\int_{\abs \lambda_1,\abs \lambda_2 \le 2^{\delta^2 m}} \mao{\fai_1}(\lambda_1)\mao{\fai_2}(\lambda_2) \Bigl(\FF^{-1}\int q_m(s)\me^{\mi (s\wan \Phi+\lambda_12^{-q}\Phi^{3,4}_{\iota_2\iota_3\iota_4}+\lambda_2 2^{-p}\wan \Phi)}\gti n \fai^{l_0}_l(\Phi)\\
		 \times \fai_{\le \log_2\kappa_\theta}(\Theta) \mao{\va}(\xi-\eta,s) \mao{\ppd s \vvc}(\eta-\sigma,s)\mao{V^{\iota_4,p_4}_{k_4}}(\sigma,s) \dif \eta \dif \sigma \dif s \Bigr) \dif \lambda_1 \dif \lambda_2 +O_{L^2}(\epsl_1^2 2^{-10m}),
\end{multline*}
and (Case 1 of) \lemmaref{interfirstpart} can handle it. The term $\TT'_{m,l,\le \kappa_\theta,q,p}(\va,\vvc,\ppd sV^{\iota_4,p_4}_{k_4})$ can be estimated similarly. The $II^{p_2}_{k_2}$-part can be handled by (Case 1 of) \lemmaref{interfirstpart}.

For the $III^{p_2}_{k_2}$-part, we aim to estimate
\begin{align*}
\FF H(\xi)={}&2^{-l}\int q_m(s)\me^{\mi s \Psi} \gti m_q \fai^{l_0}_l(\Phi) \fai_{\le \log_2\kappa_\theta}(\Theta) \mao{\va}(\xi-\eta,s)\\
&\times \fai_{<-2\delta^2m}(\abs{\eta}-2\gamma_0)G(\eta,s) \dif \eta \dif s,
\end{align*}
where $\Psi=-\Lambda(\xi)+\iota_1\Lambda(\xi-\eta)+2\iota_2 \Lambda(\eta/2)$. We divide the proof into several steps.

\noindent\textbf{Step1:} We decompose $\va$ into $\Vaa$, and the case $j_1 \ge 5m$ is easy to handle. Note that $\abs{\xi-\eta}$ and $\abs{\xi}$ are away from $2\gamma_0$. We first exclude the case $j_1 \ge (1/2)m$. If $j_1 \ge (1/2)m$, then \lemmaref{sizeest}(iii) yields 
\begin{align*}
	&2^{10k_+}2^{\delta k}2^{(1-20\delta)(j+k)}2^{-(1/2-18\delta)n}\norm{A^{(j)}_nQ_{jk}H}_2\\
	\lle{} &2^{6\delta^2m}2^{(1-20\delta)j}2^m 2^{-l}2^{l}\kappa_\theta \sup_{s \simeq 2^m}\norm{G(s)}_\infty \norm{\Vaa(s)}_2.
\end{align*}
We insert localizations of $\nabla_\eta \Psi$ and $\nabla_\xi \Psi$ with parameters $r_1 \ge r_1^0=\flo{-(1/2-5\delta)m}$ and $r_2 \ge r_2^0=\flo{-(1/2-5\delta)m}$ and thereby obtain $H_{r_1r_2}$. Integration by parts in $\eta$ shows that the case $r_1 \ge (-l-m+10\delta^2m) \vee (r^0_1+1)$ is negligible. Hence we can assume that $r_1 \lle -(\delta/3) m$. If $j \ge m+r_2+D$, then integration by parts in $\xi$, together with an averaging argument, shows that this case is negligible. Therefore, we can assume that $j \lle m+r_2$. The bound $r_1 \lle -(\delta/3) m$ implies that $\abs{\lambda'(\abs{\eta-\xi})-\lambda'(\abs{\eta}/2)} \lle 2^{-(\delta/3) m}$. Hence $\abs{\eta-\xi}$ is close to $\gamma_0$. 
Then the smallness of $\Phi$ and the restriction $\abs{\Theta} \lle  \kappa_\theta$ imply that $\abs{\xi}$ is close to $\gamma_0$ and $(\iota_1,\iota_2)=(-,+)$.
If $r_1>r_1^0$, then the identity
\begin{align*}
	\abs{\nabla_\eta \Psi}&=\abs{-\lambda'(\abs{\eta-\xi})\mao{\eta-\xi}+\lambda'(\abs{\eta}/2)\mao \eta}\\
&=\abs{\lambda'(\abs{\eta}/2)-\lambda'(\abs{\eta-\xi})}+O(\kappa_\theta)
\end{align*}
implies that
\[
2^{r_1} \simeq \abs{\abs{\eta}/2-\abs{\eta-\xi}}.
\]
Note that
\[
\abs{\eta-\xi}+\abs{\xi}-\abs{\eta}=\frac{2\abs{\eta-\xi}\abs{\xi}(1-\mao{\eta-\xi}\cdot \mao \xi)}{\abs{\eta-\xi}+\abs{\xi}+\abs{\eta}} \lle \kappa_\theta^2.
\]
It follows that
\[
\abs{\abs{\eta}-2\abs{\xi}} \simeq 2^{r_1}.
\]
Then we calculate $\Psi$. We have
\begin{align*}
\Psi&=-\lambda(\abs{\xi})-\lambda(\abs{\eta-\xi})+2\lambda(\abs{\eta}/2)\\
&=-\lambda(\abs{\xi})-\lambda(\abs \eta-\abs{\xi})+2\lambda(\abs \eta/2)+O(\kappa_\theta^2)\\
&\simeq -(\abs{\eta}-2\abs{\xi})^2 \simeq -2^{2r_1}.
\end{align*}
Moreover,
\[
\abs{\nabla_\xi \Psi}=\abs{\lambda'(\abs \xi)\mao \xi-\lambda'(\abs{\eta-\xi})\mao{\eta-\xi}} \lle 2^{r_1}.
\]
Hence $r_2 \lle r_1$.

\noindent\textbf{Step 2:}
If $r_1=r_1^0$, then, in view of \lemmaref{resoregion}, we estimate
\begin{align*}
	&2^{10k_+}2^{\delta k}2^{(1-20\delta)(j+k)}2^{-(1/2-18\delta)n}\norm{A^{(j)}_nQ_{jk}H_{r_1r_2}}_2\\
	\lle{} & \varepsilon_1^3 2^{\delta^2 m}2^{(1-20\delta)(m+r_2)}2^m2^{-l}(2^l+2^{r_2+\delta^2 m})^{1/2}(2^{l}+2^{2r_2+2\delta^2 m})\kappa_\theta 2^{-(1-5\delta)m}.
\end{align*}
If $r_1>r_1^0$, then integration by parts in time gives
\begin{align*}
	\FF H_{r_1r_2}(\xi)={}&-2^{-l}\int q'_m(s)\frac{\me^{\mi s \Psi}}{\mi \Psi} \wan{\gti m}\mao{\Vaa}(\xi-\eta,s) \fai_{<-2\delta^2m}(\abs{\eta}-2\gamma_0)G(\eta,s) \dif \eta \dif s\\
	&-2^{-l}\int q_m(s)\frac{\me^{\mi s \Psi}}{\mi \Psi}\wan{\gti m} \mao{\ppd s \Vaa}(\xi-\eta,s) \fai_{<-2\delta^2m}(\abs{\eta}-2\gamma_0)G(\eta,s) \dif \eta \dif s\\
	&-2^{-l}\int q_m(s)\frac{\me^{\mi s \Psi}}{\mi \Psi} \wan{\gti m} \mao{\Vaa}(\xi-\eta,s) \fai_{<-2\delta^2m}(\abs{\eta}-2\gamma_0)G_s(\eta,s) \dif \eta \dif s\\
	&=H_{r_1r_2;1}+H_{r_1r_2;2}+H_{r_1r_2;3},
\end{align*}
where
\[
\wan{\gti m}=\gti m_q \fai^{l_0}_l(\Phi) \fai_{\le\log_2 \kappa_\theta}(\Theta) \fai^{r_1^0}_{r_1}(\nabla_\eta \Psi)\fai^{r_2^0}_{r_2}(\nabla_\xi \Psi).
\]
We have, using \lemmaref{sizeest}(iii) and \lemmaref{resoregion},
\begin{align*}
	\norm{\psi_{k}(\xi)H_{r_1r_2;1}}_2 &\lle \varepsilon_1^2 2^{-l}2^{6\delta^2 m}2^{-2r_1}2^l\kappa_\theta2^{-(1-5\delta)m} \sup_{s \simeq 2^m}\norm{\Vaa}_2\\
	&\lle \varepsilon_1^3 2^{6\delta^2 m}2^{-2r_1}\kappa_\theta2^{-(1-5\delta)m}(2^l+2^{r_2+\delta^2m})^{1/2}.
\end{align*}
For $H_{r_1r_2;2}$, it follows from \lemmaref{sizeest}(iii), \lemmaref{resoregion}, and \lemmaref{timederifli} that
\begin{align*}
	\norm{\psi_{k}(\xi)H_{r_1r_2;2}}_2 &\lle \varepsilon_1^2 2^m2^{-l}2^{6\delta^2m}2^{-2r_1}2^l\kappa_\theta2^{-(1-5\delta)m}\sup_{s \simeq 2^m}\norm{\ppd s\Vaa}_2\\
	&\lle \varepsilon_1^3 2^m2^{6\delta^2 m}2^{-2r_1}\kappa_\theta2^{-(1-5\delta)m}(2^l+2^{r_2+\delta^2m})^{1/2}2^{-(1-25\delta)m}.
\end{align*}
The term $H_{r_1r_2;3}$ can be estimated in a way similar to that for $H_{r_1r_2;2}$.

\subsection{Cubic Terms}\label{section:cubic}
In this subsection, we deal with cubic terms $\TT_m(\va,\vb,\vvc)$ and those arising in $\BB'_{m,l}(\ppd s \va,\vb)$ and $\BB'_{m,l}(\va,\ppd s \vb)$ for $l>-100\delta m$. We denote by $\gti n$ the multipliers in these terms. These terms can then be written as $\TT_{\gti n,m}(\va,$ $\vb,\vvc)$ or $\TT'_{\gti n,m,l}(\va,\vb,\vvc)$. In what follows, we suppress the multiplier $\gti n$ from the notation. We begin by recalling the structure of the multiplier $\gti n$. Its general form is
\begin{align*}
	\gti n(\xi_1,\xi_2,\xi_3)={}&f_1(\xi_1)f_2(\xi_2)f_3(\xi_3)\\
	&\times f_{23}(\xi_{23})f_{13}(\xi_{13})f_{12}(\xi_{12})\\
	&\times f(\xi_1+\xi_2+\xi_3),
\end{align*}
where all factors belong to $\gti S^{100}$. We refer to the factors $f_{23}$, $f_{13}$, and $f_{12}$ as flag factors. Let $K_0=\set{k_\alpha \colon \gti n \text{ contains a }\xi_\alpha-\text{factor}, \alpha \in \set{1,2,3,23,13,12,\emptyset}}$, and set $\cl k=\max K_0$ and $\ul k=\min K_0$. The multiplier $\gti n$ has the following structures:
\begin{itemize}
	\item $\norm{\gti n}_{S_{K_0}} \lle 2^{2\cl k}2^{10\cl k_+}$;
	\item For $\TT'_{m,l}(\va,\vb,\vvc)$, the multiplier $\gti n$ contains only one type of flag factors.
\end{itemize}
We use $k_{**}$ to denote the localization parameter associated with the flag factor.

The cases $\cl k \ge \delta^3 (m+j)$ and $\ul k \le -3(m+j)$ are easy to handle (see the proof of \lemmaref{timederi2}). The case $j \ge 10m$ can be treated by the approximate-finite-speed-of-propagation argument. Moreover, the case $\cl k \le -j/3-m/10$ is straightforward. We only handle $\TT'_{m,l}(\va,\vb,\vvc)$ since $\TT_m(\va,\vb,\vvc)$ is similar. Note that $\norm{\fai^{l_0}_l(\Phi(\xi,\eta))}_{S_{k,k_1,k_2}} \lle 2^{-5l+\delta^2 m}$, as a consequence of \lemmaref{symbound}(iii). We aim to show that
\[
2^{10k_+}2^{\delta k}2^{(1-20\delta)(j+k)}2^{-(1/2-18\delta)n} \norm{A^{(j)}_nQ_{jk}\TT'_{m,l}(\va,\vb,\vvc)}_2 \lle \epsl_1^3 2^{-5\delta^4 m}.
\]

\subsubsection{Approximate Finite Speed of Propagation}
We aim to exclude the case $j \ge m+\cl k+D$. We decompose $\TT'_{m,l}(\va,\vb,\vvc)$ into $\TT'_{m,l}(\Vaa,\Vbb,\Vcc)$, and the case $j_1 \vee j_2 \vee j_3 \ge 30m$ is easy to deal with. The approximate-finite-speed-of-propagation argument fails if $j_1 \wedge j_2 \wedge j_3 \ge (1-\delta^2)j$, or $(-k) \vee (-k_{**}) \ge (1-\delta^2)j$. We next address these cases. We may assume that $k_1 \ge k_2 \ge k_3$.

\noindent \underline{Case 1: $j_1 \wedge j_2 \wedge j_3 \ge (1-\delta^2)j$.}\\
Define the size function
\[
\sigma(k,n)=
\begin{cases}
2^{k} & \text{if } n=0;\\
2^{-n/2} & \text{if } n>0.
\end{cases}
\]
Then one can bound the $L^2$ norm of $P_k\TT'_{m,l}(\Vaa,\Vbb,\Vcc)$ by
\[
2^{100\delta m}2^m2^{2k_1}2^{10\cl k_+} \sigma (k_2,n_2) \sigma(k_3,n_3) \sup_{s \simeq 2^m}\norm{\Vaa(s)}_2\norm{\Vbb(s)}_2\norm{\Vcc(s)}_2.
\]

\noindent \underline{Case 2: $(-k)$ $\vee$ $(-k_{**}) \ge (1-\delta^2)j$.}\\
If $-k \ge (1-\delta^2)j$, then, using \cororef{decayestcom} and \eqref{energyest}, we estimate
\begin{align*}
&2^{10k_+}2^{\delta k}2^{(1-20\delta)(j+k)}2^{-(1/2-18\delta)n} \norm{A^{(j)}_nQ_{jk}\TT'_{m,l}(\Vaa,\Vbb,\Vcc)}_2\\
\lle{} & 2^{\delta k}2^{(1-20\delta)\delta^2 j} 2^m2^{k} 2^{700\delta m} 2^{2\cl k} 2^{10\cl k_+}\sup_{s \simeq 2^m}\norm{E^{\iota_1}\Vaa}_\infty \norm{\Vbb}_2\norm{\Vcc}_2 \\
\lle{} & \varepsilon_1^3 2^{\delta m} 2^{-(1-\delta^2)j} 2^{700\delta m} 2^m 2^{2\cl k} 2^{-(1-21\delta)m}2^{-(1-39\delta)k_1} 2^{-\delta k_2} 2^{-\delta k_3} \\
\lle{} & \varepsilon_1^3 2^{-10\delta^4 m}.
\end{align*}
If $-k_{**} \ge (1-\delta^2)j$, then $l \gge 0$, and using \eqref{energyest}, we have
\begin{align*}
&2^{10k_+}2^{\delta k}2^{(1-20\delta)(j+k)}2^{-(1/2-18\delta)n} \norm{A^{(j)}_nQ_{jk}\TT'_{m,l}(\Vaa,\Vbb,\Vcc)}_2\\
\lle{}&	2^{\delta^2 m}2^{(1-20\delta)j}2^{-l}2^m2^{2\cl k}2^{10\cl k_+} 2^{2k_{**}}\sup_{s \simeq 2^m}\norm{\Vaa}_2\norm{\Vbb}_2\norm{\Vcc}_2\\
\lle{}& \varepsilon_1^3 2^{2\delta^2 m}2^{(-1-20\delta)j} 2^{2\cl k} 2^m 2^{-\delta k_1}2^{-\delta k_2}2^{-\delta k_3}\\
\lle{}& \varepsilon_1^3 2^{-10\delta^4 m}.
\end{align*}
Therefore, we can assume that $j \le m+\cl k+D$. In what follows, we may assume that $k_1 \ge k_2 \ge k_3$. Moreover, we can assume that $\cl k \ge -3m/7$ and $ \ul k \ge -2m$. We will consider two cases for the parameters in $K_0$:
\begin{enumerate}[label=(\Alph*)]
	\item $\norm{\gti n}_{S_{K_0}} < 2^{k_1+k_2+10\delta^2m}$,
	\item $\norm{\gti n}_{S_{K_0}} \ge 2^{k_1+k_2+10\delta^2m}$,
\end{enumerate}
where Case A is easier.

\subsubsection{Case A}
In this part, we consider the easy case $\norm{\gti n}_{S_{K_0}} < 2^{k_1+k_2+10\delta^2m}$. We first handle the case $l \ge -\delta^2 m$. \cororef{decaydecom} yields $E_{\iota_1 s}\va(s)=I^{p_1}_{k_1}(s)+II^{p_1}_{k_1}(s)$ for any $s \simeq 2^m$, and the $II^{p_1}_{k_1}$-part is easy to handle. For the $I^{p_1}_{k_1}$-part, applying \cororef{decaydecom} again yields $E_{\iota_2 s}\vb(s)=I^{p_2}_{k_2}(s)+II^{p_2}_{k_2}(s)$ for any $s \simeq 2^m$, and the $II^{p_2}_{k_2}$-part can likewise be treated without difficulty. For the remaining parts, we have
\begin{align*}
&2^{10k_+}2^{\delta k}2^{(1-20\delta)(j+k)}2^{-(1/2-18\delta)n} \norm{A^{(j)}_nQ_{jk}\TT'_{m,l}(E^{-\iota_1}I^{p_1}_{k_1},E^{-\iota_2}I^{p_2}_{k_2},\vvc)}_2\\
\lle{}& 2^{\delta^2 m}2^{(1-20\delta)j} 2^m 2^{-6l+\delta^2 m} 2^{k_1+k_2+10\delta^2m} \sup_{s \simeq 2^m} \norm{I^{p_1}_{k_1}}_{\infty} \norm{I^{p_2}_{k_2}}_{\infty} \norm{\vvc}_2   \lle{} \varepsilon_1^3 2^{-10\delta^4 m}.
\end{align*}
In what follows, we can assume that $l \leq -\delta^2 m$. We insert localizations of the second resonance function $\wan \Phi$ with parameter
 $p \ge p_0=\flo{-m/2+1000\delta m}$.
 If $p>p_0$, then integration by parts in time yields the desired estimate. Hence, we focus on $\TT_{m,l,p_0}(\va,\vb,\vvc)$, and it is easy to see that $\ul k \ge -\delta^2 m$. We next decompose $\TT_{m,l,p_0}(\va,\vb,\vvc)$ into $\TT_{m,l,p_0}(\Vaa,\Vbb,\Vcc)$, and the case $j_1 \vee j_2 \vee j_3 \ge 30m$ is easy to exclude.
If $j_1 \vee j_2 \vee j_3 \ge (1/2)m$, then we may assume that $j_1 \ge (1/2)m$, and using an averaging argument, we estimate
\begin{align*}
	&2^{10k_+} 2^{\delta k}2^{(1-20\delta)(j+k)}2^{-(1/2-18\delta)n} \norm{A^{(j)}_nQ_{jk}\TT_{m,l,p_0}(\Vaa,\Vbb,\Vcc)}_2\\
	\lle{}& 2^{\delta^2 m}2^{(1-20\delta)j} 2^m 2^{-6l+\delta^2 m}2^{k_1+k_2+10\delta^2m}(\sup_{s \simeq s_1 \simeq s_2 \simeq 2^m}\norm{\Vaa(s)}_2 \norm{E_{\iota_2 s_1}\Vbb(s)}_{\infty}\\
	&\times \norm{E_{\iota_3 s_2}\Vcc(s)}_{\infty} +\epsl_1^3 2^{-10m}).
\end{align*}
We next insert a cutoff function $\fai_{>-(1/2-5\delta^2)m}(\wan \Xi)$ into $\TT_{m,l,p_0}(\Vaa,\Vbb,\Vcc)$. For $\TT_{m,l,p_0,>-(1/2-5\delta^2)m}(\Vaa,\Vbb,\Vcc)$, integration by parts in $\eta$ and $\sigma$ shows that this part is negligible. For $\TT_{m,l,p_0,\le-(1/2-5\delta^2)m}(\Vaa,\Vbb,\Vcc)$, $\abs{\Phi^{ab,c}_{+\iota_{ab}\iota_c}} \le 2^l$ and $\abs{\wan \Phi} \lle 2^{p_0}$ imply that $\abs{\Phi^{a,b}_{\iota_{ab}\iota_a\iota_b}} \lle 2^{l}$. Here we assume that the frequency variable of the flag factor is $\xi_{ab}$ and $\set{\xi_a,\xi_b,\xi_c}=\set{\xi_1,\xi_2,\xi_3}$. Combined with $\abs{\wan \Xi} \lle 2^{-(1/2-5\delta^2)m}$, \lemmaref{resoregion} implies that $\abs{\xi_a}$, $\abs{\xi_b}$, and $\abs{\xi_{ab}}$ are close to $\gamma_0$, $\gamma_0$, and $2\gamma_0$, respectively, and that $\abs{\abs{\xi_a}-\gamma_0} \lle 2^l$. It follows that $\abs{\abs{\xi_i}-\gamma_0} \lle 2^l$, $1 \le i \le 3$. In view of \lemmaref{iteratedreso}, we have $\abs{\nabla_\xi \wan \Phi} \lle 2^{-(1/2-5\delta^2)m}$. If $j \le (1/2+10\delta^2)m$, then the desired bound is clear. If $j \ge (1/2+10\delta^2)m$, then integration by parts in $\xi$ shows that this case is negligible. This concludes the proof of Case A.

We now turn to Case B: $\norm{\gti n}_{S_{K_0}} \ge 2^{k_1+k_2+10\delta^2m}$, which implies that $k_1 \gg k_2$, $k_2 \le -5\delta^2 m$, $l \gge 0$, and the flag factor is a $\xi_{12}$- or $\xi_{13}$-factor.

We insert localizations of the second resonance function $\wan \Phi$ with parameter $p \ge p_0=\lfloor -2m/3-m/100 \rfloor$. When $p>p_0$, integration by parts in time yields
\begin{equation}\label{4.4:main}
	\begin{aligned}
		&\TT_{m,l,p}(\va,\vb,\vvc)=\II'_{m,l,p}(\va,\vb,\vvc)+\TT'_{m,l,p}(\ppd s \va,\vb,\vvc)\\
		&+\TT'_{m,l,p}(\va,\ppd s \vb,\vvc)+\TT'_{m,l,p}(\va,\vb,\ppd s \vvc),
	\end{aligned}
\end{equation}
where the superscript $'$ indicates that the multiplier is divided by $-\mi \wan \Phi$.

\subsubsection{Small Modulation Cases and Boundary Terms}
In this subsection, we address the small modulation cases and boundary terms in \eqref{4.4:main}. The case $p>-m/2+m/100$ for the boundary terms is easy to handle. It suffices to prove that
\begin{multline*}
2^{10k_+}2^{\delta k}2^{(1-20\delta)(j+k)}2^{-(1/2-18\delta)n} L_p\norm{A^{(j)}_nQ_{jk}I_{\gti n,s,l,p}(\va(s),\vb(s),\vvc(s))}_2 \\
\lle \epsl_1^3 2^{-10\delta^4 m},
\end{multline*}
where $s \in [2^{m-1},2^{m+1}]$, $p_0 \le p \le -m/2+m/100$, and
\[
L_p=
\begin{cases}
	2^{-p} & p>p_0;\\
	2^m & p=p_0.
\end{cases}
\]
We first exclude the case $\iota_1=-$. In this case, we have $2^{2k_1}\lle \abs{\wan{\Phi}} \lle 2^p$, so $k_1\lle p/2$. Therefore, using an averaging argument, \cororef{decayestcom}, \cororef{decayestlow}, and \eqref{energyest}, we estimate
\begin{align*}
	&2^{10k_+}2^{\delta k}2^{(1-20\delta)(j+k)}2^{-(1/2-18\delta)n} L_p\norm{A^{(j)}_nQ_{jk}I_{s,l,p}(\va(s),\vb(s),\vvc(s))}_2\\
	\lle{}& 2^{2\delta^2 m} 2^{(1-20\delta)j} 2^{(1-19\delta)k} L_p 2^{2\cl k}2^{10\cl k_+}(\sup_{s_1 \simeq s_2 \simeq 2^m}\norm{E_{\iota_1s_1}\va(s)}_{\infty} \norm{E_{\iota_2s_2}\vb(s)}_{\infty}\\
	&\times \norm{\vvc(s)}_{2}+\varepsilon_1^3 2^{-10m})\\
	\lle{}& \varepsilon_1^3 2^{3\delta^2 m} 2^{(1-20\delta)m} L_p2^{(4-39\delta)\cl k} 2^{-(1-21\delta)m} 2^{-(1-39\delta)k_1}2^{-(1/2-5\delta)m} 2^{-\delta k_3}+\epsl_1^3 2^{-5m}\\
	\lle{}& \varepsilon_1^3 2^{-10\delta^4 m}.
\end{align*}

We next improve the approximate-finite-speed-of-propagation bound. We aim to exclude the case $j \ge (m+k_2+D)\vee (1/2+\delta^2)m$. We decompose $I_{s,l,p}(\va(s),\vb(s),\vvc(s))$ into $I_{s,l,p}(\Vaa(s),\Vbb(s),\Vcc(s))$, and the case $j_1 \vee j_2 \vee j_3 \ge 30m$ is easy to exclude. The approximate-finite-speed-of-propagation argument, together with an averaging argument applied to $\fai^{p_0}_{p}(\wan \Phi)$, gives the desired bound if $j_1 \le (1-\delta^2)j$. We next exclude the case $j_1 \ge (1-\delta^2)j$.

The case $\cl k \le -10\delta m$ is easy to handle. Then we consider two cases: $n \gge n_1$ and $n \ll n_1$.

\noindent\ul{Case 1: $n \gge n_1$.}\\
Note that $\abs{\wan \Phi} \lle 2^{p}$ implies that
\[
\abs{-\lambda(\abs{\xi})+\lambda(\abs{\xi-\eta})+\iota_2+\iota_3} \lle 2^{2k_2}+2^{p}.
\]
By means of Schur's test, we estimate
\begin{align*}
	&2^{10k_+}2^{\delta k}2^{(1-20\delta)(j+k)}2^{-(1/2-18\delta)n} L_p \norm{A^{(j)}_nQ_{jk}I_{s,l,p}(\Vaa(s),\Vbb(s),\Vcc(s))}_2\\
\lle{} &2^{2\delta^2 m}2^{(1-20\delta)j}2^{-(1/2-18\delta)n}L_p 2^{2k_1}\\
&\times\norm{\Vbb(s)}_2\norm{\Vcc(s)}_2(2^{p\vee 2k_2}/2^{-10\delta m})2^{k_2-k_1}\norm{\Vaa(s)}_2
\end{align*}
if $k_2 \le -2m/5$.
On the other hand, using an averaging argument, we have
\begin{align*}
	&2^{10k_+}2^{\delta k}2^{(1-20\delta)(j+k)}2^{-(1/2-18\delta)n} L_p\norm{A^{(j)}_nQ_{jk}I_{s,l,p}(\Vaa(s),\Vbb(s),\Vcc(s))}_2\\
	\lle{} &2^{2\delta^2 m}2^{(1-20\delta)j}2^{-(1/2-18\delta)n}L_p2^{2k_1}(\norm{\Vaa(s)}_2\norm{E_{\iota_2[2^{m-2},2^{m+2}]}\Vbb(s)}_\infty\\
	&\times\norm{E_{\iota_3[2^{m-2},2^{m+2}]}\Vcc(s)}_\infty+\epsl_1^3 2^{-10m})
\end{align*}
if $k_2 \ge -2m/5$.

\noindent\ul{Case 2: $n \ll n_1$.}\\
It suffices to prove that
\begin{multline*}
	2^{10k_+}2^{\delta k}2^{(1-20\delta)(j+k)}L_p \\ \times \sup_{0 \le n \le j+1}2^{-(1/2-18\delta)n} \norm{A^{(j)}_nP_kI_{s,l,p}(\Vaa(s),\Vbb(s),\Vcc(s))}_2
	\lle \epsl_1^3 2^{-10\delta^4 m}.
\end{multline*}
It is easy to see that $\iota_2+\iota_3=0$ and $2^{-n} \lle 2^{-j} \vee 2^{p \vee 2k_2}$. If $2k_2 \ge p$, then the case $2k_2 \le -j$ is easy to handle by the following estimate, obtained via an averaging argument:
\begin{align*}
	&2^{10k_+}2^{\delta k}2^{(1-20\delta)(j+k)}2^{-(1/2-18\delta)n}L_p \norm{A^{(j)}_nP_kI_{s,l,p}(\Vaa(s),\Vbb(s),\Vcc(s))}_2\\
	\lle{} &2^{2\delta^2 m}2^{(1-20\delta)j}2^{-(1/2-18\delta)j}L_p(\norm{\Vaa(s)}_2\norm{E_{\iota_2[2^{m-2},2^{m+2}]}\Vbb(s)}_\infty\\
	&\times\norm{E_{\iota_3[2^{m-2},2^{m+2}]}\Vcc(s)}_\infty+\epsl_1^3 2^{-10m}).
\end{align*}
For the case $2k_2>-j$, if $j_2 \ge (1-\delta^2)m+k_2$, then, using an averaging argument, we estimate
\begin{align*}
	&2^{10k_+}2^{\delta k}2^{(1-20\delta)(j+k)}2^{-(1/2-18\delta)n}L_p \norm{A^{(j)}_nP_kI_{s,l,p}(\Vaa(s),\Vbb(s),\Vcc(s))}_2\\
	\lle{} &2^{2\delta^2 m}2^{(1-20\delta)j}2^{(1-36\delta)k_2}L_p(\norm{E_{\iota_1[2^{m-2},2^{m+2}]}\Vaa(s)}_\infty\norm{\Vbb(s)}_2\\
	&\times\norm{E_{\iota_3[2^{m-2},2^{m+2}]}\Vcc(s)}_\infty+\epsl_1^3 2^{-10m}),
\end{align*}
while if $j_2 \le (1-\delta^2)m+k_2$, then, using an averaging argument, \eqref{decayestg}, and \cororef{decayestlow}, we have
\begin{align*}
	&2^{10k_+}2^{\delta k}2^{(1-20\delta)(j+k)}2^{-(1/2-18\delta)n}L_p \norm{A^{(j)}_nP_kI_{s,l,p}(\Vaa(s),\Vbb(s),\Vcc(s))}_2\\
	\lle{} &2^{2\delta^2 m}2^{(1-20\delta)j}2^{(1-36\delta)k_2}L_p(\norm{\Vaa(s)}_2\norm{E_{\iota_2[2^{m-2},2^{m+2}]}\Vbb(s)}_\infty\\
	&\times\norm{E_{\iota_3[2^{m-2},2^{m+2}]}\Vcc(s)}_\infty+\epsl_1^3 2^{-10m})\\
	\lle{} &\epsl_1^3 2^{2\delta^2 m}2^{(1-20\delta)j}2^{(1-36\delta)k_2}L_p2^{-(1/2-2\delta)j_1}2^{-(1-\delta^2)m}2^{-(1+\delta)k_2}2^{-(1/2-2\delta)m}+\epsl_1^3 2^{-5m}\\
	\lle{} &\epsl_1^3 2^{-10\delta^4m}.
\end{align*}
If $2k_2 \le p$, then using Schur's test, we have
\begin{align*}
	&2^{10k_+}2^{\delta k}2^{(1-20\delta)(j+k)}2^{-(1/2-18\delta)n}L_p \norm{A^{(j)}_nP_kI_{s,l,p}(\Vaa(s),\Vbb(s),\Vcc(s))}_2\\
\lle{} &2^{\delta^2 m}2^{(1-20\delta)j}2^{-(1/2-18\delta)j}L_p\norm{\Vbb(s)}_2\norm{\Vcc(s)}_2(2^{p}/2^{-10\delta m})2^{k_2}\norm{\Vaa(s)}_2
\end{align*}
if $-j=-j \vee p$, and
\begin{align*}
	&2^{10k_+}2^{\delta k}2^{(1-20\delta)(j+k)}2^{-(1/2-18\delta)n}L_p \norm{A^{(j)}_nP_kI_{s,l,p}(\Vaa(s),\Vbb(s),\Vcc(s))}_2\\
\lle{} &2^{\delta^2 m}2^{(1-20\delta)j}2^{(1/2-18\delta)p}L_p\norm{\Vbb(s)}_2\norm{\Vcc(s)}_22^{-n_1/2}(2^{p}/2^{-10\delta m})^{1/2}2^{k_2}\\
&\times\norm{\Vaa(s)}_2
\end{align*}
if $p=-j \vee p$.

We now deal with the case $j \le (m+k_2+D)\vee (1/2+\delta^2)m$. By \cororef{decaydecom}, for any $s \simeq s_1 \simeq 2^m$, $E_{\iota_1 s_1}\va(s)=I^{p_1}_{k_1}(s,s_1)+II^{p_1}_{k_1}(s,s_1)$. If $m+k_2+D \le (1/2+\delta^2)m$, then, using an averaging argument, we estimate 
\begin{align*}
	&2^{10k_+}2^{\delta k}2^{(1-20\delta)(j+k)}2^{-(1/2-18\delta)n}L_p \norm{A^{(j)}_nQ_{jk}I_{s,l,p}(\va(s),\vb(s),\vvc(s))}_2\\
\lle{} & 2^{2\delta^2 m}2^{(1-20\delta)(1/2+\delta^2)m} 2^m 2^{2k_1}(\sup_{s_1 \simeq s_2 \simeq 2^m}\norm{I^{p_1}_{k_1}(s,s_1)}_{\infty} \norm{E_{\iota_2 s_2}\vb(s)}_\infty\norm{\vvc(s)}_2\\
&+\sup_{s_1 \simeq s_2 \simeq s_3 \simeq 2^m}\norm{II^{p_1}_{k_1}(s,s_1)}_2\norm{E_{\iota_2 s_2}\vb(s)}_\infty\norm{E_{\iota_3 s_3}\vvc(s)}_\infty+\varepsilon_1^3 2^{-10m}).
\end{align*}
If $m+k_2+D \ge (1/2+\delta^2)m$, then it follows that $k_2 \gge -m/2$. We continue to decompose $I_{m,l,p}(\va(s),\vb(s),\vvc(s))$ into $I_{m,l,p}(\va(s),\Vbb(s),\vvc(s))$, and we may assume that $j_2 \le 10m$. The case $j_2 \ge (1-\delta^2)m+k_2$ is handled via an averaging argument as follows:
\begin{align*}
	&2^{10k_+}
	2^{\delta k}2^{(1-20\delta)(j+k)}2^{-(1/2-18\delta)n}L_p \norm{A^{(j)}_nQ_{jk}I_{s,l,p}(\va(s),\Vbb(s),\vvc(s))}_2\\
\lle{} & 2^{2\delta^2 m}2^{(1-20\delta)(m+k_2)}  2^m 2^{2k_1}\\
&\times (\norm{I^{p_1}_{k_1}(s,[2^{m-2},2^{m+2}])}_{\infty} \norm{\Vbb(s)}_2 \norm{E_{\iota_3[2^{m-2},2^{m+2}]}\vvc(s)}_{\infty}\\
&+\norm{II^{p_1}_{k_1}(s,[2^{m-2},2^{m+2}])}_2 \norm{E_{\iota_2[2^{m-2},2^{m+2}]}\Vbb(s)}_\infty\norm{E_{\iota_3[2^{m-2},2^{m+2}]}\vvc(s)}_{\infty}\\
&+\varepsilon_1^3 2^{-10m}).
\end{align*}
For the case $j_2 \le (1-\delta^2)m+k_2$, using an averaging argument we have
\begin{align*}
	&2^{10k_+}2^{\delta k}2^{(1-20\delta)(j+k)}2^{-(1/2-18\delta)n}L_p \norm{A^{(j)}_nP_kI_{s,l,p}(\va(s),\Vbb(s),\vvc(s))}_2\\
\lle{} & 2^{2\delta^2 m}2^{(1-20\delta)(m+k_2)} 2^m 2^{2k_1}\\
&\times (\norm{I^{p_1}_{k_1}(s,[2^{m-2},2^{m+2}])}_{\infty} \norm{E_{\iota_2[2^{m-2},2^{m+2}]}\Vbb(s)}_\infty \norm{\vvc(s)}_2\\
&+\norm{II^{p_1}_{k_1}(s,[2^{m-2},2^{m+2}])}_2 \norm{E_{\iota_2[2^{m-2},2^{m+2}]}\Vbb(s)}_\infty\norm{E_{\iota_3[2^{m-2},2^{m+2}]}\vvc(s)}_{\infty}\\
&+\varepsilon_1^3 2^{-10m}),
\end{align*}
and \eqref{decayestg} can handle the term involving $I^{p_1}_{k_1}$.

\subsubsection{Large Modulation Cases}
In this subsection, we deal with
\[
\TT'_{m,l,p}(\ppd s \va,\vb, \vvc), \ \TT'_{m,l,p}(\va,\ppd s \vb,\vvc),
\]
and
\[
\TT'_{m,l,p}(\va,\vb,\ppd s \vvc)
\]
in \eqref{4.4:main}. The case $p \ge -2\delta^2 m$ is easy to handle, so we can assume that $p \le -2\delta^2 m$. We begin with $\TT'_{m,l,p}(\va, \ppd s \vb,\vvc)$ and $\TT'_{m,l,p}(\va,\vb,\ppd s \vvc)$, and we only consider $\TT'_{m,l,p}(\va,\ppd s \vb,\vvc)$ as the other is similar. Neglecting the cubic terms and the remainders, we write it as a quadrilinear operator
\[
2^{-p}\QQ_{\gti p,m,l,p}(\va,\vx 4,\vx 5,\vvc).
\]

We denote by $\xi_1$, $\xi_4$, $\xi_5$, and $\xi_3$ the frequency variables of $\va$, $\vx 4$, $\vx 5$, and $\vvc$, respectively. The multiplier $\gti p$ is a product of $\xi_1$-, $\xi_4$-, $\xi_5$-, $\xi_3$-, $\xi_{45}$-, $\xi_{145}$-, $\xi_{13}$-, and $\xi$-factors in $\gti S^{100}$. Let $K=\set{k_\alpha \colon \gti p \text{ contains a }\xi_\alpha-\text{factor}}$, and set $\cl k=\max K$ and $\ul k=\min K$. Then $\norm{\gti p}_{S_K} \lle 2^{2\cl k}2^{k_4 \vee k_5}2^{15\cl k_+}$. (See Section~\ref{notations} and \ref{dispersiveformu}.) In the following, we omit the multiplier $\gti p$ from the notation. We may assume that $-3m \le \ul k \le \cl k \le \delta^3 m$. Note that the third resonance function $\Phi^{1,4,5,3}_{+\iota_1\iota_4\iota_5\iota_3}=\Phi^{1,45,3}_{+\iota_1\iota_2\iota_3}+\Phi^{4,5}_{\iota_2\iota_4\iota_5}$, and $\abs{\Phi^{4,5}_{\iota_2\iota_4\iota_5}} \gge 1$ since $k_{45} \le -5\delta^2 m$. It follows that $\abs{\Phi^{1,4,5,3}_{+\iota_1\iota_4\iota_5\iota_3}} \gge 1$. Then we insert localizations of $\Phi^{1,4,5,3}_{+\iota_1\iota_4\iota_5\iota_3}$ with parameter $q>-D$, and integration by parts in time gives
\begin{equation}\label{eq:largemodu1}
\begin{aligned}
&\QQ_{m,l,p,q}(\va,\vx 4,\vx 5,\vvc)=\II'_{m,l,p,q}(\va,\vx 4,\vx 5,\vvc)\\
	&+\QQ'_{m,l,p,q}(\ppd s \va,\vx 4,\vx 5,\vvc)+\QQ'_{m,l,p,q}(\va,\ppd s\vx 4,\vx 5,\vvc)\\
	&+\QQ'_{m,l,p,q}(\va,\vx 4, \ppd s \vx 5,\vvc)+\QQ'_{m,l,p,q}(\va,\vx 4,\vx 5, \ppd s \vvc).
\end{aligned}
\end{equation}
Note that \lemmaref{symbound3} yields $\norm{\psi_q(\Phi^{1,4,5,3}_{+\iota_1\iota_4\iota_5\iota_3})}_{S_{K}} \lle 2^{\delta^2 m}$.
All terms except
\[
\QQ'_{m,l,p,q}(\ppd s \va,\vx 4,\vx 5,\vvc)
\]
can be easily estimated by averaging arguments and an $L^2 \times L^\infty$ argument. Indeed, we apply an averaging argument to the cutoff function $\fai^{p_0}_p(\Phi^{1,45,3}_{+\iota_1\iota_2\iota_3})$. The other modulation cutoff functions are absorbed into the multiplier $\gti p$, and we apply an averaging argument to the multiplier. (See Section~\ref{techlemma}.)

For the term
\[
\QQ'_{m,l,p,q}(\ppd s \va,\vx 4,\vx 5,\vvc),
\]
it can be written as a quintilinear operator
\[
2^{-q}\QU_{\gti q,m,l,p,q}(\vx 6, \vx 7,\vx 4, \vx 5,\vvc)
\]
(where the cubic terms and the remainders in $\ppd s \va$ are easy to handle). Here we define
\[
\QU_{\gti m,m}(f_1,f_2,f_3,f_4,f_5):=\int q_m(s)I_{\gti m,s}(f_1(s),f_2(s),f_3(s),f_4(s),f_5(s)) \dif s.
\]
We denote by $\xi_6$, $\xi_7$, $\xi_4$, $\xi_5$, and $\xi_3$ the frequency variables of $\vx 6$, $\vx 7$, $\vx 4$, $\vx 5$, and $\vvc$, respectively. The multiplier $\gti q$ is a product of $\xi_6$-, $\xi_7$-, $\xi_4$-, $\xi_5$-, $\xi_3$-, $\xi_{67}$-, $\xi_{45}$-, $\xi_{673}$-, $\xi_{6745}$-, and $\xi$-factors in $\gti S^{100}$. Let $K=\set{k_\alpha \colon \gti q \text{ contains a }\xi_\alpha-\text{factor}}$, and set $\cl k=\max K$ and $\ul k=\min K$. Then $\norm{\gti q}_{S_{K}} \lle 2^{2\cl k}2^{k_6 \vee k_7}2^{k_4 \vee k_5}2^{20\cl k_+}$. In what follows, we suppress the multiplier $\gti q$ from the notation. It is easy to exclude the cases $\ul k < -3m$ and $\cl k>\delta^3 m$. We may assume that $k_6 \geq k_7$ and $k_4 \geq k_5$. We next exploit the gap between $\xi_6$ and $\xi_3$.
We decompose
\[
\QU_{m,l,p,q}(\vx 6, \vx 7,\vx 4, \vx 5,\vvc)
\]
into
\[
\QU_{m,l,p,q}(\Vff, \vx 7,\vx 4, \vx 5,\Vcc).
\]
The case $j_6 \vee j_3 \ge 30m$ is easy to discard. We now exclude the case $j_3 \vee j_6 \geq (1-\delta^2)m + k_6$. If $j_3 \geq (1-\delta^2)m + k_6$, then averaging arguments yield the estimate
\begin{align*}
	&2^{10k_+}2^{\delta k}2^{(1-20\delta)(j+k)}2^{-(1/2-18\delta)n} 2^{-p}\\
	&\times \norm{A^{(j)}_nQ_{jk}\QU_{m,l,p,q}(\Vff, \vx 7,\vx 4, \vx 5,\Vcc)}_{2}\\
	\lle{} & 2^{3\delta^2 m} 2^{(1-20\delta)m} 2^{-p} 2^m2^{2k_1} 2^{k_6} 2^{k_4} (\sup_{s \simeq s_1 \simeq 2^m}\norm{E_{\iota_6 s}\Vff(s)}_{\infty}\norm{\vx 7(s)}_{2}\\
	&\times 2^{k_{45}}\norm{E_{\iota_4 s}\vx 4(s)}_{\infty} \norm{\vx 5(s)}_{2} \norm{E_{\iota_3 s_1}\Vcc(s)}_{\infty}+\varepsilon_1^3 2^{-10m}),
\end{align*}
(see Section~\ref{techlemma} for a discussion of the averaging argument).
If $j_6\geq (1-\delta^2)m + k_6$, then we insert cutoff functions $\wan \fai^{-\infty,0}_{-n_{67}}(\abs{\xi_{67}}-2\gamma_0)$ to obtain
\[
\QU_{m,l,p,q,n_{67}}(\Vff, \vx 7,\vx 4, \vx 5,\Vcc).
\]
Using averaging arguments, we obtain
\begin{align*}
	&2^{10k_+}2^{\delta k}2^{(1-20\delta)(j+k)}2^{-(1/2-18\delta)n} 2^{-p}\\ 
	&\times\norm{A^{(j)}_nQ_{jk}\QU_{m,l,p,q,n_{67}}(\Vff, \vx 7,\vx 4, \vx 5,\Vcc)}_{2}\\
	\lle{} & 2^{3\delta^2 m} 2^{(1-20\delta)m} 2^{-p} 2^m 2^{2k_1}2^{k_6}2^{k_4}( \sup_{s \simeq 2^m} 2^{-n_{67}/2} \norm{\Vff}_2 \norm{E\vx 7}_{\infty} \norm{E\vx 4}_{\infty} \norm{\vx 5}_{2}\\
	&\times\norm{E_{\iota_3[2^{m-2},2^{m+2}]}\Vcc}_{\infty}+\varepsilon_1^3 2^{-10m}).
\end{align*}
If $n_{67} \gge n_6$, then the desired estimate follows immediately. If instead $n_{67} \ll n_6$, then $k_7 \gge -n_{67}$, and interpolating \cororef{decayestcom} with $\norm{E \vx 7}_\infty \lle \epsl_1 2^{(1-\delta)k_7}$ yields the desired bound. Moreover, it is easy to exclude the case $k_6 \le -m/8$. Now, integration by parts in the frequency variables concludes this term.

We next deal with $\TT'_{m,l,p}(\ppd s \va,\vb,\vvc)$. Neglecting the cubic terms and the remainders, we insert a cutoff function $\fai_{>-\delta^2 m}(\Phi)$ into $\ppd s \va$ and decompose it into $I_{\gti m_q,>-\delta^2 m}(\vx 4, \vx 5)$ and $I_{\gti m_q,\le -\delta^2 m}(\vx 4, \vx 5)$, with $-3m \le k_4, k_5 \le \delta^3 m$. For the $\le -\delta^2 m$-part, we write it as a quadrilinear operator
\[
2^{-p}\QQ_{\gti p,m,l,p,\le -\delta^2 m}(\vx 4,\vx 5,\vb,\vvc).
\]
We denote by $\xi_4$, $\xi_5$, $\xi_2$, and $\xi_3$ the frequency variables of $\vx 4$, $\vx 5$, $\vb$, and $\vvc$, respectively. The multiplier $\gti p$ is a product of $\xi_4$-, $\xi_5$-, $\xi_2$-, $\xi_3$-, $\xi_{45}$-, $\xi_{452}$-, $\xi_{453}$-, and $\xi$-factors in $\gti S^{100}$. Let $K=\set{k_\alpha \colon \gti p \text{ contains a }\xi_\alpha-\text{factor}}$, and set $\cl k=\max K$ and $\ul k=\min K$. Then $\norm{\gti p}_{S_{K}} \lle 2^{2\cl k}2^{k_4 \vee k_5}2^{15\cl k_+}$. In what follows, the multiplier $\gti p$ is omitted from the notation. The restriction $\abs{\Phi^{4,5}_{\iota_1\iota_4\iota_5}} \lle 2^{-\delta^2 m}$ implies $k_4, k_5 \ge -\delta^2 m$. We aim to exploit the absence of space resonances between $\xi_4, \xi_5$ and $\xi_2, \xi_3$. We decompose
\[
\QQ_{m,l,p,\le -\delta^2 m}(\vx 4,\vx 5,\vb,\vvc)
\]
into
\[
\QQ_{m,l,p,\le -\delta^2 m}(\Vxx 4,\Vxx 5,\Vbb,\Vcc),
\]
and we may assume that $j_4 \vee j_5 \vee j_2 \vee j_3 \le 30 m$. We next exclude the cases $j_4 \wedge j_5 \ge (1-5\delta^2)m$ and $j_2 \wedge j_3 \ge (1-5\delta^2)m$. If $j_4 \wedge j_5 \ge (1-5\delta^2)m$, then, by means of an averaging argument, we estimate 
\begin{align*}
	&2^{10k_+}2^{\delta k}2^{(1-20\delta)(j+k)}2^{-(1/2-18\delta)n} 2^{-p}\\  
	& \times \norm{A^{(j)}_nQ_{jk}\QQ_{m,l,p,\le -\delta^2 m}(\Vxx 4,\Vxx 5,\Vbb,\Vcc)}_{2}\\
\lle{} &2^{2\delta^2 m} 2^{(1-20\delta)m}2^{-p}2^m 2^{2\cl k} (\sup_{s \simeq 2^m}\norm{I_{\gti m_q, \le -\delta^2 m}(\Vxx 4,\Vxx 5)}_2\norm{E_{\iota_2[2^{m-2},2^{m+2}]}\Vbb}_\infty\\
& \times \norm{E_{\iota_3[2^{m-2},2^{m+2}]}\Vcc}_\infty+\epsl_1^3 2^{-10m}),
\end{align*}
where $I_{\le -\delta^2 m}(\Vxx 4,\Vxx 5)$ is handled by \lemmaref{sizeest}(i) and (ii).
If $j_2 \wedge j_3 \ge (1-5\delta^2)m$, then an averaging argument and an $L^2 \times L^\infty$ argument yield the desired estimate.
Now, integration by parts in the frequency variables concludes this part.

Finally, we address the $>-\delta^2 m$-part. We write it as a quadrilinear operator
\[
2^{-p}\QQ_{\gti p,m,l,p,>-\delta^2 m}(\vx 4,\vx 5,\vb,\vvc).
\]
We denote by $\xi_4$, $\xi_5$, $\xi_2$, and $\xi_3$ the frequency variables of $\vx 4$, $\vx 5$, $\vb$, and $\vvc$, respectively. The multiplier $\gti p$ is a product of $\xi_4$-, $\xi_5$-, $\xi_2$-, $\xi_3$-, $\xi_{45}$-, $\xi_{452}$-, $\xi_{453}$-, and $\xi$-factors in $\gti S^{100}$. Let $K=\set{k_\alpha \colon \gti p \text{ contains a }\xi_\alpha-\text{factor}}$, and set $\cl k=\max K$ and $\ul k=\min K$. Then $\norm{\gti p}_{S_{K}} \lle 2^{2\cl k}2^{k_4 \vee k_5}2^{15\cl k_+}$. For simplicity, we omit the multiplier $\gti p$ from the notation. We may assume that $k_4 \ge k_5$. Since $\Phi^{4,5,2,3}_{+\iota_4\iota_5\iota_2\iota_3}=\Phi^{45,2,3}_{+\iota_1\iota_2\iota_3}+\Phi^{4,5}_{\iota_1\iota_4\iota_5}$, we have $\abs{\Phi^{4,5,2,3}_{+\iota_4\iota_5\iota_2\iota_3}} \gge 2^{-\delta^2 m}$. After inserting localizations of the third resonance function $\Phi^{4,5,2,3}_{+\iota_4\iota_5\iota_2\iota_3}$ with parameter $q \gge -\delta^2 m$, integration by parts in time yields 
\begin{align*}
	&\QQ_{m,l,p,>-\delta^2 m,q}(\vx 4,\vx 5,\vb,\vvc)\\
	={}&\II'_{m,l,p,>-\delta^2 m,q}(\vx 4,\vx 5,\vb,\vvc)\\
	&+\QQ'_{m,l,p,>-\delta^2 m,q}(\ppd s\vx 4,\vx 5,\vb,\vvc)\\
	&+\QQ'_{m,l,p,>-\delta^2 m,q}(\vx 4,\ppd s\vx 5,\vb,\vvc)\\
	&+\QQ'_{m,l,p,>-\delta^2 m,q}(\vx 4, \vx 5,\ppd s\vb,\vvc)\\
	&+\QQ'_{m,l,p,>-\delta^2 m,q}(\vx 4,\vx 5,\vb, \ppd s \vvc).
\end{align*}
Note that \lemmaref{symbound3} and \lemmaref{symbound}(iii) imply that
\[
 \norm{\psi_q(\Phi^{4,5,2,3}_{+\iota_4\iota_5\iota_2\iota_3})}_{S_{K}}, \norm{\fai_{>-\delta^2 m}(\Phi^{4,5}_{\iota_1\iota_4\iota_5})}_{S_{k_4,k_5,k_{45}}} \lle 2^{10\delta^2 m}.
 \]
 Hence the terms
 \[
 \II'_{m,l,p,>-\delta^2 m,q}(\vx 4,\vx 5,\vb,\vvc)
 \text{ and }
 \QQ'_{m,l,p,>-\delta^2 m,q}(\vx 4,\ppd s\vx 5,\vb,\vvc)
 \]
 are straightforward to bound. The terms
 \[
 \QQ'_{m,l,p,>-\delta^2 m,q}(\vx 4, \vx 5,\ppd s\vb,\vvc),
 \QQ'_{m,l,p,>-\delta^2 m,q}(\vx 4,\vx 5,\vb, \ppd s \vvc)
 \]
 can be treated as in the treatment of
 \[
 \QQ'_{m,l,p,q}(\ppd s \va,\vx 4,\vx 5,\vvc)
 \]
 in \eqref{eq:largemodu1}.
It remains to estimate
\[
\QQ'_{m,l,p,>-\delta^2 m,q}(\ppd s\vx 4,\vx 5,\vb,\vvc).
\]
The argument is analogous to that used for
\[
\QQ'_{m,l,p,q}(\ppd s \va,\vx 4,\vx 5,\vvc)
\]
in \eqref{eq:largemodu1}. Neglecting the cubic terms and the remainders, we write it as a quintilinear operator
\[
2^{-q}\hti{QU}_{\gti q,m,l,p,>-\delta^2 m,q}(\vx 6,\vx 7,\vx 5, \vb,\vvc).
\]
We denote by $\xi_6$, $\xi_7$, $\xi_5$, $\xi_2$, and $\xi_3$ the frequency variables of $\vx 6$, $\vx 7$, $\vx 5$, $\vb$, and $\vvc$, respectively. The multiplier $\gti q$ is a product of $\xi_6$-, $\xi_7$-, $\xi_5$-, $\xi_2$-, $\xi_3$-, $\xi_{67}$-, $\xi_{675}$-, $\xi_{6752}$-, $\xi_{6753}$-, and $\xi$-factors in $\gti S^{100}$. Let $K=\set{k_\alpha \colon \gti q \text{ contains a }\xi_\alpha-\text{factor}}$, and set $\cl k=\max K$ and $\ul k=\min K$. Then $\norm{\gti q}_{S_{K}} \lle 2^{3\cl k}2^{k_6 \vee k_7}2^{20\cl k_+}$. In what follows, we suppress the multiplier $\gti q$ from the notation. We may assume that $-3m \le \ul k \le \cl k \le \delta^3 m$ and $k_6 \ge k_7$. We focus on the case $k_6 \gg k_7, k_5, k_2, k_3$ as the other cases can be handled easily. Set $\mao k=\max \set{k_7,k_5,k_2,k_3}$. 
We next exploit the gap between $\xi_6$ and $\xi_3$. We decompose
\[
\hti{QU}_{m,l,p,>-\delta^2 m,q}(\vx 6,\vx 7,\vx 5, \vb,\vvc)
\]
into
\[
\hti{QU}_{m,l,p,>-\delta^2 m,q}(\Vff,\vx 7,\vx 5, \vb,\Vcc),
\]
and it is easy to discard the case $j_3 \vee j_6 \ge 30m$. We now exclude the case $j_3 \vee j_6 \ge (1-\delta^2)m + k_6$. If $j_3 \ge (1-\delta^2)m + k_6$, then averaging arguments and an $L^2 \times L^{\infty}$ argument give the desired estimate.
If $j_6 \ge (1-\delta^2)m + k_6$, then it suffices to show that
\begin{align*}
	&2^{10k_+}2^{\delta k}2^{(1-20\delta)(j+k)}2^{-p} \\ 
	& \times \sup_{0 \le n \le j+1}2^{-(1/2-18\delta)n} \lVert A^{(j)}_nP_k\hti{QU}_{m,l,p,>-\delta^2 m,q}(\Vff,\vx 7,
	\vx 5, \vb,\Vcc) \rVert_{2}\\ 
	&\lle \epsl_1^32^{-10\delta^4m}.
\end{align*}
Using averaging arguments, we have
\begin{align*}
		&2^{10k_+}2^{\delta k}2^{(1-20\delta)(j+k)}2^{-p}2^{-(1/2-18\delta)n}\\ 
		& \times \norm{A^{(j)}_nP_k\hti{QU}_{m,l,p,>-\delta^2 m,q}(\Vff,\vx 7,\vx 5, \vb,\Vcc)}_{2}\\
	\lle{} & 2^{15\delta^2 m} 2^{(1-20\delta)m} 2^{-p}2^{-(1/2-18\delta)n}2^m2^{4\cl k}  (\sup_{s \simeq 2^m}\norm{\Vff}_{2} \norm{E\vx 7}_{\infty} \\
	&\times \norm{E\vx 5}_{\infty} \norm{E_{\iota_2[2^{m-2},2^{m+2}]}\vx 2}_{\infty}  \norm{E_{\iota_3[2^{m-2},2^{m+2}]}\Vcc}_{\infty} + \varepsilon_1^3 2^{-10m}).
\end{align*}
If $n \gge n_6$, then the conclusion follows immediately. If instead $n \ll n_6$, then we deduce that $-n \lle \mao k$ or $n=j+1$. The case $n=j+1$ is easy to handle. When $-n \lle \mao k$, suppose that $k_a=\mao k$, and then interpolating \cororef{decayestcom} with $\norm{E_{\iota
_a[2^{m-2},2^{m+2}]}\vx a}_\infty \lle \epsl_1 2^{(1-\delta)\mao k}$ or $\norm{E_{\iota
_a[2^{m-2},2^{m+2}]}\Vxx a}_\infty \lle \epsl_1 2^{(1-\delta)\mao k}$ yields the desired bound.
Furthermore, it is easy to exclude the case $k_6 \le -m/8$. Finally, under the assumption $j_3 \vee j_6 \leq (1-\delta^2)m + k_6$, integration by parts in the frequency variables concludes this term.

\subsection{Remainders}\label{sec:remainder}
In this subsection, we deal with the remainders $R$ and $\BB'_{m,l}(\va,$ $E^{-\iota_2}(R_0)^{\iota_2,p_2}_{k_2})$, where $R_0=T_{\gti m_c}(U^{\iota_1},U^{\iota_2},U^{\iota_3})$ or $R$. Recall that
\[
R=\gti m(D) [(D^{\alpha_1} \Lambda \rho) (D^{\alpha_2} \Lambda \rho) (D^{\alpha_3} \Lambda \rho) (D^{\alpha_4} \Lambda \rho) f(\Lambda \rho)],
\]
where $\gti m \in \gti S^0$ satisfies $\Omega \gti m=0$, $0 \le \abs{\alpha_i} \le 2$ for $1 \le i \le 4$, and $f \in C^\infty(\bR)$. The argument is analogous to \cite[Section~4.8]{arXiv:2510.18655}. We aim to show that
\begin{gather*}
	\Norm{\int_0^t \me^{-\mi s \Lambda(D)}\Omega^{[0,N_2]}R \dif s}_Z \lle \epsl_1^2,\\
	\norm{\BB'_{m,l}(\va,E^{-\iota_2}(R_0)^{\iota_2,p_2}_{k_2})}_Z \lle \epsl_1^2 2^{-2\delta^4 m},
\end{gather*}
where $D^4<m \le L$, $l>-100\delta m$, $-40m \le k_1 \wedge k_2 \le k_1 \vee k_2 \le 20\delta^3 m$, and $0 \le p_1, p_2 \le N_2$. It suffices to establish
\begin{gather}
	2^{11k_+}2^{(1-20\delta)j}\Norm{Q_{jk}\int q_m(s) \me^{-\mi s \Lambda(D)}\Omega^{[0,N_2]}R \dif s}_2 \lle \epsl_1^2 2^{-\delta^4 m}, \label{remd1}\\
	2^{10k_+}2^{(1-19\delta)k}2^{(1-20\delta)j}2^{-l}\Norm{Q_{jk} \int q_m(s) \me^{-\mi s \Lambda(D)} T_{\gti m_q,l}(U^{\iota_1,p_1}_{k_1},(R_0)^{\iota_2,p_2}_{k_2}) \dif s}_2 \lle \epsl_1^2 2^{-2\delta^4 m}, \label{remd2}
\end{gather}
for any $D^4<m \le L$.
For \eqref{remd1}, we can assume that $k \le \delta^3 (m+j)$, and by \lemmaref{simplez}(ii) and (iii), it is enough to prove that
\begin{equation}\label{remd3}
\norm{\spa{x}^{1-19\delta}\Omega^{[0,N_2]}R(s)}_{2} \lle \epsl_1^2 2^{-m-\delta^2 m}
\end{equation}
for any $s \simeq 2^m$. Similarly, \eqref{remd2} reduces to
\[
2^{(1-19\delta)k}\norm{\spa{x}^{1-20\delta}P_kT_{\gti m_q,l}(U^{\iota_1,p_1}_{k_1}(s),(R_0)^{\iota_2,p_2}_{k_2}(s))}_2 \lle \epsl_1^2 2^{-m-m/100}	
\]
for any $s \simeq 2^m$.
By \lemmaref{simplez}(iv), it suffices to prove that
\begin{equation}\label{remd4}
\norm{\spa{x}^{1-20\delta}(R_0)^{\iota_2,p_2}(s)}_2  \lle \epsl_1 2^{-m/2-m/50}.
\end{equation}
Indeed, it follows from the argument used in the localized multiplier theory that 
\[
\norm{\spa{y}^{1-20\delta} \FF^{-1}(\gti m_q \psi_k(\xi)\psi_{k_1}(\xi_1) \psi_{k_2}(\xi_2) \fai_l^{l_0}(\Phi))}_1 \lle 2^{\alpha (\ul k)_-}2^{1000\delta m};
\]
to absorb the low-frequency loss $2^{(1-20\delta) (\ul k)_-}$, we use Bernstein's inequality if $\ul k=k_1 \wedge k_2$ and the factor $2^{(1-19\delta)k}$ if $\ul k=k$. Since \eqref{remd3} and \eqref{remd4} are similar, we only handle \eqref{remd3}. Because $\spa x^{2(1-19\delta)}$ is an $A_2$ weight and \lemmaref{weightedinfty} applies, we have
\begin{align*}
&\norm{\spa{x}^{1-19\delta}\Omega^{[0,N_2]}R(s)}_{2}\\
\lle{} &\norm{\spa x^{1-19\delta}[(\Omega^{p_1}D^{\alpha_1} \Lambda \rho) (\Omega^{p_2}D^{\alpha_2} \Lambda \rho) (\Omega^{p_3}D^{\alpha_3} \Lambda \rho) (\Omega^{p_4}D^{\alpha_4} \Lambda \rho) \Omega^{p}[f(\Lambda \rho)]]}_2\\
\lle{} &\norm{\Omega^{p_1}D^{\alpha_1} \Lambda \rho}_2 \norm{\spa x^{1/2-19\delta}\Omega^{p_2}D^{\alpha_2} \Lambda \rho}_\infty \norm{\spa x^{1/2}\Omega^{p_3}D^{\alpha_3} \Lambda \rho}_\infty \norm{\Omega^{p_4}D^{\alpha_4} \Lambda \rho}_\infty \norm{\Omega^{p}[f(\Lambda \rho)]}_\infty\\
\lle{} &\epsl_1^2 2^{-m-\delta^2 m},
\end{align*}
where $p_i \in [0,N_2]$ for $1 \le i \le 4$ and $p \in [0,N_2]$.

\section*{Acknowledgement}
The first author would like to thank South China Normal University and Capital Normal University for their invitations and hospitality. Part of this work was carried out during these visits. This work is based on the Ph.D. thesis of the second author, completed under the supervision of Professor Zhouping Xin at the Institute of Mathematical Sciences, The Chinese University of Hong Kong. The research of the second author was supported in part through Professor Zhouping Xin by the Zheng Ge Ru Foundation, Hong Kong RGC Earmarked Research Grants CUHK-14301421, CUHK-14300917, CUHK-14302819, and CUHK-14300819, the Key Project of the National Natural Science Foundation of China (Grant No.~12131010), and the Guangdong Basic and Applied Basic Research Foundation (Grant No.~2020B1515310002).

\textbf{AI Statement:}
The authors used GPT-5 to assist in drafting Mathematica code to investigate the non-degeneracy condition involving $\det [(\nabla \Upsilon)^T,(\nabla \Phi)^T]$ for all $m(0)>0$, and to obtain preliminary feedback on their proposed use of almost orthogonality in the small modulation energy estimates. The authors supplied the references for Section~1.3 and used GPT-5.6 to prepare an initial draft of that section. They subsequently checked the draft for accuracy and revised it. GPT-6 suggested that the small modulation assumption in Lemma~\ref{upsilon} could be removed. GPT also assisted with manuscript error checking and language editing. The authors reviewed the AI feedback and are responsible for the final manuscript.

\textbf{Data Availability Statement:}
Data sharing is not applicable to this article as no datasets were generated or analyzed during the current study.

\textbf{Conflict of Interest:}
The authors declare that they have no conflicts of interest related to this work.

\end{document}